\documentclass{amsart}

\usepackage{amsmath}
\usepackage{amsthm,amssymb,color,comment}
\usepackage{mathtools}
\usepackage{bbm}
\usepackage{hyperref}
\usepackage[TS1,T1]{fontenc}
\usepackage{inputenc}
\usepackage{dsfont}
\usepackage{tikz}
\usepackage{enumitem}
\usepackage{soul}
\usepackage[sort]{cite}

\numberwithin{equation}{section}

\newtheorem{thm}{Theorem}[section]
\newtheorem{prop}[thm]{Proposition}
\newtheorem{lem}[thm]{Lemma}
\newtheorem{cor}[thm]{Corollary}

\newtheorem{Def}[thm]{Definition}

\theoremstyle{definition}
\newtheorem{Ass}[thm]{Assumption}
\newtheorem{rem}[thm]{Remark}

\DeclareMathOperator{\DIV}{div}

\newcommand{\R}{\mathbb{R}}
\newcommand{\N}{\mathbb{N}}

\newcommand{\supp}{\text{supp}}
\newcommand{\diff}{\mathop{}\!\mathrm{d}}

\newcommand{\erww}[1]{\mathbb{E}\left[{#1}\right]}

\newcommand{\wstar}{\overset{\ast}{\rightharpoonup}}

\makeatletter
\newcommand{\doublewidetilde}[1]{{%
  \mathpalette\double@widetilde{#1}%
}}
\newcommand{\double@widetilde}[2]{%
  \sbox\z@{$\m@th#1\widetilde{#2}$}%
  \ht\z@=.9\ht\z@
  \widetilde{\box\z@}%
}
\makeatother
\author[P. Gwiazda]{Piotr Gwiazda}
\address{Institute of Mathematics of Polish Academy of Sciences, Jana i J\k edrzeja \'Sniadeckich 8, 00-656 Warsaw, Poland}
\email{pgwiazda@mimuw.edu.pl}
\thanks{Piotr Gwiazda was supported by National Science Center, Poland through project no. 2021/43/B/ST1/02851.}

\author[J. Woźnicki]{Jakub Woźnicki}
\address{Faculty of Mathematics, Informatics and Mechanics, University of Warsaw, Stefana Banacha 2, 02-097 Warsaw, Poland; Institute of Mathematics of Polish Academy of Sciences, Jana i J\k edrzeja \'Sniadeckich 8, 00-656 Warsaw, Poland}
\email{jw.woznicki@student.uw.edu.pl}
\thanks{Jakub Woźnicki was supported by National Science Center, Poland through project no. 2021/43/O/ST1/03031}

\author[A. Wr\'oblewska-Kami\'nska]{Aneta Wr\'oblewska-Kami\'nska}
\address{Institute of Mathematics of Polish Academy of Sciences, Jana i J\k edrzeja \'Sniadeckich 8, 00-656 Warsaw, Poland}
\email{awrob@impan.pl}
\thanks{Aneta Wr\'oblewska-Kami\'nska was supported by National Science Center, Poland through project no. 2020/38/E/ST1/00469.}

\author[A. Zimmermann]{Aleksandra Zimmermann}
	\address{Institute of Mathematics, Clausthal University of Technology, 38678 Clausthal-Zellerfeld, Germany}
\email{aleksandra.zimmermann@tu-clausthal.de}

\begin{document}

\title[Stochastic parabolic equations]{Stochastic parabolic equations in Musielak--Orlicz spaces with discontinuous in time $N$-function}

\begin{abstract}
We consider a stochastic parabolic partial differential equation with Dirichlet boundary conditions, multiplicative stochastic noise, and a monotone parabolic operator $A$. The growth and coercivity of $A$ is controlled by a general $N$-function $M$, which depends on time, and spatial variable, but we do not assume any regularity with respect to the former. We show the existence of weak solutions to such system. As auxiliary result, we also provide the proof for the It\^{o}'s formula in Orlicz spaces. This general result applies to the ones studied in the literature, such as $p(t, x)$-Laplacian and double phase problems.
\end{abstract}


\maketitle

\section{Introduction}
We are interested in a stochastic partial differential equation
\begin{align}\label{eq:main_sys}
    \left\{\begin{array}{ll}
         \diff u - \DIV_x A(t, x, \nabla_x u)\diff t = h(u)\diff W &\text{ in }\Omega\times (0, T)\times D,  \\
         u = 0 &\text{ in }\Omega\times (0, T)\times\partial D,\\
         u(\omega, 0, x) = u_0 &\text{ in }\Omega\times D,
    \end{array}
    \right.
\end{align}
where $T > 0$ is a given time, $D\subset \R^d$ is a bounded, Lipschitz domain in arbitrary space dimension $d\in\mathbb{N}$, and $(\Omega, \mathcal{F}, P)$ is a probability space with a given, countably generated $\sigma$-field $\mathcal{F}$, a probability measure $P$ defined on it, and a normal filtration $(\mathcal{F}_t)_{t\geq 0}$. In our setting, the assumption of a countably generated $\sigma$-field will ensure the separability of some Lebesgue and Orlicz--Musielak spaces which will be used in the sequel.\\
When it comes to the Wiener process, we assume that $W=(W(t))_{t\geq 0}$ it is a $U$-valued Wiener process. More precisely let $(e_j)_{j\in\N}$ be an orthonormal basis of $L^2(D)$, and let $U$ be a Hilbert space defined as a closure of $L^2(D)$ with respect to the scalar product
$$
\langle u, v\rangle_U = \sum_{j \in \N}\frac{u_j\, v_j}{j^2},\text{ for }u = \sum_{j\in\N}u_j\,e_j,\quad v = \sum_{j \in \N}v_j\,e_j.
$$
Then $W=(W(t))_{t\geq 0}$ can be defined as 
$$
W(t) = \sum_{j\in\N}e_j\beta_j(t) = \sum_{j\in\N}\frac{1}{j}j e_j\beta_j(t),
$$
for a family of independent, real-valued Brownian motions $(\beta_j)_{j\in\mathbb{N}}$, $\beta_j=(\beta_j(t))_{t\geq 0}$ for all $j\in\mathbb{N}$, adapted to $(\mathcal{F}_t)_{t\geq 0}$ (see \cite[Section 2.5]{Liu2015Stochastic} for more information), and can be interpreted as $Q$-Wiener process with covariance Matrix $Q = \operatorname{diag}(1/j^2)$ and values in $U$. Since $Q^{1/2}(U)=L^2(D)$, the stochastic integral in the sense of Itô with respect to $(W(t))_{t\geq 0}$ may be defined for all integrable and predictable\footnote{Recall the predictable $\sigma$-field $\mathcal{P}_{T}:=\sigma(\{ F_s\times (s,t] \ | \ F_s\in \mathcal{F}_s, \ 0\leq s < t \leq T\} \cup \{ F_0 \times\{0\} \ | \ F_0\in \mathcal{F}_0 \})$ (see \cite[p. 33]{Liu2015Stochastic}). A mapping defined on $\Omega\times (0,T)$ with values in a separable Banach space $E$ is predictable if it is $\mathcal{P}_{T}$-measurable.} $\Phi: \Omega \times(0, T) \rightarrow H S\left(L^2(D)\right)$, where $HS(L^2(D))$ denotes the separable space of Hilbert-Schmidt operators from $L^2(D)$ to $L^2(D)$. As for the function $h$, the assumptions will be given in Assumption \ref{ass:func_h_multiplicative} in Section  \ref{sec:preliminaries}.
When it comes to the parabolic operator $A$, we assume that it is strictly monotone and its growth and coercivity conditions are controlled by the so-called $N$-function $M:(0, T)\times\R\times\R^d \rightarrow \R$, i.e. for a.e. $(t, x)\in (0, T)\times D$, and all $\xi\in \R^d$ we have
\begin{align}\label{4}
M(t, x, \xi) + M^*(t, x, A(t, x, \xi)) \leq cA(t, x, \xi)\cdot\xi + g(t, x),
\end{align}
where $M^*$ denotes a convex conjugate of $M$, $c > 0$, and $g\in L^\infty((0, T)\times D)$. In the most classical setting, one could consider $M(t, x, \xi) = |\xi|^p$. Then,
$$
L^p(0, T; W^{1,p}_0(D)) \ni u \mapsto \DIV A(t, x, \nabla_x u)\in (L^{p'}(0, T; L^{p'}(D)))^*
$$
is given, and the standard approaches such as Galerkin approximation yield desired results, see \cite{brezis1979strongly, landes1981existence} for a deterministic case, and \cite{Krylov2007stochastic} for a stochastic case, as well as references therein. However, when $M$ does not have polynomial growth, and depends on $(t, x)$, the considerations become much more complicated. Then, one has to look for a solution $u$, such that $\nabla_x u$ belongs to the Musielak--Orlicz space $L_M(\Omega\times (0, T)\times D)$ (see Section \ref{sec:musielaki}). Two main difficulties appear.
\begin{enumerate}[label = (P\arabic*)]
    \item In general, the Musielak--Orlicz spaces do not factorize, that is
    $$
    L_M((0, T)\times D) \ncong L_M(0, T; L_M(D)).
    $$
    \item In general, the duality between $A$ and $\nabla_x u$ "is not correct", which means that due to \eqref{4} one can expect $A(t, x, \nabla_x u)\in L_{M^*}(\Omega\times (0, T)\times D)$ and $\nabla_x u\in L_M(\Omega\times (0, T)\times D)$, but $(L_M(\Omega\times(0, T)\times D))^* \ncong L_{M^*}(\Omega\times (0, T)\times D)$.\label{problem_2}
\end{enumerate}
First attempts in a general direction were made for $M$ which is isotropic, and  independent of $(t, x)$, that is $M(t, x, \xi) = N(|\xi|)$. Under the additional coercitivity assumption $|\xi|^2 << N(|\xi|)$ and the so-called $\Delta_2$-condition for a convex conjugate of $N$, i.e.:
\begin{align}\label{delta_2}
N^*(2 |\xi| ) \leq CN^*(|\xi|),
\end{align}
for some constant $C>0$, the satisfactory results in a deterministic case were obtained in \cite{donaldson1974inhomogeneous, robert1974inequations}. In \cite{elmahi2002strongly}, one can find a similar result, but assuming $N(|\xi|) << |\xi|^{d/(d-1)}$ and $N(C|\xi_1||\xi_2|) \leq N(|\xi_1|)N(|\xi_2|)$. The main advantage of assuming \eqref{delta_2} is the fact, that it gets rid of \ref{problem_2}.

A different classical approach is to try to obtain some sort of energy estimates for \eqref{eq:main_sys}. Thus, the question arises, when is $C^\infty_c$ dense in the Musielak--Orlicz space $L_M$ (at least in the sense of modular convergence, see Section \ref{sec:musielaki}). It is known that in the variable exponent Lebesgue spaces ($M(t, x, \xi) = |\xi|^{p(t, x)}$), for the density to hold, $p$ needs to satisfy some continuity constraints in $(t, x)$. Such a way to tackle this problem was first posed in \cite{elmahi2005parabolic}, but again only for $M(t, x, \xi) = N(|\xi|)$. Later, it was extended to more general $N$-functions $M$, but still keeping some continuity assumptions, for a deterministic case in \cite{chlebicka2018well, chlebicka2019parabolic, gwiazda2011parabolic, swierczewska2014nonlinear,swierczewska2014anisotropic}, and for a stochastic case in \cite{Bauzet2013onaptx, vallet2016thestochastic}.

All of those results have a disadvantage, in the form of assumed regularity in time of $M$. Indeed, the somehow standard example to illustrate it is a system
\begin{align*}
    \partial_t u = \left\{\begin{array}{ll}
         \DIV_x \nabla_x u &\text{in }(0, 1]\times D, \\
         \DIV_x(|\nabla_x u|^2\nabla_x u) &\text{in } (1, 2]\times D,
    \end{array}
    \right.
\end{align*}
which can be solved piecewisely, and therefore is well-posed. This simple motivation gave a recent surge to study the problem \eqref{eq:main_sys}, with a general $N$-function $M$, which is discontinuous in time. The deterministic case have been solved in \cite{bulicek2021parabolic} with a fully general $N$-function $M$, and a deterministic forcing in $L^\infty((0, T)\times D)$. This result was further expanded in \cite{bulicek2025parabolicnonstandardgrowth} to forcing in $L^1((0, T)\times D)$, but in the setting $M(t, x, \xi) = |\xi|^{p(t, x)}$, for $p$ discontinuous in time. At last, there is as well a result \cite{bulicek2023nonnewtonian} of well-posedness for a $p(t,x)$-Navier--Stokes type equation with Dirichlet boundary conditions.

Our goal in the presented work is to combine the results in \cite{bulicek2021parabolic} and \cite{vallet2016thestochastic}, to obtain well-posedness of \eqref{eq:main_sys} for a general $N$-function $M$. The main difference when comparing \eqref{eq:main_sys} with stochastic forcing, to the case in \cite{bulicek2021parabolic} with deterministic force (in terms of what $N$- functions $M$ we can consider), is the regularity of the force itself. Indeed, to obtain reasonable a priori bounds, one would like to have $\int_0^\cdot h\diff W(s) \in E_m(\Omega\times(0, T)\times D)$, where $m$ is a Young function dominating $M$ from above (see Section \ref{sec:musielaki}). Even for a "very nice" constant $h$ it does not seem possible to have an arbitrary growth of $m$, as then the question reduces to the integrability of a Gaussian. As its density is $\sim e^{-x^2}$, thus asymptotically $m$ should have maximally growth of $\sim e^{x^2}$. In fact, one can show that if $m \sim e^{x^2}$, then $\int_0^\cdot h\diff W(s) \in L_m(\Omega\times(0, T)\times D)$, but as our technique requires the space $E_m$, we needed to slightly reduce the growth to the Assumption \ref{ass:growth_Young}. Moreover, since we will work in the proof with a fixed $\omega\in \Omega$ and extract certain subsequences, we need the uniqueness of the limiting equation (so that the subsequence does not depend on the fixed $\omega$), thus our considerations reduce only to isotropic $N$-functions $M(t, x, \xi) = M(t, x, |\xi|)$.

The structure of the paper is as follows. In Section \ref{sec:musielaki} we recall basic definitions and properties of Musielak--Orlicz spaces. Section \ref{sec:preliminaries} serves as an introduction to notation, assumptions and it contains the statement of the main result. We conduct the proof in Section \ref{sec:additive} and \ref{sec:multiplicative}, first for the additive case (that is when the stochastic forcing does not depend on the solution), and then extend it to the multiplicative case. In the Appendix, one can find auxiliary results, in particular the proof of the It\^{o}'s formula in Orlicz spaces.

\section{Musielak--Orlicz spaces}\label{sec:musielaki}

In this section we recall basic theory of Musielak--Orlicz spaces. For more precise information we refer to an extensive book \cite{chlebicka2021book} and an article on the application to the deterministic PDEs \cite{bulicek2021parabolic}.

In what follows, we let $D\subset \R^d$ be a bounded domain, and fix arbitrary time horizon $T>0$. With this, we denote $Q_T := (0, T)\times D$.

\subsection{Basic definitions and properties}

\begin{Def}[Young function]
    We say that $m:[0, +\infty)\rightarrow [0, +\infty)$ is a Young function if it satisfies
    \begin{enumerate}[label = (Y\arabic*)]
        \item $m(s) = 0 \Longleftrightarrow s = 0$,
        \item $m$ is convex,
        \item $m$ is superlinear, i.e. $\liminf_{s\to 0^+}\frac{m(s)}{s} = 0$, $\limsup_{s\to+\infty}\frac{m(s)}{s} = +\infty$.
    \end{enumerate}
\end{Def}
\begin{rem}
		We recall that, since $m$ is convex and nonnegative, from $(Y1)$ it follows that $m$ is strictly increasing.
	\end{rem}
\begin{Def}[$N$-function]
    We say that $M:(0, T)\times D\times \R^d \rightarrow \R$ is an $N$-function if it satisfies
    \begin{enumerate}[label = (N\arabic*)]
        \item $M(t, x, \xi) = M(t, x, -\xi)$ for a.e. $(t, x)\in (0, T)\times D$ and every $\xi\in\R^d$,
        \item $M(t, x, \xi)$ is a Carath\'{e}odory function, that is for a.e. $(t, x)\in (0, T)\times D$ the mapping $\R^d\ni \xi\mapsto M(t, x, \xi)\in \R$ is a continuous function, and for every $\xi\in \R^d$ the mapping $(0, T)\times D\ni (t, x)\mapsto M(t, x, \xi)\in \R$ is measurable.
        \item for a.e. $(t, x)\in (0, T)\times D$ the map $\xi \mapsto M(t, x, \xi)$ is convex,
        \item there exist two Young functions $\overline{m}$, $m$, such that for a.e. $(t, x)\in (0, T)\times D$ and every $\xi\in \R^d$
        $$
        \overline{m}(|\xi|) \leq M(t, x, \xi)\leq m(|\xi|).
        $$ \label{def:N_func_Young_growth}
    \end{enumerate}
\end{Def}

\begin{Def}[Isotropic $N$-function]
    We say that the $N$-function $M$ is isotropic if
    $$
    M(t, x, \xi) = M(t, x, |\xi|).
    $$
\end{Def}

\begin{Def}[Convex conjugate]
    Let $m$ be a Young function. Then we define its convex conjugate $m^*$ as
    $$
        m^*(s) = \sup_{x\in [0, +\infty)}(s\,x - m(x)).
    $$
    For the $N$-function $M$, we define its convex conjugate $M^*$ by
    $$
        M^*(t, x, \eta) = \sup_{\xi\in \R^d}(\xi\cdot\eta - M(t,x,\xi)).
    $$
\end{Def}

\begin{lem}\textup{(Basic properties of $N$- and Young functions)}. Let $m$ be a Young function and $M$ be an $N$-function. Then
    \begin{itemize}
        \item Function $\frac{m(t)}{t}$ is nondecreasing.
        \item $m^*$ is a Young function.
        \item $M^*$ is an $N$-functiom.
        \item If $f_n : Q_T\rightarrow \R^d$ is a sequence of functions such that $\int_{Q_T}M(t, x, f_n)\diff x\diff t \leq C$ for some $C > 1$, then $\|f_n\|_{L_M} \leq C$.
    \end{itemize}
\end{lem}

\begin{Def}[Musielak--Orlicz space $L_M(Q_T)$]
    Let $M$ be an $N$-function. Then, the Musielak--Orlicz space $L_M(Q_T)$ is defined as
    $$
    L_M(Q_T) := \left\{f:Q_T\rightarrow\R^d:\, \text{there exists }\lambda > 0\text{ such that }\int_{Q_T}M\left(t, x, \frac{f}{\lambda}\right)\diff x\diff t < +\infty \right\}.
    $$
\end{Def}
This is a Banach space when equipped with the norm
$$
\|f\|_{L_M(Q_T)} := \inf\left\{\lambda > 0:\, \int_{Q_T}M\left(t, x, \frac{f}{\lambda}\right) \diff x\diff t\leq 1\right\}.
$$
If $m$ is a Young function, then similarly we can define a Musielak--Orlicz space $L_m(Q_T)$ (or simply Orlicz in this case).

One can prove the so-called Fenchel--Young inequality and a form of H\"{o}lder's inequality in Musielak--Orlicz spaces.

\begin{lem}
    Let $M$ be an $N$-function and $M^*$ be its convex conjugate. Then, for every $f\in L_M(Q_T)$, $g\in L_{M^*}(Q_T)$ the following inequalities are true
    \begin{align*}
        &\int_{Q_T}f\,g\diff x\diff t \leq \int_{Q_T}M(t, x, f)\diff x\diff t + \int_{Q_T}M^*(t, x, g)\diff x\diff t ,\\
        &\int_{Q_T}f\,g\diff x\diff t\leq 2\|f\|_{L_M(Q_T)}\|g\|_{L_{M^*}(Q_T)}.
    \end{align*}
\end{lem}

\subsection{Notions of convergence in Musielak--Orlicz spaces}

Obviously one can define convergence in norm, in Musielak--Orlicz spaces. The other useful notion is the so-called modular convergence.

\begin{Def}[Modular convergence in $L_M(Q_T)$]
    Given a sequence of functions $\{f_n\}_{n\in\N}\subset L_M(Q_T)$, we say that it converges to $f$ modularly, denoted by $\xrightarrow{M}$, if there exists $\lambda > 0$, such that
    $$
    \int_{Q_T}M\left(t, x, \frac{f_n - f}{\lambda}\right)\diff x\diff t\rightarrow 0,\text{ as }n\to +\infty.
    $$
\end{Def}

By convexity it follows, that if $\{f_n\}_{n\in\N}\subset L_M(Q_T)$ and $f_n \xrightarrow{M} f$ modularly in $L_M(Q_T)$, then $f\in L_M(Q_T)$. Due to the Vitali convergence theorem, we obtain the following characterization of the modular convergence.

\begin{prop}\textup{(\cite[Theorem 3.4.4]{chlebicka2021book})}\label{prop:vitali_in_musielaki}
    Let $\{f_n\}_{n\in\N}\subset L_M(Q_T)$ and $f\in L_{M}(Q_T)$. Then, $f_n \xrightarrow{M} f$ modularly in $L_M(Q_T)$ if, and only if
    \begin{enumerate}[label = (V\arabic*)]
        \item $f_n\rightarrow f$ in measure,
        \item there exists $\lambda > 0$, such that $\left\{M\left(t, x, \frac{f_n}{\lambda}\right)\right\}_{n\in\N}$ is uniformly equiintegrable.
    \end{enumerate}
\end{prop}

Due to the above, we obtain the following corollary, which gives us the ability to converge in products, even without the information on the strong convergence of the sequences.

\begin{prop}\textup{(\cite[Corollary 1.11]{bulicek2021parabolic})}\label{prop:conv_of_product_in_musielaki}
    Suppose $\{f_n\}_{n\in\N} \subset L_M(Q_T)$ and $\{g_n\}_{n\in\N}\subset L_{M^*}(Q_T)$ be such that $f_n\xrightarrow{M} f$ modularly in $L_M(Q_T)$ and $g_n\xrightarrow{M^*} g$ modularly in $L_{M^*}(Q_T)$. Then, $f_n\, g_n\rightarrow f\, g$ strongly in $L^1(Q_T)$.
\end{prop}

Moving forward, we want to define weak* convergence in Musielak--Orlicz spaces, the basic intuition would be that $(L_M(Q_T))^* = L_{M^*}(Q_T)$, but unfortunately, this is not always true. Therefore, we need to define a particular subspace.

\begin{Def}
    Let $M$ be an $N$-function. We define a space $E_M(Q_T)$ as
    $$
    E_M(Q_T) := \left\{f:Q_T\rightarrow\R^d:\, \text{for every }\lambda > 0\text{ we have }\int_{Q_T}M\left(t, x, \frac{f}{\lambda}\right)\diff x\diff t < +\infty \right\}.
    $$
\end{Def}
Similarly if $m$ is a Young function, then we can define a function space $E_m(Q_T)$.

With the above definition in mind, we can state the following result on the weak* convergence.

\begin{lem}[\cite{chlebicka2021book}]
Let $M$ be an $N$-function. Then, $E_M(Q_T)$ is separable and $(E_M(Q_T))^* = L_{M^*}(Q_T)$. In particular, any bounded sequence in $L_{M^*}(Q_T)$ has a weak* convergent subsequence.
\end{lem}

\section{Notation, assumptions, and the main result}\label{sec:preliminaries}
We recall that $(\Omega, \mathcal{F}, P)$ is a probability space with a given $\sigma$-field $\mathcal{F}$, a probability measure $P$ defined on it, and a normal filtration $(\mathcal{F}_t)_{t\geq 0}$, $(0, T)$ is a time interval, and $D\subset\R^d$ is a given Lipschitz domain of arbitrary space dimension $d\in\mathbb{N}$. Throughout the article, the triple $(\omega, t, x)$ constitutes of an elementary event, a time, and a space variable. Recall that we have already defined $Q_t := (0, t)\times D$. For any two vectors $a,b\in\R^d$, we denote by $a\cdot b$ their standard vector product, i.e. $a\cdot b = \sum_{i = 1}^d a_i\,b_i$. We employ the standard notation for Lebesgue ($L^p$) and Sobolev ($W^{1,p}$ or $W^{1,p}_0$) spaces, and $\mathcal{D}'(D)$ denotes the space of distributions on $D$. Note that the Musielak--Orlicz spaces $L_M$ are defined in the preceding section. We frequently do not distinguish between vector- and scalar- valued functions. The expected value of a random variable will be denoted by $\mathbb{E}$.

The following basic definitions for a stochastic integral one can find in \cite{Liu2015Stochastic}. For a separable Banach space $E$ we define the Banach space of $E-$valued square integrable martingales $\mathcal{M}_T^2(E)$ with a norm
$$
\|\mathbb{M}\|_{\mathcal{M}_T^2} = \sup_{t\in [0, T]}(\mathbb{E}\|\mathbb{M}(t)\|^2_E)^{1/2} =  (\mathbb{E}\|\mathbb{M}(T)\|^2_E)^{1/2}.
$$
For any two separable Hilbert spaces $V$, $H$, let $\mathcal{L}(V; H)$ denote the space of bounded, linear functionals from $V$, taking values in $H$. Then, we define the space of Hilbert--Schimdt operators as
$$
HS(V; H) := \left\{T\in \mathcal{L}(V; H):\quad \sum_{j\in\N}\|T\mathfrak{e}_j\|_{H}^2 < +\infty, \ (\mathfrak{e}_j)_{j\in\N}\text{ is an orthonormal basis of }V\right\}. 
$$
Note that $HS(V; H)$ is a separable Hilbert space with the norm
$$
\|T\|_{HS}^2 = \sum_{j\in\N}\|T \mathfrak{e}_j\|_{H}^2.
$$
In the following, we will use the notation $HS(H):=HS(H,H)$.
Furthermore, we define by It\^{o}'s isometry \cite[Proposition 2.3.5]{Liu2015Stochastic} the norm
$$
\|\Phi\|_T^2 := \mathbb{E}\int_0^T \|\Phi\|_{HS}^2\diff t = \left\|\int_0^\cdot \Phi\diff W(t)\right\|_{\mathcal{M}_T^2},
$$
where $\Phi:[0,T] \times \Omega \rightarrow HS(L^2(D); H)$, and as such, for any separable Hilbert space $H$ we can define a space of It\^{o} integrable mappings with respect to the $Q$-Wiener process $(W(t))_{t\geq 0}$ 
$$
\mathcal{N}^2_W(0, T; H) = \left\{\Phi:[0, T]\times\Omega\rightarrow HS(L^2(D); H):\, \Phi\text{ is predictable and }\|\Phi\|_{T} < +\infty\right\}. 
$$
With this, for any $\Phi\in \mathcal{N}^2_W(0, T; H)$, and measurable $\xi : \Omega\times (0, T) \rightarrow H$ such that
$$
\int_0^T \|\Phi(t,\omega)^*\xi(t,\omega)\|^2_{L^2(D)}\diff t,\text{ almost surely,}
$$
we define
$$
\int_0^T\langle \xi,\,\Phi(t)\diff W(t)\rangle_H := \int_0^T\tilde{\Phi}_\xi(t)\diff W(t),
$$
for
$$
\tilde{\Phi}_\xi(t)(f) = \langle\xi,\,\Phi(t)f\rangle_H,\quad f\in L^2(D).
$$
Moving further, we define truncation at the level $k\in\N$ by
\begin{align}\label{def:trunc}
    T_k(z) := \left\{\begin{array}{ll}
         z\quad\text{ when }|z|\leq k ,  \\
         \mathrm{sgn}(z)k\quad\text{ otherwise.} 
    \end{array}
    \right.
\end{align}
And its primitive function
\begin{align}\label{def:primitive_trunc}
    G_k(z) = \int_0^z T_k(s)\diff s.
\end{align}
Moreover, let 
\begin{align}\label{def:trunc_smooth}
T_{k,\delta}(z):= \left\{\begin{array}{ll}
     z \quad\text{ when }|z|\leq k, \\
     \mathrm{sign}(z)(k + \delta/2)\quad\text{ for }|z|\geq k + \delta,
     & 
\end{array}
\right.
\end{align}
whereas $T_{k,\delta}$ is defined on $(k, k+\delta)$ in such a way that $T_{k,\delta} \in C^2(\R)$, $0\leq T'_{k,\delta}\leq 1$, $T_{k,\delta}$ is concave on $\R_+$ and convex on $\R_-$ with its second derivative satisfying $|T''_{k,\delta}|\leq C\delta^{-1}$. With this, we define $G_{k,\delta}$ as a primitive function of $T_{k,\delta}$ similarly as in \eqref{def:primitive_trunc}.

\subsection{Assumptions on $A$ and the $N$-function $M$}

We start with the assumption on the $N$-function $M$, and the operator $A$. We remark here, that Assumption \ref{ass:N_func_M} is quite hidden in our work, and not used explicitly, but very much necessary. It gives good approximation behaviour in $L_M(Q_T)$ space, which can be seen in Proposition \ref{prop:double_mollifier_conv}. When compared to earlier works (cf. \cite{chlebicka2019parabolic, chlebicka2021book}) one can notice that we have a lower exponent in the third variable of the function $\Theta$, that is $\Theta(t, \delta, C\delta^{-1})$, instead of $\Theta(t, \delta, C\delta^{-d})$. This is because of the observation made in \cite{bulicek2021parabolic}, which is the fact that one only has to have a good approximation of the functions of the form
$$
\nabla_x (T_k(u) + \phi),\quad \nabla_x u\in L_M(Q_T),\, \phi\in C^\infty_c(Q_T),
$$
see Proposition \ref{prop:double_mollifier_conv}. At last, we would like to also mention, that Assumption \ref{ass:N_func_M} in the less general variable exponent spaces $L^{p(t, x)}(Q_T)$ implies the log-H\"{o}lder continuity (in space) of the exponent $p$, which can be seen in the works such as \cite{bulicek2023nonnewtonian, bulicek2025parabolicnonstandardgrowth}. In our general case, the continuity in space of the function $M$ is not necessarily implied. Indeed, it is quite easy to see that Assumption \ref{ass:N_func_M} implies the continuity (in space) of M iff
$$
\lim_{\varepsilon\to 0^+}\Theta(t, \varepsilon, |\xi|) = 1.
$$

\begin{Ass}[Assumption on the $N$-function $M$]\label{ass:N_func_M}
    We assume that $M$ is isotropic. Moreover, there exists a function $\Theta: (0, T)\times [0, 1]\times [0, +\infty)\rightarrow [0, +\infty)$, which is nondecreasing with respect to the second and third variable, such that
    $$
        \limsup_{\delta\to 0^+}\Theta(t, \delta, C\delta^{-1}) \text{ is bounded uniformly in time }t\in(0, T),
    $$
    and
    $$
        \frac{M(t, x, r)}{M(t, y, r)}\leq \Theta(t, |x - y|, r).
    $$
\end{Ass}
\begin{Ass}[Assumption on growth of the Young function]\label{ass:growth_Young}
    There exists a Young function $m$, constant $B > 0$ and $\beta\in (0, 1)$ such that
    $$
    M(t, x, \xi)\leq m(|\xi|),
    $$
    and
    $$
    \lim_{x\to +\infty}m(x)e^{-B x^{1 + \beta}} < +\infty.
    $$
\end{Ass}

As mentioned in the introduction, the following constraint on the growth of $M$ is connected to good behaviour of the stochastic noise $\int_0^t h\diff W(s)$, see Lemma \ref{lem:stoch_integral_in_musielak_space}.

\begin{Ass}[Assumptions on $A$]\label{ass:stress_A}
    We assume that $A:Q_T\times \R^d\rightarrow \R^d$ satisfies
    \begin{enumerate}[label = (A\arabic*)]
        \item $A(t, x, \xi)$ is a Carath\'{e}odory function, that is for a.e. $(t, x)\in (0, T)\times D$ the mapping $\R^d\ni \xi\mapsto A(t, x, \xi)\in \R$ is a continuous function, and for every $\xi\in \R^d$ the mapping $(0, T)\times D\ni (t, x)\mapsto A(t, x, \xi)\in \R$ is measurable.
        \item (coercivity and growth bound) there exists a constant $c > 0$ and a function $g\in L^\infty(Q_T)$ such that for all $\xi\in \R^d$ and almost every $(t, x)\in Q_T$
        $$
            M(t, x, \xi) + M^*(t, x, A(t, x, \xi)) \leq cA(t, x, \xi)\cdot\xi + g(t, x).
        $$\label{ass:on_A_2}
        \item (strict monotonicity) for every $\xi_1\neq\xi_2\in\R^d$ and a.e. $(t, x)\in Q_T$ we have
        $$
            (A(t, x, \xi_1) - A(t, x, \xi_2))\cdot(\xi_1 - \xi_2) > 0.
        $$\label{ass:on_A_3}
        \item for a.e. $(t, x)\in Q_T$ $A(t, x, 0) = 0$.
    \end{enumerate}
\end{Ass}

\begin{prop}\textup{(\cite[Lemma 1.22]{bulicek2021parabolic})}\label{prop:stress_bound_for_bounded}
    Let $A$ satisfy Assumption \ref{ass:stress_A}. Then, for every $K>0$, there exists a constant $C(K) > 0$ depending on $K$, such that $|A(t, x, \xi)|\leq C(K)$ for a.e. $(t, x)\in Q_T$ and all $\xi\in \R^d$ satisfying $|\xi|\leq K$.
\end{prop}

At last, we come to the assumption on the stochastic noise $h$, which are important whenever we consider a multiplicative case.
\begin{Ass}[Assumptions on $h$]\label{ass:func_h_multiplicative} We assume that $h: \Omega\times (0, T)\times L^2(D)\rightarrow HS(L^2(D))$
    $$
        h(\omega, t, u)(e_j) = \{x\mapsto h_j(\omega, t, x, u(x))\},
    $$
    and
    \begin{enumerate}[label = (H\arabic*)]
        \item For any $j\in \N$, $h_j: \Omega\times (0, T)\times D\times \R\rightarrow \R$ is a Carath\'{e}odory function, such that for all $\lambda\in \R$, $(\omega,t)\mapsto h_j(\omega,t,\cdot,\lambda)$ is predictable with values in $L^2(D)$,  
        and $\lambda\mapsto h_j(\omega, t,x, \lambda)$ is continuous $dP\otimes dt\otimes dx$-almost everywhere. 
        \item There exist $C_1, C_2 > 0$ and $C_3\in L^1(D)$ such that for a.e. $(\omega, t,x)$ and all $\lambda \in \R$ we have
        $$
        \sum_{j\in\N}|h_j(\cdot,\lambda)|^2 \leq C_1|\lambda|^2 + C_3(x),
        $$
        and
        \begin{equation*}
        \sum_{j\in\N}|h_j(\cdot, \lambda_1) - h_j(\cdot, \lambda_2)|^2\leq C_2|\lambda_1 - \lambda_2|^2 \quad \mbox{ for all }\lambda_1,\ \lambda_2  \in \mathbb{R}.
        \end{equation*} \label{ass:on_h_2}
    \end{enumerate}
\end{Ass}

\subsection{Main result}

We start our discussion of the main result with a definition of the solution to the problem \eqref{eq:main_sys}.

\begin{Def}\label{def:solution_to_main_system}
    We will say that $u\in \mathcal{C}([0, T]; L^2(\Omega\times D))\cap L^1(\Omega\times(0, T); W^{1,1}_0(D))$ with $\nabla_x u\in L_M(\Omega\times Q_T)$ is a solution to the problem \eqref{eq:main_sys} iff for all $t\in [0, T]$, a.s. in $\Omega$,
    \begin{align}\label{eq:weak_form_final}
        u(t) - u_0 - \int_0^t \DIV_x A(\tau, x, \nabla_x u)\diff\tau = \int_0^t h(u)\diff W(\tau)\quad\text{ in }\mathcal{D}'(D).
    \end{align}
\end{Def}

Finally we may state our main result.

\begin{thm}\label{mainthm}
    Suppose $u_0\in L^2(\Omega;L^2(D))$ is $\mathcal{F}_0$-measurable. Let $M$ be an $N$-function satisfying the Assumption \ref{ass:N_func_M}, with its upper bound Young function $m$ satisfying the Assumption \ref{ass:growth_Young}. Moreover, let $A$ be an operator respecting the Assumption \ref{ass:stress_A}, and the stochastic noise $h$ satisfies Assumption \ref{ass:func_h_multiplicative}. Then, there exists a solution $u$, in the sense of the Definition \ref{def:solution_to_main_system} to the system \eqref{eq:main_sys}, satisfying the energy equality
    \begin{equation}\label{eq:final_energy_equality}
    \begin{split}
        &\frac{1}{2}\mathbb{E}\|u(t)\|_{L^2(D)}^2 - \frac{1}{2}\mathbb{E}\|u_0\|_{L^2(D)}^2\\
        &\qquad = -\mathbb{E}\int_0^t\int_D A(s, x, \nabla_x u)\cdot \nabla_x u\diff x\diff s + \mathbb{E}\int_0^t\|h\|_{HS(L^2(D))}^2\diff s.
    \end{split}
    \end{equation}
\end{thm}

\begin{rem}
    By Lemma \ref{lem:trunc_energy_eq} the solution given by Theorem \ref{mainthm} satisfies a truncated version of the energy equality as well, which holds without the expected value applied to both sides. For $G_{k}$, $T_{k}$ defined in \eqref{def:trunc}, \eqref{def:primitive_trunc}, the following energy equality holds almost surely
    \begin{multline*}
        \int_{D}G_{k}(u(t))\diff x - \int_{D}G_{k}(u_0)\diff x = -\int_0^t\int_D A(s, x, \nabla_x u)\cdot \nabla_x T_{k}(u)\diff x\diff s\\
       + \int_0^t\langle T_{k}(u), h\diff W(s)\rangle_2 + \int_0^t\sum_{j = 1}^\infty \int_D T'_{k}(u)\,|h(u)(e_j)|^2 \diff x \diff s.
    \end{multline*}
\end{rem}
\section{Additive case}\label{sec:additive}

We first prove Theorem \ref{mainthm} in the additive case, that is for $h$ which is independent of $u$, i.e., $h: \Omega\times (0, T)\times \mathbb{R}\rightarrow HS(L^2(D))$, $h(\omega, t,\lambda)(e_j) = \{x\mapsto h_j(\omega, t, x, \lambda)\}$, for any $j\in\mathbb{N}$, where $h_j:\Omega\times (0, T)\times D\times\mathbb{R}\rightarrow\mathbb{R}$ satisfies (H1) and (H2).

\subsection{Approximation of h}\label{approxh}
In the following, we will work with $\lambda\in\mathbb{R}$ fixed. Our first step in the approximation procedure is the approximation of the function $h$. To this end, we project the Wiener process onto the finite-dimensional space. For $N\in\mathbb{N}$ we define $h^N:\Omega\times (0, T)\rightarrow HS(L^2(D))$ (projection onto the space spanned by first $N$ vectors of the basis) by setting
\begin{align*}
h^N(e_j) = \left\{\begin{array}{ll}
     h(\omega,t,\lambda)(e_j) \quad \text{ for }j\leq N , \\
     0 \quad \text{ otherwise.} 
\end{array}
\right.
\end{align*}
and for $h^N$ we consider a sequence of elementary processes $\{h^N_{n}\}_{n\in\mathbb{N}} \subset S^2_W(0, T; W^{s, 2}_0(D))$ (cf. \cite[Definition 2.3.1]{Liu2015Stochastic}) for $s >1$ large enough, such that $\partial_{x_i} h^N_n(e_j)\in L^\infty(\Omega\times Q_T)$, for all $i = 1, ...,d$ and all $j=1, ..., N$, and
$$
\lim_{n\rightarrow\infty}\|h^N_n - h^N\|_T=0.
 $$
To avoid the overload of indexes we skip the subscript $n$ and superscript $N$ in the subsequent arguments.
\begin{rem}\label{rem:L_infty_behaviour_of_h_approx}
We recall that from \cite[Proposition 4.20]{daprato2014stochastic} it follows that $\int_0^\cdot h\diff W\in L^2(\Omega;\mathcal{C}([0,T];W^{s,2}_0(D)))$. Thanks to the regularity of $h$ and with a slight abuse of notation, we may define $\nabla_x h \in \mathcal{N}^2_W(0, T; L^2(D))$, by setting $[\nabla_x h](e_j)=\nabla_x h_j$. Then, since $\nabla_x$ is a continuous operator from $W^{1,2}_0(D)$ into $L^2(D)$, by \cite[Lemma 2.4.1]{Liu2015Stochastic}
it follows that
\begin{align*}
\nabla_x\int_0^t h\diff W(s) =\nabla_x\left(\sum_{j=1}^N \int_0^t h_j\diff \beta_j(s)\right)
=\sum_{j=1}^{N}\int_0^t \nabla_x h_j\diff \beta_j(s)=\int_0^t \nabla_x h \diff W(s),
\end{align*}
thus $\nabla_x\int_0^\cdot h\diff W\in L^2(\Omega;\mathcal{C}([0,T];L^2(D)))$. 

Moreover, by Burkholder--Davis--Gundy inequality \cite[Theorem 3.2.47]{LR17} for any $r\geq 2$ and any $t\in [0,T]$
\begin{align*}
\mathbb{E}\left(\left\Vert\nabla_x\int_0^t h\diff W\right\Vert_{L^2}^r\right) \leq C\mathbb{E}\left[\left(\int_0^T \|\nabla_x h\|_{HS}^2\diff \tau\right)^{r/2}\right] \leq C \sup_{j\in\{1, ..., N\}}\|\nabla_x h(e_j)\|^r_{L^\infty(\Omega\times Q_T)},
\end{align*}
Thus $\int_0^\cdot h\diff W\in L^r(\Omega\times (0, T); W^{1,2}_0(D))$, and there exists $C(\omega) > 0$, such that for a.e. $(t, x)\in Q_T$ we have $\nabla_x\int_0^\cdot h\diff W \leq C(\omega)$. In particular $\nabla_x\int_0^\cdot h\diff W \in L^\infty(Q_T)$ almost surely. 
Similarly we can argue for $\int_0^\cdot h\diff W$.
\end{rem}

Although a previous remark gives us nice behavior of $\int_0^t h\diff W$ for fixed $\omega$, we still need to extend it to the whole space.

\begin{lem}\label{lem:stoch_integral_in_musielak_space}
    Let $h\in \mathcal{S}_W^2(0, T; W^{s,2}_0(D))$ for $s$ such that
    $$
    W^{s-1,2}_0(D) \hookrightarrow L^\infty(D),
    $$
    and let $m$ satisfy the Assumption \ref{ass:growth_Young}. Then, $\int_0^\cdot h\diff W\in E_m(\Omega\times Q_T)$ and $\nabla_x\int_0^\cdot h\diff W\in E_m(\Omega\times Q_T)$.
\end{lem}
\begin{proof}
    We only show the first inclusion, as the second one follows by the same argument. Since $h\in \mathcal{S}_W^2(0, T; W^{s,2}_0(D))$, then there exists a partition $0 = t_0 < t_1 <... < t_k = T$ and
    $$
    \Phi_l: \Omega\rightarrow HS(L^2(D);W^{s,2}_0(D))
    \quad 0\leq l\leq k-1
    $$
    taking only finitely many distinct values in $HS(L^2(D);W^{s,2}_0(D))$,
    such that
    $$
    h(t) = \sum_{l = 0}^{k-1}\Phi_l\mathbf{1}_{(t_l, t_{l+1}]}(t),\quad t\in[0, T],
    $$
    and as such
  \begin{align*}
  	\int_0^t h\diff W(s) &=	\sum_{j=1}^N \int_0^t h(e_j)\diff \beta_j(t)\\
  &=	\sum_{j=1}^N \sum_{l = 0}^{k - 1}\Phi_l(e_j)(\beta_j(t_{l+1}\wedge t) - \beta_j(t_l \wedge t)).
    \end{align*}
    Take an arbitrary $\lambda > 0$. By convexity
    \begin{align*}
            &\int_\Omega\int_0^T\int_D m\left(\frac{\left|\int_0^t h\diff W(s)\right|}{\lambda}\right)\diff x\diff t\diff P(\omega)\\
           &\leq \sum\limits_{\substack{j=1,\ldots, N  \\ l=0,\ldots k-1}}\frac{\lambda_{j,l}}{\lambda} \int_\Omega\int_0^T\int_D m\left(\frac{\left|\Phi_l(e_j)(\beta_j(t_{l+1}\wedge t) - \beta_j(t_l \wedge t))\right|}{\lambda_{l,j}}\right)\diff x\diff t\diff P(\omega),
    \end{align*}
    where $\sum\limits_{\substack{j=1,\ldots, N  \\ l=0,\ldots k-1}}\lambda_{l,j}=\lambda$, with $\lambda_{l,j} \geq 0$. As the sum on the right-hand side is finite, we will simply show finiteness of any summand. To this end, we fix $j\in 1,\ldots N$ and $l\in\{0, ..., k-1\}$. Recalling that $\Phi_l$ takes only finitely many values in $HS(L^2(D);W^{s,2}_0(D))$, let the image of $\Phi_l$ be equal to $\{L_l^1, ..., L_l^{n_l}\}$ for some $n_l\in\mathbb{N}$. For $A_\nu := \Phi_l^{-1}(L_l^\nu)$, it follows that
    \begin{equation*}
    \begin{split}
        &\int_\Omega\int_0^T\int_D m\left(\frac{\left|\Phi_l(e_j)(\beta_j(t_{l+1}\wedge t) - \beta_j(t_l \wedge t))\right|}{\lambda_{l,j}}\right)\diff x\diff t\diff P(\omega)\\
        & = \sum_{\nu= 1}^{n_l}\int_{A_\nu}\int_0^T\int_D m\left(\frac{\left|L_l^{\nu}(e_j)(\beta_j(t_{l+1}\wedge t) - \beta_j(t_l \wedge t))\right|}{\lambda_{l,j}}\right) \diff x\diff t\diff P(\omega).
    \end{split}
    \end{equation*}
    Again, to show finiteness of the whole sum we may focus on a single summand. Using the fact that $m$ is nondecreasing and $L_l^{\nu}$ is in $L^{\infty}(D)$ for any $\nu\in \{1,\ldots,n_l\}$ we obtain
    \begin{equation*}
        \begin{split}
        &	\int_{A_\nu}\int_0^T\int_D m\left(\frac{\left|L_l^{\nu}(e_j)(\beta_j(t_{l+1}\wedge t) - \beta_j(t_l \wedge t))\right|}{\lambda_{l,j}}\right) \diff x\diff t\diff P(\omega)\\
            & \leq |D| \int_{A_\nu}\int_0^T m\left(\frac{\Vert L_l^{\nu}(e_j)\Vert_{\infty}|\beta_j(t_{l+1}\wedge t) - \beta_j(t_l \wedge t)|}{\lambda_{l,j}}\right)\diff t\diff P(\omega)\\
        \end{split}
    \end{equation*}
    In the following, we define $C(l) := \sup_{\nu\in\{1, ..., n_l\}}\|L_l^{\nu}(e_j)\|_{\infty}$.
    Next, we estimate
    \begin{equation*}
        \begin{split}
        	&\leq |D| \int_{A_\nu}\int_0^T m\left(\frac{\Vert L_l^{\nu}(e_j)\Vert_{\infty}|\beta_j(t_{l+1}\wedge t) - \beta_j(t_l \wedge t)|}{\lambda_{l,j}}\right)\diff t\diff P(\omega)\\
            &\leq |D|\int_{A_\nu}\int_0^T m\left(\frac{C(l)|\beta_j(t_{l+1}\wedge t) - \beta_j(t_l \wedge t)|}{\lambda_{l,j}}\right)\diff t\diff P(\omega)\\
            &\leq|D|\int_0^T \erww{m\left(\frac{C(l)|\beta_j(t_{l+1}\wedge t) - \beta_j(t_l \wedge t)|}{\lambda_{l,j}}\right)}\diff t\\
            &= |D|\int_{t_l}^{t_{l+1}}\erww{m\left(\frac{C(l)|\beta_j(t) - \beta_j(t_l)|}{\lambda_{l,j}}\right)}\diff t\\
            &+|D|\int_{t_{l+1}}^T\erww{m\left(\frac{C(l)|\beta_j(t_{l+1}) - \beta_j(t_l)|}{\lambda_{l,j}}\right)}\diff t.
        \end{split}
    \end{equation*}
Now, we recall that, for any $t>t_l$, $\beta_j(t)-\beta_j(t_l)\sim\mathcal{N}(0,t-t_l)$ and, $\beta_j(t_{l+1})-\beta_j(t_l)\sim\mathcal{N}(0,t_{l+1}-t_l)$. Then, since $m$ is non-negative and Borel measurable on $[0,\infty)$, $\Delta t:=\max_{l=0,\ldots k-1}t_{l+1}-t_l$, and the monotonicity of the Young function $m$ we obtain 
\begin{align*}
	&\int_{t_l}^{t_{l+1}}\erww{m\left(\frac{C(l)|\beta_j(t) - \beta_j(t_l)|}{\lambda_{l,j}}\right)}=\int_{t_l}^{t_{l+1}}\frac{1}{\sqrt{2\pi (t-t_l)}}\int_{\mathbb{R}}m\left(\frac{C(l)|r|}{\lambda_{l,j}}\right) \exp\left(-\frac{r^2}{2(t-t_l)}\right)\diff r\diff t\\
	&\leq \int_{t_l}^{t_{l+1}} \frac{1}{\sqrt{2\pi (t-t_l)}}\diff t\int_{\mathbb{R}} m\left(\frac{C(l)|r|}{\lambda_{l,j}}\right) \exp\left(-\frac{r^2}{2\Delta t}\right) \diff r\\
	&\leq \sqrt{\frac{2\Delta t}{\pi}}\int_{\mathbb{R}}m\left(\frac{C(l)|r|}{\lambda_{l,j}}\right)  \exp\left(-\frac{r^2}{2\Delta t}\right) \diff r
	\end{align*}
	Using Assumption \ref{ass:growth_Young} it follows that there exists $B,K> 0$ and $\beta\in (0,1)$ such that $m(|r|)\leq Ke^{B|r|^{1+\beta}}$. Therefore,
	\begin{align*}
	&\int_{\mathbb{R}}m\left(\frac{C(l)|r|}{\lambda_{l,j}}\right)  \exp\left(-\frac{r^2}{2\Delta t}\right) \diff r
	\leq K\int_{\mathbb{R}} \exp\left(\frac{\widetilde{C}(l)}{\lambda_{l,j}^{1+\beta}}|r|^{1+\beta}\right)  \exp\left(-\frac{r^2}{2\Delta t}\right) \diff r\\
	&=K\int_{\mathbb{R}} \exp\left(|r|(K_1|r|^{\beta}-K_2|r|)\right)\diff r\\
	\end{align*}
For $K_1,K_2>0$. If we choose 
\[M\geq \max\left(1,\left(\frac{K_1}{K_2}+1\right)^{\frac{1}{1-\beta}}\right),\]
it follows that 
\[(K_1|r|^{\beta}-K_2|r|)\leq -K_4|r|\]
for all $|r|\geq M$ with $K_4=K_2(1-\frac{K_1}{K_1+K_2})>0$. Therefore, there exists $K_3>0$, such that
\begin{align*}
	\int_{\mathbb{R}}m\left(\frac{C(l)|r|}{\lambda_{l,j}}\right)  \exp\left(-\frac{r^2}{2\Delta t}\right) \diff r\leq K_3+\int_{\{|r|>M\}}\exp\left(-K_4|r|^2\right) \diff r<\infty.
 	\end{align*}
In the same manner,
\begin{align*}
	&\int_{t_{l+1}}^T\erww{m\left(\frac{C(l)|\beta_j(t_{l+1}) - \beta_j(t_l)|}{\lambda_{l,j}}\right)}\diff t=\int_{t_{l+1}}^T\frac{1}{\sqrt{2\pi(t_{l+1}-t_l)}}\int_{\mathbb{R}}m\left(\frac{C(l)|r|}{\lambda_{l,j}}\right)\exp\left(-\frac{r^2}{2\Delta t}\right) \diff r\diff t\\
	&= \frac{T-t_{l+1}}{\sqrt{2\pi(t_{l+1}-t_l)}}\int_{\mathbb{R}}m\left(\frac{C(l)|r|}{\lambda_{l,j}}\right)\exp\left(-\frac{r^2}{2\Delta t}\right) \diff r<\infty.
\end{align*}
    Since the above is true for arbitrary nonnegative $\lambda_{l,j}$, therefore our $\lambda > 0$ can indeed be arbitrary, and we obtain the desired result.
\end{proof}

\subsection{Approximation of the operator}

For fixed $h$ as in Subsection \ref{approxh}, let $m$ be a Young function given by \ref{def:N_func_Young_growth} and consider an approximate operator
\begin{align}\label{def:approx_operator}
A^\varepsilon(t, x, \nabla_x u) = A(t, x, \nabla_x u) + \varepsilon \nabla_\xi m(|\nabla_x u|),
\end{align}
and the equation

\begin{align}\label{eq:approx_problem}
    \diff u^\varepsilon - \DIV_x A^\varepsilon(t, x, \nabla_x u^\varepsilon)\diff t = h\diff W.
\end{align}

The main advantage of such approximation is that now $A^\varepsilon$ belongs to  $L_{m^*}$ instead of $L_{M^*}$, which is detailed in the following proposition. The proof can be found in the proof of Lemma 3.9 in \cite{bulicek2021parabolic}.

\begin{prop}
    Suppose that $A$ satisfies Assumption \ref{ass:stress_A}, $M$ is a N-function, and $m$ is a Young function such that $M(t, x, \xi)\leq m(|\xi|)$ for all $\xi \in \mathbb{R}^d$ and a.e. $(t,x)\in(0,T)\times   D$. Then $A^\varepsilon(t, x, \xi) := A(t, x, \xi) + \varepsilon\nabla_\xi m(|\xi|)$ satisfies
    \begin{enumerate}[label = (Ap\arabic*)]
        \item $A^\varepsilon(t, x, \xi)$ is a Carath\'{e}odory function, that is for a.e. $(t, x)\in (0, T)\times D$ the mapping $\R^d\ni \xi\mapsto A^\varepsilon(t, x, \xi)\in \R$ is a continuous function, and for every $\xi\in \R^d$ the mapping $(0, T)\times D\ni (t, x)\mapsto A^\varepsilon(t, x, \xi)\in \R$ is measurable.\label{ass:A_varepsilon_1}
        \item (coercivity and growth bound) there exists a constant $c^\varepsilon > 0$ and a function $g^\varepsilon\in L^\infty(\Omega\times Q_T)$ such that for all $\xi\in \R^d$ and almost every $(t, x)\in Q_T$
        $$
            m(|\xi|) + m^*(|A^\varepsilon(t, x, \xi)|) \leq c^\varepsilon A^\varepsilon(t, x, \xi)\cdot\xi + g^\varepsilon(t, x).
        $$\label{ass:A_varepsilon_3}
        \item (strict monotonicity) for every $\xi_1\neq\xi_2\in\R^d$ and a.e. $(t, x)\in Q_T$ we have
        $$
            (A^\varepsilon(t, x, \xi_1) - A^\varepsilon(t, x, \xi_2))\cdot(\xi_1 - \xi_2) > 0.
        $$\label{ass:A_varepsilon_2}
        \item for a.e. $(t, x)\in Q_T$ $A^\varepsilon(t, x, 0) = 0$.\label{ass:A_varepsilon_4}
    \end{enumerate}
\end{prop}

As such we can expect $\nabla_x u^\varepsilon \in L_m(\Omega\times Q_T)$, but we will lose this property after $\varepsilon\to 0^+$. Either way, we may finally show the following existence result for our approximation

\begin{thm}\label{thm:existence_for_approx_in_epsilon}
    Suppose $A$ satisfies \ref{ass:A_varepsilon_1} - \ref{ass:A_varepsilon_4}, $h\in S^2_W(0, T; L^2(D))$, and the initial datum $u_0\in L^2(\Omega\times D)$ is $\mathcal{F}_0$-measurable with values in $L^2(D)$. Then, there exists a continuous $L^2(D)$-valued and $(\mathcal{F}_t)_{t\geq 0}$ adapted process $(u(t))_{t\in [0, T]}$ such that its $P\otimes \diff t$ almost everywhere version is in $L^1(\Omega\times [0, T]; W^{1,1}_0(D))\cap L_m(\Omega\times Q_T)$, $\nabla_x u\in L_m(\Omega\times Q_T)$ and it satisfies almost surely
    \begin{align}\label{eq:weak_form_for_approx}
        u(t) = u(0) + \int_0^t\DIV_x A(s, x, \nabla_x u)\diff s + \int_0^t h\diff W(s),\quad t\in[0, T],\text{ in }\mathcal{D'}(D).
    \end{align}
\end{thm}
\begin{proof}
As the proof of the theorem is relatively standard, we refer to Appendix \ref{proof:existence_theorem_approx}.
\end{proof}

\subsection{A priori bounds in $\varepsilon$}

\begin{lem}\label{lem:apropri_u_epsilon}
    Suppose that $u^\varepsilon$ is a solution to \eqref{eq:approx_problem} given by Theorem \ref{thm:existence_for_approx_in_epsilon}. Then
    \begin{enumerate}[label = ($B^\varepsilon$\arabic*)]
        \item $\{u^\varepsilon\}_{\varepsilon > 0}$ is bounded in $L^{2}(\Omega;L^{\infty}(0,T;L^2(D)))$, the space of all  weak-$\ast$ measurable mappings $v:\Omega\rightarrow L^{\infty}(0,T;L^2(D))$ such that $\erww{\Vert v\Vert^2_{L^{\infty}(0,T;L^2(D)}}<\infty$. In particular, the convergence also holds in $L^2(\Omega;L^2(0,T;L^2(D)))$ and in $L^2(0,T;L^2(\Omega;L^2(D)))$, since by \cite[Corollary 1.2.23, p.25]{HNVW16} this space is isometrically isomorphic to the latter one.
        \item $\{\varepsilon \nabla_\xi m(|\nabla_x u^\varepsilon|)\cdot\nabla_x u^\varepsilon\}_{\varepsilon > 0}$ is bounded in $L^1(\Omega\times Q_T)$,
        \item $\{\nabla_x u^\varepsilon\}_{\varepsilon>0}$ is bounded in $L_M(\Omega\times Q_T)$,
        \item $\{A(t, x, \nabla_x u^\varepsilon)\}_{\varepsilon > 0}$ is bounded in $L_{M^*}(\Omega\times Q_T)$.
    \end{enumerate}
\end{lem}
\begin{proof}
By It\^{o}'s lemma \ref{thm:orlicz_ito_formula} we obtain an energy equality for all $t\in [0,T]$, and a.s. (this set might depend on $\varepsilon$, but since we will converge with $\varepsilon\to 0^+$ in the sequel, we can make sure to fix a countable set of $\varepsilon$, and by a diagonal argument find the set of full measure in $\Omega$, for which the above holds true for any $\varepsilon$ in our countable set)
\begin{multline}\label{eq:ito_energy_u}
    \|u^\varepsilon(t)\|_{L^2(D)} - \|u_0\|_{L^2(D)} + 2\varepsilon\int_0^t\int_D \nabla_\xi m(|\nabla_x u^\varepsilon|)\cdot\nabla_x u^\varepsilon\diff x\diff \tau + 2\int_0^t\int_D A(t, x, \nabla_x u^\varepsilon)\cdot \nabla_x u^\varepsilon\diff x\diff \tau\\ = 2\int_0^t\langle u^\varepsilon,\, h\diff W\rangle_2+ \int_0^t \|h\|_{HS(L^2(D))}^2\diff\tau.
\end{multline}
Using Burkholder--Davis--Gundy inequality \cite[Theorem 3.2.47]{LR17}, denoting the adjoint operator of $h(t)$ by $h(t)^{\ast}$, we can estimate the stochastic term:
\begin{equation}\label{44bis}
\begin{split}
            &\mathbb{E}\left(\sup_{t\in[0, T]}\left|\int_0^t \langle u^\varepsilon(\tau) ,\,h (\tau)\diff W\rangle_{L^2(D)}\right|\right)\\
            &\leq C\,\mathbb{E}\left(\sqrt{\int_0^T \|h(\tau)^* u^\varepsilon(\tau)\|_{L^2(D)}^2\diff\tau}\right)\\
            &\leq C\, \mathbb{E}\left(\sqrt{\int_0^T\|u^\varepsilon(\tau)  \|^2_{L^2(D)}\|h (\tau)\|^2_{HS(L^2(D))}\diff \tau}\right)\\
            &\leq \frac{1}{4}\mathbb{E}\left(\sup_{t\in [0, T]}\|u^\varepsilon(t)  \|^2_{L^2(D)}\right) + C\mathbb{E}\int_0^T\|h\|^2_{HS(L^2(D))}\diff t.
\end{split}
\end{equation}
Plugging the above into \eqref{eq:ito_energy_u}, and using the coercivity condition \ref{ass:on_A_2} we obtain
\begin{multline*}
    \mathbb{E}(\sup_{t\in [0, T]}\|u^\varepsilon\|^2_{L^2(D)}) + 4\varepsilon\,\mathbb{E}\left(\int_{Q_T}\nabla_\xi m(|\nabla_x u^\varepsilon|)\cdot\nabla_x u^\varepsilon\diff x\diff t\right) + \frac{4}{c}\,\mathbb{E}\left(\int_{Q_T}M(t, x, \nabla_x u^\varepsilon)\diff x\diff t\right)\\ 
    + \frac{4}{c}\,\mathbb{E}\left(\int_{Q_T}M^*(t, x, A(t, x, \nabla_x u^\varepsilon))\diff x\diff t\right)\leq C\mathbb{E}\int_0^T\|h\|^2_{HS(L^2(D))}\diff t + 2\|u_0\|^2_{L^2(\Omega\times D)} + C\|g\|_\infty,
\end{multline*}
which gives our thesis.
\end{proof}

\begin{rem}\label{rem:on_u_and_A_varepsilon_m_behaviour}
    Note that \eqref{eq:ito_energy_u}, by \ref{ass:A_varepsilon_3} implies as well that $\nabla_x u^\varepsilon\in L_m(\Omega \times Q_T)$ and $A^\varepsilon(t, x, \nabla_x u^\varepsilon)\in L_{m^*}(\Omega\times Q_T)$.
\end{rem}

In what follows we would like to take full advantage of the already existing wide theory for deterministic equations developed in \cite{bulicek2021parabolic}. Thus, let us fix $\omega\in \Omega'\subset\Omega$, where $\Omega'$ is a set of full measure, for which some properties hold (it will become clear which ones as the argument progresses). For $\omega\in \Omega'$ fixed, we define
$$
v^\varepsilon(t, x) := u^\varepsilon(\omega, t, x) - \left(\int_0^th(\omega, t)\diff W\right)(x),
$$
for $u^\varepsilon$ being the solution to \eqref{eq:approx_problem} evaluated at $\omega\in\Omega'$. Recalling Remark \ref{rem:L_infty_behaviour_of_h_approx} it follows that $v^\varepsilon\in L_m(Q_T)$, $\nabla_x v^\varepsilon\in L_m(Q_T)$ and $v^{\varepsilon}$ solves
\begin{align}\label{eq:approx_epsilon_v}
    \partial_t v^\varepsilon - \DIV_x A^\varepsilon\left(t, x, \nabla_x v^\varepsilon + \nabla_x\int_0^t h\diff W\right) = 0,\text{ in }\mathcal{D}'([0, T)\times D)
\end{align}
with initial condition $v_0 = u_0$. For the convenience of the reader, we recall the following results from \cite{bulicek2021parabolic}.

\begin{prop}\label{prop:deterministic_energy_theo}\textup{(\cite[Lemma 4.2, Remark 4.3]{bulicek2021parabolic})}
   \ Let $v\in C([0, T]; L^2(D))\cap L^1(0, T; W^{1,1}_0(D))$ with $\nabla_x v\in L_M(Q_T)$ satisfy the equation
    $$
    \partial_t v = \DIV_x\alpha\quad \text{ in }\mathcal{D}'([0, T)\times D),\qquad \mbox{  where }\alpha\in L_{M^*}(Q_T),
    $$
    for an isotropic N-function $M$ satisfying Assumption \ref{ass:N_func_M} and an initial condition $v_0\in L^2(D)$. 
    Then, the following energy equalities hold true (recall the definitions \eqref{def:trunc}, \eqref{def:primitive_trunc}) for all $t\in (0, T)$, and $k\in\N$
    \begin{align*}
       \frac{1}{2}\int_D v^2(t)\diff x - \frac{1}{2}\int_D v_0^2\diff x &= -\int_0^t\int_D \alpha(s, x)\cdot\nabla_x v\diff x\diff s,\\
        \int_D G_k(v(t))\diff x - \int_D G_k(v_0(x))\diff x & = -\int_0^t\int_D \alpha(s,x)\cdot\nabla_xT_k(v(s, x))\diff x\diff s.
    \end{align*}
\end{prop}

With the above in mind, we may prove the a priori bounds for $v^\varepsilon$ in $Q_T$.

\begin{lem}\label{lem:apriori_v_epsilon}
    Let $A$ satisfy Assumption \ref{ass:stress_A} with an isotropic N-function $M$ satisfying the Assumption \ref{ass:N_func_M}. Suppose $v^\varepsilon$ solves \eqref{eq:approx_epsilon_v} with $A^\varepsilon$ given by \eqref{def:approx_operator}, and $M(t, x, \xi)\leq m(|\xi|)$. Then
    \begin{enumerate}[label = ($Bv^\varepsilon$\arabic*)]
        \item $\{v^\varepsilon\}_{\varepsilon > 0}$ is bounded in $L^\infty(0, T; L^2(D)))$,
        \item $\left\{\varepsilon\nabla_\xi m\left(\left|\nabla_x\left(v^\varepsilon + \int_0^t h\diff W\right)\right|\right)\cdot \nabla_x\left(v^\varepsilon + \int_0^t h\diff W\right)\right\}_{\varepsilon > 0}$ is bounded in $L^1(Q_T)$, \label{some_bound_Young_fun_on_v_varepsilon}
        \item $\left\{\nabla_x\left(v^\varepsilon + \int_0^t h\diff W\right)\right\}_{\varepsilon > 0}$ is bounded in $L_M(Q_T)$,
        \item $\left\{A\left(t, x, \nabla_x\left(v^\varepsilon + \int_0^t h\diff W\right)\right)\right\}_{\varepsilon > 0}$ is bounded in $L_{M^*}(Q_T)$.\label{some_bound_on_A_v_varepsilon}
    \end{enumerate}
\end{lem}
\begin{proof}
    Notice that from Remark \ref{rem:on_u_and_A_varepsilon_m_behaviour} we know that the operator
    $$
    \alpha(t, x) = A^\varepsilon\left(t, x, \nabla_x v^\varepsilon + \nabla_x\int_0^t h\diff W\right),
    $$
    is in $L_{m^*}(Q_T)$, thus by Proposition \ref{prop:deterministic_energy_theo}, we know that the solution $v^\varepsilon$ to the equation \eqref{eq:approx_epsilon_v} satisfies the energy equality
    \begin{align}\label{eq:energy_v_epsilon}
        \frac{1}{2}\frac{d}{dt}\|v^\varepsilon\|^2_{L^2(D)} + \int_D A^\varepsilon\left(t, x, \nabla_x v^\varepsilon + \nabla_x\int_0^t h\diff W\right)\cdot\nabla_x v^\varepsilon\diff x = 0.
    \end{align}
    Hence, by the coercivity condition \ref{ass:on_A_2}, the Fenchel--Young inequality and \eqref{eq:energy_v_epsilon}
    \begin{align*}
        &\frac{c}{2}\frac{d}{dt}\|v^\varepsilon\|^2_{L^2(D)} + c\varepsilon\int_D \nabla_\xi m\left(\left|\nabla_x\left(v^\varepsilon + \int_0^t h\diff W\right)\right|\right)\cdot \nabla_x\left(v^\varepsilon + \int_0^t h\diff W\right)\diff x\\
        &+ \int_D M\left(t, x, \nabla_x\left(v^\varepsilon + \int_0^t h\diff W\right)\right)\diff x
        + \int_D M^*\left(t, x, A\left(t, x, \nabla_x\left(v^\varepsilon + \int_0^t h\diff W\right)\right)\right)\diff x \\
        &\phantom{=}\leq\frac{c}{2}\frac{d}{dt}\|v^\varepsilon\|^2_{L^2(D)} + c\varepsilon\int_D \nabla_\xi m\left(\left|\nabla_x\left(v^\varepsilon + \int_0^t h\diff W\right)\right|\right)\cdot \nabla_x\left(v^\varepsilon + \int_0^t h\diff W\right)\diff x + \|g\|_{\infty} \\
        &\qquad\qquad + c\int_D A\left(t, x, \nabla_x\left(v^\varepsilon + \int_0^t h\diff W\right)\right)\cdot \nabla_x\left(v^\varepsilon + \int_0^t h\diff W\right)\diff x\\
        & \phantom{=}= \|g\|_{\infty} + c\int_D A^\varepsilon\left(t, x, \nabla_x\left(v^\varepsilon + \int_0^t h\diff W\right)\right)\cdot\nabla_x\int_0^t h\diff W\diff x \\
        & \phantom{=}\leq \|g\|_{\infty} + c\varepsilon\int_D\frac{1}{2} m^*\left(\nabla_\xi m\left(\left|\nabla_x\left(v^\varepsilon + \int_0^t h\diff W\right)\right|\right)\right)\diff x + c\varepsilon\int_D \frac{1}{2}m\left(2\left|\nabla_x\int_0^t h\diff W\right|\right)\diff x +\\
        &\qquad\qquad + \int_D \frac{1}{2}M^*\left(t, x, A\left(t, x, \nabla_x\left(v^\varepsilon + \int_0^t h\diff W\right)\right)\right)\diff x + \int_D \frac{1}{2}M\left(t, x, 2c\,\nabla_x\int_0^t h\diff W\right)\diff x .
    \end{align*}
Using the equality
    $$
    \nabla_\xi m(|\xi|)\cdot\xi = m(|\xi|) + m^*(\nabla_\xi m(|\xi|))
    $$
    we arrive at
    \begin{equation}\label{ineq:energy_ineq_for_v_varepsilon}
    \begin{split}
        &\frac{c}{2}\frac{d}{dt}\|v^\varepsilon\|^2_{L^2(D)} + \frac{c\varepsilon}{2}\int_D \nabla_\xi m\left(\left|\nabla_x\left(v^\varepsilon + \int_0^t h\diff W\right)\right|\right)\cdot \nabla_x\left(v^\varepsilon + \int_0^t h\diff W\right)\diff x\\
        &+ \frac{1}{2}\int_D M\left(t, x, \nabla_x\left(v^\varepsilon + \int_0^t h\diff W\right)\right)\diff x
        + \frac{1}{2}\int_D M^*\left(t, x, A\left(t, x, \nabla_x\left(v^\varepsilon + \int_0^t h\diff W\right)\right)\right)\diff x \\
        &\phantom{=}\lesssim \|g\|_{\infty} + c\varepsilon\int_D m\left(2\left|\nabla_x\int_0^t h\diff W\right|\right)\diff x + \int_D M\left(t, x, 2c\,\nabla_x\int_0^t h\diff W\right)\diff x,
    \end{split}
    \end{equation}
    which by Remark \ref{rem:L_infty_behaviour_of_h_approx} gives our thesis.
\end{proof}

\subsection{Convergence for the petrurbation of the operator}

We start with a result of convergence from \cite{bulicek2021parabolic}, which is a consequence of \ref{some_bound_Young_fun_on_v_varepsilon}.

\begin{lem}\label{lem:convergence_to_zero_with_approx}(\cite[Lemma 2.9]{bulicek2021parabolic})
    Suppose $v^\varepsilon$ is as in Lemma \ref{lem:apriori_v_epsilon}. Then, for any $\phi\in L^\infty(Q_T)$
    \begin{align*}
        \varepsilon\int_{Q_T}\nabla_\xi m\left(\left|\nabla_x \left(v^\varepsilon + \int_0^t h\diff W\right)\right|\right)\cdot\phi\diff x\diff t \rightarrow 0 \quad \text{ as }\varepsilon\to 0.
    \end{align*}
\end{lem}

Using the above we may obtain a limiting equation.

\begin{lem}\label{lem:conv_v_varepsilon}
    Suppose $v^\varepsilon$ is as in Lemma \ref{lem:apriori_v_epsilon}. Then, there exists $v\in L^\infty(0, T; L^2(D))$ with $\nabla_x v\in L_M(Q_T)$ and $\chi\in L_{M^*}(Q_T),$ such that (up to the subsequence, which right now possibly depends on $\omega$)
    \begin{align}
        v^\varepsilon &\wstar v &&\text{ weakly* in }L^\infty(0, T; L^2(D)),\label{conv:v_varepsilon_conv}\\
        \nabla_x v^\varepsilon &\wstar \nabla_x v &&\text{ weakly* in }L_M(Q_T),\label{conv:nabla_v_varepsilon_conv}\\
        A\left(t, x, \nabla_x\left(v^\varepsilon + \int_0^t h\diff W\right)\right) &\wstar \chi &&\text{ weakly* in }L_{M^*}(Q_T),\label{conv:stress_conv}\\
        v^\varepsilon & \rightarrow v &&\text{ strongly in }L^1(Q_T). \label{conv:strong_convergence_of_v_varepsilon}
    \end{align}
    Moreover $v$ and $\chi$ given above satisfy the equation
    \begin{align}\label{eq:after_limit_with_v_before_identification_equation}
        \partial_t v - \DIV_x \chi = 0\quad \text{ in }\mathcal{D}'([0, T)\times D).
    \end{align}
        In particular, by Proposition \ref{prop:deterministic_energy_theo} it satisfies the energy equalities
    \begin{align}\label{eq:energy_equality_after_limit}
        \frac{1}{2}\frac{d}{dt}\|v(t)\|_{L^2(D)} + \int_D\chi(t, x)\cdot\nabla_x v(t, x)\diff x= 0,
    \end{align}
    \begin{align}\label{eq:energy_equality_after_limit_trunc}
        \frac{d}{dt}\int_D G_k(v(t))\diff x + \int_D \chi(t, x)\cdot\nabla_x T_k(v(t, x))\diff x = 0.
    \end{align}
\end{lem}

\begin{proof}
    We simply apply Lemma \ref{lem:apriori_v_epsilon} together with the Banach--Alaoglu theorem, then we use Lemma \ref{lem:convergence_to_zero_with_approx} to converge in the weak formulation \eqref{eq:approx_epsilon_v} for $v^\varepsilon$. At last to prove the strong convergence of $v^\varepsilon$ note that by Lemma \ref{lem:apriori_v_epsilon} the sequence $\{v^\varepsilon\}_{\varepsilon > 0}$ is bounded in $L^1(0, T; W^{1,1}_0(D))$. Moreover, looking at the equation \eqref{eq:approx_epsilon_v}, we can say using bound \ref{some_bound_on_A_v_varepsilon} and Lemma \ref{lem:convergence_to_zero_with_approx}, and arguing similarly as in \cite[Lemma 2.10 and Lemma 2.11]{bulicek2021parabolic}, that $\{\partial_t v^\varepsilon\}_{\varepsilon > 0}$ is bounded in $L^1(0, T; W^{-s, 2}(D))$, for $s$ big enough such that $W^{s-1, 2}_0(D)\hookrightarrow L^\infty(D)$. Then, as $W^{1,1}_0(D)$ embeds compactly into $L^1(D)$, which in turn, using the mapping
    $$
    L^1(D) \ni f \mapsto T_f\in W^{-s, 2}(D), \quad T_f(\phi) = \int_{D}f(x)\,\phi(x)\diff x,
    $$
    embeds continuously into $W^{-s, 2}(D)$, by a classical Aubin--Lions' lemma we get the thesis.
\end{proof}

Our aim now is to identify $\chi$ with $A\left(t, x, \nabla_x\left(v + \int_0^t h\diff W\right)\right)$, which we will do by a standard monotonicity trick. To this end, we need to compare the energy \eqref{eq:energy_equality_after_limit} with the limit of \eqref{eq:energy_v_epsilon}.

\begin{lem}\label{lem:chi_identification_by_mono_trick}
    Suppose $v^\varepsilon$ is as in Lemma \ref{lem:apriori_v_epsilon} and $v$, $\chi$ are given by Lemma \ref{lem:conv_v_varepsilon}. Then
    $$
    \chi(t, x) = A\left(t, x, \nabla_x\left(v + \int_0^t h\diff W\right)\right) \text{ for a.e. }(t, x)\in Q_T.
    $$
\end{lem}

\begin{proof}
 We may write \eqref{eq:energy_v_epsilon} equivalently
\begin{equation}\label{eq:some_equality_in_conv_with_epsilon_1}
\begin{split}
    \frac{1}{2}\frac{d}{dt}\|v^\varepsilon\|^2_{L^2(D)} &+ \int_D A^\varepsilon\left(t, x, \nabla_x v^\varepsilon + \nabla_x\int_0^t h\diff W\right)\cdot\nabla_x\left( v^\varepsilon + \int_0^t h\diff W\right)\diff x\\
    &= \int_D A^\varepsilon\left(t, x, \nabla_x v^\varepsilon + \nabla_x\int_0^t h\diff W\right)\cdot\nabla_x\int_0^t h\diff W\diff x.
\end{split}
\end{equation}
As by Remark \ref{rem:L_infty_behaviour_of_h_approx} $\nabla_x \int_0^t h\diff W\in L^\infty(Q_T)$ we can use Lemma \ref{lem:convergence_to_zero_with_approx} and \eqref{conv:stress_conv} as well as the weak lower semicontinuity of the $L^2$ norm in \eqref{eq:some_equality_in_conv_with_epsilon_1} to obtain
\begin{multline}\label{ineq:limit_energy_ineq_for_strong_conv}
    \frac{1}{2}\|v(t)\|^2_{L^2(D)} + \limsup_{\varepsilon\to 0}\int_{Q_t} A\left(s, x, \nabla_x v^\varepsilon + \nabla_x\int_0^s h\diff W\right)\cdot\nabla_x\left( v^\varepsilon + \int_0^s h\diff W\right)\diff x\diff s\\
    \leq \int_{Q_t} \chi\cdot\nabla_x\int_0^s h\diff W\diff x\diff s + \frac{1}{2}\|u_0\|_{L^2(D)}.
\end{multline}
With $Q_t:=(0,t)\times D$ for any $t\in 0,T]$. Combining the above with \eqref{eq:energy_equality_after_limit} we arrive at
\begin{multline}\label{ineq:some_ineq_epsilon_mono_trick}
    \limsup_{\varepsilon\to 0}\int_{Q_t} A\left(s, x, \nabla_x v^\varepsilon + \nabla_x\int_0^s h\diff W\right)\cdot\nabla_x\left( v^\varepsilon + \int_0^s h\diff W\right)\diff x\diff s\\
    \leq \int_{Q_t} \chi\cdot\nabla_x\left(v + \int_0^s h\diff W\right)\diff x\diff s.
\end{multline}
With this in mind, we conclude in a typical manner. By the monotonicity assumption \ref{ass:on_A_3} for any $\eta\in L^\infty(Q_T)^d$
\begin{align*}
    \int_0^T\int_D \left(A\left(t, x, \nabla_x v^\varepsilon + \nabla_x\int_0^t h\diff W\right) - A\left(t, x, \eta + \nabla_x\int_0^t h\diff W\right)\right)(\nabla_x v^\varepsilon - \eta)\diff x\diff t \geq 0.
\end{align*}
Using \eqref{conv:nabla_v_varepsilon_conv}, \eqref{conv:stress_conv}, Proposition \ref{prop:stress_bound_for_bounded} and \eqref{ineq:some_ineq_epsilon_mono_trick} we conclude
\begin{align*}
    \int_0^T\int_D \left(\chi - A\left(t, x, \eta + \nabla_x\int_0^t h\diff W\right)\right)(\nabla_x v - \eta)\diff x\diff t \geq 0,
\end{align*}
thus by the monotonicity trick, see Lemma~\ref{res:monot_trick}, $\chi = A\left(t, x, \nabla_x v + \nabla_x\int_0^t h\diff W\right)$ a.e. in $Q_T$. 
\end{proof}

\begin{cor}

Note that we can now identify $\chi$ in \eqref{eq:after_limit_with_v_before_identification_equation} and by Proposition \ref{prop:deterministic_energy_theo} with the use of the energy equality \eqref{eq:energy_equality_after_limit} we deduce that $v$ is unique as a solution. Hence, the classical subsequence argument allows us to see that in Lemma \ref{lem:conv_v_varepsilon} the whole sequence converges.
    
\end{cor}

At this point, we basically have all the needed tools to show the existence of solutions to our equations, as long as the function $h$ is nice enough. We will need to converge with the approximation of $h$ at some point and thus, we will need an energy inequality. Therefore we need to upgrade the convergence of $\nabla_x v^\varepsilon$.

\begin{lem}\label{lem:almost_everywhere_conv_of_nabla_v_varepsilon}
    Suppose $v^\varepsilon$ is as in Lemma \ref{lem:apriori_v_epsilon} and $v$ is given in Lemma \ref{lem:conv_v_varepsilon}. Then, up to the subsequence possibly dependant on $\omega$,
    $$
    \nabla_x v^\varepsilon \rightarrow \nabla_x v \text{ a.e. in }Q_T.
    $$
\end{lem}
\begin{proof}
    The proof is basically a repeat of the argument given in \cite{elmahi2005parabolic}. Since it is not stated there as a separate result with particular assumptions, we shall write it here briefly for the convenience of the reader. For the sake of the argument we come back to the variable
    $$
    w^\varepsilon := v^\varepsilon + \int_0^t h\diff W(s),\qquad w:= v + \int_0^t h\diff W(s).
    $$
    Both of those variables are closely related to $u^\varepsilon$, but since we want to keep it denoting a function defined on a probability space $\Omega$, we introduce $w$ and $w^\varepsilon$ for this argument only. Here, recall the definition \eqref{def:trunc} of the truncation $T_k$. Note that by \eqref{conv:strong_convergence_of_v_varepsilon} we can choose a subsequence of $w^\varepsilon$ which converges almost everywhere to $w$. Thus, it is easy to verify that
    \begin{align*}
        \mathbf{1}_{\{|w^\varepsilon|\leq k\}} \rightarrow \mathbf{1}_{\{|w|\leq k\}} \text{ strongly in }L_M(Q_T)\text{ and }L_{M^*}(Q_T),
    \end{align*}
    and by \eqref{conv:v_varepsilon_conv}, \eqref{conv:stress_conv}, Lemma \ref{lem:chi_identification_by_mono_trick}
    \begin{align}
        \nabla_x T_k(w^\varepsilon) &\wstar \nabla_x T_k(w) &&\text{ weakly* in }L_M(Q_T),\label{conv:nabla_w_varepsilon_conv_trunc}\\
        A\left(t, x, \nabla_xT_k(w^\varepsilon)\right) &\wstar A\left(t, x, \nabla_xT_k(w)\right) &&\text{ weakly* in }L_{M^*}(Q_T).\label{conv:stress_w_trunc}
    \end{align}
    Moreover, fix a cut-off function $\psi\in C_c^\infty(\Omega)$, $0\leq \psi\leq 1$. Due to the use of the truncation operator $T_k$ we can now apply Proposition \ref{prop:double_mollifier_conv} to conclude that there exists a sequence $\{\phi_j\}_{j\in\N}\subset L^\infty(0, T; C^\infty(D))$ such that $\nabla_x\phi_j\,\psi$ converges to $\nabla_x T_k(w)\,\psi$ modularly (and by Proposition \ref{prop:vitali_in_musielaki} also almost everywhere).
    In the following argument let $\chi^s$ denote a characteristic function of the set $\{(x, t)\in Q_T:\,\,|\nabla_x w|\leq s\}$ and $\chi_j^s$ of $\{(x, t)\in Q_T:\,\, |\nabla_x\phi_j|\leq s\}$. To start, we can write
    \begin{align*}
        &\int_{Q_T}\left(A\left(t, x, \nabla_x T_k(w^\varepsilon) \right) - A\left(t, x, \nabla_x T_k(w)\,\chi^s\right) \right)(\nabla_x T_k(w^\varepsilon) - \nabla_x T_k(w)\,\chi^s)\,\psi\diff x\diff t\\
        &\phantom{=}-\int_{Q_T}\left(A\left(t, x, \nabla_x T_k(w^\varepsilon) \right) - A\left(t, x, \nabla_x \phi_j\,\chi_j^s \right)\right)(\nabla_x T_k(w^\varepsilon) - \nabla_x\phi_j\,\chi_j^s)\,\psi\diff x\diff t\\
        &= \int_{Q_T}A\left(t, x, \nabla_x \phi_j\,\chi_j^s \right)(\nabla_x T_k(w^\varepsilon) - \nabla_x\phi_j\,\chi_j^s)\,\psi\diff x\diff t\\
        &\phantom{=}+\int_{Q_T}A\left(t, x, \nabla_x T_k(w^\varepsilon) \right)(\nabla_x\phi_j\,\chi_j^s - \nabla_x T_k(w)\,\chi^s)\,\psi\diff x\diff t\\
        &\phantom{=}+\int_{Q_T}A\left(t, x, \nabla_x T_k(w)\,\chi^s\right)(\nabla_x T_k(w)\,\chi^s - \nabla_x T_k(w^\varepsilon))\,\psi\diff x\diff t\\
        &=: I^{\varepsilon, j, s}_1 + I^{\varepsilon, j, s}_2 + I^{\varepsilon, j, s}_3.
    \end{align*}
Our goal is to converge in the following order: first with $\varepsilon\to 0^+$, then with  $j\to\infty$, and finally with $s\to\infty$. For $I_1$ by Proposition \ref{prop:stress_bound_for_bounded} and \eqref{conv:nabla_w_varepsilon_conv_trunc} we may converge with $\varepsilon\to 0^+$ to get
$$
\lim_{\varepsilon\to 0^+}I^{\varepsilon, j, s}_1 = \int_{Q_T}A\left(t, x, \nabla_x \phi_j\,\chi_j^s \right)(\nabla_x T_k(w) - \nabla_x\phi_j\,\chi_j^s)\psi\diff x \diff t =: I^{j, s}_1.
$$
Note that we can write equivalently
\begin{align*}
I^{j, s}_1 = \int_{Q_T}\chi_j^s\,A\left(t, x, \nabla_x \phi_j\right)(\nabla_x T_k(w) - \nabla_x \phi_j)\psi\diff x\diff t.
\end{align*}
Using Proposition \ref{prop:stress_bound_for_bounded} and the modular convergence of $\nabla_x\phi_j\,\psi$
\begin{align*}
    \left|\int_{Q_T}\chi_j^s\,A\left(t, x, \nabla_x \phi_j\right)(\nabla_x T_k(w) - \nabla_x \phi_j)\psi\diff x\diff t\right|
    \lesssim \int_{Q_T}|\nabla_x T_k(w) - \nabla_x \phi_j|\psi\diff x\diff t\rightarrow 0\text{ as }j\to\infty.
\end{align*}
Thus
\begin{align*}
    \lim_{s\to +\infty}\lim_{j\to\infty}I^{j, s}_1 = 0.
\end{align*}
Moving to $I_2^{\varepsilon, j, s}$ we may conclude by \eqref{conv:stress_w_trunc} and Lemma \ref{lem:chi_identification_by_mono_trick}
\begin{align*}
    \lim_{\varepsilon\to 0^+}I_2^{\varepsilon, j, s} = \int_{Q_T}A\left(t, x, \nabla_x T_k(w)\right)(\nabla_x\phi_j\,\chi_j^s - \nabla_x T_k(w)\,\chi^s)\psi\diff x\diff t =: I^{j, s}_2.
\end{align*}
To converge here we may note that 
$$
\nabla_x\phi_j\,\chi_j^s\,\psi \rightarrow  \nabla_x T_k(w)\,\chi^s\,\psi\text{ a.e. in }Q_T
$$
and by convexity of $M$
\begin{multline*}
M\left(t, x, \frac{\nabla_x\phi_j\,\chi_j^s\,\psi}{\lambda}\right) =\chi_j^s\,M\left(t, x,  \frac{\nabla_x\phi_j\,\psi}{\lambda}\right)\\
\leq \frac{\lambda_1}{\lambda}M\left(t, x, \frac{\nabla_x\phi_j\,\psi - \nabla_x T_k(w)\,\psi}{\lambda_1}\right) + \frac{\lambda_2}{\lambda}M\left(t, x, \frac{\nabla_x T_k(w)\,\psi}{\lambda_2}\right),
\end{multline*}
where $\lambda := \lambda_1 + \lambda_2$, and $\lambda_1 > 0$ is such that
$$
M\left(t, x, \frac{\nabla_x\phi_j\,\psi - \nabla_x T_k(w)\,\psi}{\lambda_1}\right) \rightarrow 0\quad\text{ as }j\to\infty.
$$
(which exists by modular convergence of $\nabla_x\phi_j\,\psi$ to $\nabla_x T_k(w)\,\psi$) and for $\lambda_2 > 0$
$$
M\left(t, x, \frac{\nabla_x T_k(w)\,\psi}{\lambda_2}\right)\in L^1(Q_T).
$$
Hence $\left\{M\left(t, x, \frac{\nabla_x\phi_j\,\chi_j^s\,\psi}{\lambda}\right)\right\}_j$ is equiintegrable and by Proposition \ref{prop:vitali_in_musielaki}
\begin{align}\label{conv:trunc_with_chi_modular}
\nabla_x\phi_j\,\chi_j^s\,\psi \xrightarrow{M}  \nabla_x T_k(w)\,\chi^s\,\psi\quad\text{ modularly in }L_{M}(Q_T).
\end{align}
In the end, by Proposition \ref{prop:conv_of_product_in_musielaki}
\begin{align*}
    \lim_{j\to \infty}I_2^{j, s} = 0.
\end{align*}
To converge in $I^{\varepsilon, j, s}_3$ we utilize once again Proposition \ref{prop:stress_bound_for_bounded} together with \eqref{conv:nabla_w_varepsilon_conv_trunc} to get
\begin{align*}
    \lim_{\varepsilon\to 0^+}I_3^{\varepsilon, j, s} = \int_{Q_T}A\left(t, x, \nabla_x T_k(w)\,\chi^s\right)(\nabla_x T_k(w)\,\chi^s - \nabla_x T_k(w))\psi\diff x\diff t = 0.
\end{align*}
Combining everything together we have proven
\begin{align*}
    &\int_{Q_T}\left(A\left(t, x, \nabla_x T_k(w^\varepsilon)\right) - A\left(t, x, \nabla_x T_k(w)\,\chi^s\right) \right)(\nabla_x T_k(w^\varepsilon) - \nabla_x T_k(w)\,\chi^s)\psi\diff x\diff t\\
    &=\int_{Q_T}\left(A\left(t, x, \nabla_x T_k(w^\varepsilon)\right) - A\left(t, x, \nabla_x \phi_j\,\chi_j^s\right)\right)(\nabla_x T_k(w^\varepsilon) - \nabla_x\phi_j\,\chi_j^s)\psi\diff x\diff t + \delta(\varepsilon, j , s),
\end{align*}
where $\delta(\varepsilon, j , s)\rightarrow 0$ as $\varepsilon\to 0^+$, $j\to +\infty$, $s\to +\infty$. Thus, for every $s\geq r > 0$ we have
\begin{align*}
    0&\leq\int_{\{|\nabla_x w|\leq r\}}\left(A\left(t, x, \nabla_x T_k(w^\varepsilon) \right) - A\left(t, x, \nabla_x T_k(w)\right)\right)(\nabla_x T_k(w^\varepsilon) - \nabla_x T_k(w))\psi\diff x\diff t\\
    &\phantom{=}\leq \int_{\{|\nabla_x v|\leq s\}}\left(A\left(t, x, \nabla_x T_k(w^\varepsilon) \right) - A\left(t, x, \nabla_x T_k(w)\right)\right)(\nabla_x T_k(w^\varepsilon) - \nabla_x T_k(w))\psi\diff x\diff t\\
    &\phantom{=}\leq \int_{Q_T}\left(A\left(t, x, \nabla_x T_k(w^\varepsilon) \right) - A\left(t, x, \nabla_x T_k(w)\,\chi^s\right) \right)(\nabla_x T_k(w^\varepsilon) - \nabla_x T_k(w)\,\chi^s)\psi\diff x\diff t \\
    &=\int_{Q_T}\left(A\left(t, x, \nabla_x T_k(w^\varepsilon)\right) - A\left(t, x, \nabla_x \phi_j\,\chi_j^s\right)\right)(\nabla_x T_k(w^\varepsilon) - \nabla_x\phi_j\,\chi_j^s)\psi\diff x\diff t\\
    &\qquad + \delta(\varepsilon, j , s)\\
    &=-\int_{Q_T}A\left(t, x, \nabla_x \phi_j\,\chi_j^s\right)(\nabla_x T_k(w^\varepsilon) - \nabla_x\phi_j\,\chi_j^s)\psi\diff x\diff t\\
    &\phantom{=}+ \int_{Q_T}A\left(t, x, \nabla_x T_k(w^\varepsilon)\right) \nabla_x T_k(w^\varepsilon)\,\psi\diff x\diff t\\
    &\phantom{=}-\int_{Q_T}A\left(t, x, \nabla_x T_k(w^\varepsilon)\right)\nabla_x\phi_j\,\chi_j^s\,\psi\diff x\diff t\\
    &\qquad +\delta(\varepsilon, j, s)\\
    &= J^{\varepsilon, j, s}_1 + J^{\varepsilon, j, s}_2 + J^{\varepsilon, j, s}_3 + \delta(\varepsilon, j, s).
\end{align*}
Similarly to $I^{\varepsilon, j, s}_1$ we can deduce
$$
J^{\varepsilon, j, s}_1  \rightarrow 0 \text{ as }\varepsilon\to 0^+, j\to +\infty, s\to+\infty.
$$
At the same time by, arguing the same way as for \eqref{ineq:some_ineq_epsilon_mono_trick}, but instead using Proposition \ref{prop:deterministic_energy_theo} and \eqref{eq:energy_equality_after_limit_trunc}
$$
\limsup_{\varepsilon\to 0^+}J_2^{\varepsilon, j, s} \leq J_2 := \int_{Q_T}A\left(t, x, \nabla_x T_k(w)\right) \nabla_x T_k(w)\,\psi\diff x\diff t.
$$
To treat the third term, we simply apply \eqref{conv:stress_w_trunc} to get
$$
\lim_{\varepsilon\to 0^+}J_3^{\varepsilon, j, s} = -\int_{Q_T}A\left(t, x, \nabla_x T_k(w)\right)\nabla_x\phi_j\,\chi_j^s\,\psi\diff x\diff t =: J_3^{j,s}.
$$
Moreover, by \eqref{conv:trunc_with_chi_modular}
$$
\lim_{j\to +\infty}J_3^{j, s} = -\int_{Q_T}A\left(t, x, \nabla_x T_k(w)\right)\nabla_x T_k(w)\,\chi^s\,\psi\diff x\diff t =: J_3^{s}.
$$
Now, we can write
\begin{align*}
    J_2^{} + J_3^{s} = \int_{\{|\nabla_x T_k(w)| > s\}}A\left(t, x, \nabla_x T_k(w)\right)\nabla_x T_k(w)\,\psi\diff x\diff t,
\end{align*}
which converges to $0$, as $s\to +\infty$, since $A\left(t, x, \nabla_x T_k(w)\right)\nabla_x T_k(w)\,\psi\in L^1(Q_T)$.

Thus, as $A$ is strictly monotone by \ref{ass:on_A_3}, up to the subsequence
$$
\nabla_x T_k(w^\varepsilon) \rightarrow \nabla_x T_k(w) \text{ a.e. in }\{|\nabla_x v|\leq r\}\cap \{\psi\equiv 1\},
$$
and by diagonal argument a.e. in $Q_T$. Hence, by applying de la Vall\'{e}e Poussin criterion and the Vitali convergence theorem
$$
\nabla_x T_k(w^\varepsilon) \rightarrow \nabla_x T_k(w) \text{ strongly in }L^1(Q_T).
$$
At last, we can show the stated almost everywhere convergence. We have
\begin{align*}
    &\lim_{\varepsilon\to 0^+}\int_{Q_T}\sqrt{|\nabla_x w^\varepsilon - \nabla_x w|}\diff x\diff t\\
    &\leq \lim_{\varepsilon\to 0^+}\int_{Q_T}\sqrt{|\nabla_x w^\varepsilon - \nabla_x T_k(w^\varepsilon)|} + \sqrt{|\nabla_xT_k(w^\varepsilon) - \nabla_x T_k(w)|} + \sqrt{|\nabla_x T_k(w) - \nabla_x w|}\diff x\diff t\\
    &\leq \lim_{\varepsilon\to 0^+}C\left(\|\nabla_x w^\varepsilon\|_1^{1/2}|\{|w^\varepsilon > k\}|^{1/2} + \|\nabla_x w\|_1^{1/2}|\{|w > k\}|^{1/2}\right) \leq \frac{C}{\sqrt{k}}.
\end{align*}
Letting $k\to +\infty$ we obtain the thesis up to the subsequence.
\end{proof}

\begin{lem}\label{lem:modular_conv_v_varepsilon}
    Suppose $v^\varepsilon$ is as in Lemma \ref{lem:apriori_v_epsilon} and $v$ is given in Lemma \ref{lem:conv_v_varepsilon}. Then
    \begin{align*}
        \nabla_x v^\varepsilon&\xrightarrow{M} \nabla_x v\qquad &&\text{ modularly in }L_M(Q_T),\\
        A\left(t, x, \nabla_x v^\varepsilon + \nabla_x\int_0^t h\diff W\right) &\xrightarrow{M^*} A\left(t, x, \nabla_x v + \nabla_x\int_0^t h\diff W\right)\qquad &&\text{ modularly in }L_{M^*}(Q_T).
    \end{align*}
\end{lem}

\begin{proof}
    The proof is directly the same as, for example in \cite{gwiazda2008onnonnewtonian}, but we will repeat it for the convenience of the reader. Take an arbitrary subsequence of $\nabla_x v^\varepsilon$. By Lemma \ref{lem:almost_everywhere_conv_of_nabla_v_varepsilon} there exists a subsequence that converges almost everywhere in $Q_T$. By the continuity of $A$ on its third variable
    $$
    A\left(t, x, \nabla_x v^\varepsilon + \nabla_x\int_0^t h\diff W\right) \rightarrow A\left(t, x, \nabla_x v + \nabla_x\int_0^t h\diff W\right)
    $$
    a.e. in $Q_T$ as well. Moreover, by Fatou's lemma
    \begin{multline*}
        \liminf_{\varepsilon\to 0^+}\int_{Q_T}A\left(t, x, \nabla_x v^\varepsilon + \nabla_x\int_0^t h\diff W\right)\cdot\nabla_x\left( v^\varepsilon + \int_0^t h\diff W\right)\diff x\diff t\\
        \geq \int_{Q_T} A\left(t, x, \nabla_x v + \nabla_x\int_0^t h\diff W\right)\cdot\nabla_x\left(v + \int_0^t h\diff W\right)\diff x\diff t
    \end{multline*}
    and by \eqref{ineq:some_ineq_epsilon_mono_trick}
    \begin{multline*}
        \limsup_{\varepsilon\to 0^+}\int_{Q_t} A\left(s, x, \nabla_x v^\varepsilon + \nabla_x\int_0^s h\diff W\right)\cdot\nabla_x\left( v^\varepsilon + \int_0^s h\diff W\right)\diff x\diff s\\
        \leq \int_{Q_t} A\left(s, x, \nabla_x v + \nabla_x\int_0^s h\diff W\right)\cdot\nabla_x\left(v + \int_0^s h\diff W\right)\diff x\diff s.
    \end{multline*}
    Hence,
    \begin{multline*}
        \lim_{\varepsilon\to 0^+}\int_{Q_T} A\left(t, x, \nabla_x v^\varepsilon + \nabla_x\int_0^t h\diff W\right)\cdot\nabla_x\left( v^\varepsilon + \int_0^t h\diff W\right)\diff x\diff t\\
        =\int_{Q_T} A\left(t, x, \nabla_x v + \nabla_x\int_0^t h\diff W\right)\cdot\nabla_x\left(v + \int_0^t h\diff W\right)\diff x\diff t.
    \end{multline*}
    Noticing that
    \begin{equation*}
        \begin{split}
            &\int_{Q_T}\Bigg|A\left(t, x, \nabla_x v^\varepsilon + \nabla_x\int_0^t h\diff W\right)\cdot\nabla_x\left( v^\varepsilon + \int_0^t h\diff W\right)\\
            &\qquad- A\left(t, x, \nabla_x v + \nabla_x\int_0^t h\diff W\right)\cdot\nabla_x\left( v + \int_0^t h\diff W\right)\Bigg|\diff x\diff t\\
            &=\int_{Q_T}\Bigg(A\left(t, x, \nabla_x v^\varepsilon + \nabla_x\int_0^t h\diff W\right)\cdot\nabla_x\left( v^\varepsilon + \int_0^t h\diff W\right)\\
            &\qquad- A\left(t, x, \nabla_x v + \nabla_x\int_0^t h\diff W\right)\cdot\nabla_x\left( v + \int_0^t h\diff W\right)\Bigg)\diff x\diff t\\
            &+ 2\int_{Q_T}\Bigg[A\left(t, x, \nabla_x v^\varepsilon + \nabla_x\int_0^t h\diff W\right)\cdot\nabla_x\left( v^\varepsilon + \int_0^t h\diff W\right)\\
            &\qquad- A\left(t, x, \nabla_x v + \nabla_x\int_0^t h\diff W\right)\cdot\nabla_x\left( v + \int_0^t h\diff W\right)\Bigg]_{-}\diff x\diff t,
        \end{split}
    \end{equation*}
    we conclude by Lebesgue's Dominated Convergence theorem
    $$
    A\left(t, x, \nabla_x v^\varepsilon + \nabla_x\int_0^t h\diff W\right) \rightarrow A\left(t, x, \nabla_x v + \nabla_x\int_0^t h\diff W\right),
    $$
    strongly in $L^1(Q_T)$. By Dunford--Pettis theorem, this implies uniform equiintegrability, which in turn, by the coercivity and growth condition \ref{ass:on_A_2} yields the uniform bound and equiintegrability in $L^1$ of $\{M(t, x, \nabla_x v^\varepsilon + \nabla_x\int_0^t h\diff W)\}_{\varepsilon > 0}$ and $\left\{M^*\left(t, x,  A\left(t, x, \nabla_x v^\varepsilon + \nabla_x\int_0^t h\diff W\right)\right)\right\}_{\varepsilon > 0}$. By Proposition \ref{prop:vitali_in_musielaki} and the standard subsequence argument we obtain the thesis.
\end{proof}


\begin{cor}
    By Lemma \ref{lem:modular_conv_v_varepsilon} and Proposition \ref{prop:conv_of_product_in_musielaki} we have
    \begin{multline}\label{conv:connvergence_of_the_product_mono_varepsilon}
        \lim_{\varepsilon\to 0^+}\int_{Q_t} A\left(s, x, \nabla_x v^\varepsilon + \nabla_x\int_0^s h\diff W\right)\cdot\nabla_x\left( v^\varepsilon + \int_0^s h\diff W\right)\diff x\diff s\\
        =\int_{Q_s} A\left(s, x, \nabla_x v + \nabla_x\int_0^s h\diff W\right)\cdot\nabla_x\left(v + \int_0^s h\diff W\right)\diff x\diff s,
    \end{multline}
    for the whole sequence.
\end{cor}
\begin{cor}
    By \eqref{conv:connvergence_of_the_product_mono_varepsilon} and Lemma \ref{lem:chi_identification_by_mono_trick} applied to \eqref{eq:some_equality_in_conv_with_epsilon_1} (compare also with \eqref{ineq:limit_energy_ineq_for_strong_conv}) we can deduce
    $$
    \limsup_{\varepsilon\to 0^+}\|v^\varepsilon(t)\|_{L^2(D)} \leq \|v(t)\|_{L^2(D)}\quad\text{ for a.e. }t\in [0, T],
    $$
    which together with weak lower semicontinuity of the $L^2$ norm implies
    $$
    \|v^\varepsilon(t)\|_{L^2(D)}\rightarrow \|v(t)\|_{L^2(D)}\quad \text{ for a.e. }t\in[0, T].
    $$
    This in turn easily gives us
    \begin{align}\label{conv:strong_conv_in_L2_of_v}
        v^\varepsilon\rightarrow v  \quad\text{ strongly in }L^2(Q_T).
    \end{align}
\end{cor}
Here, note that \eqref{conv:strong_conv_in_L2_of_v}, together with \eqref{conv:stress_conv} and the identification in Lemma \ref{lem:chi_identification_by_mono_trick} are enough to converge in the weak formulation for the approximation \eqref{eq:weak_form_for_approx} to obtain \eqref{eq:weak_form_final} for $u:= v + \int_0^th\diff W$, but we would lack the information about the measurability in the probability space. At last, we want to upgrade our convergences to $\Omega\times Q_T$, which is the main reason, why we need Assumption \ref{ass:growth_Young}, and Lemma \ref{lem:stoch_integral_in_musielak_space}.

\begin{lem}
    Suppose $v^\varepsilon$ is as in Lemma \ref{lem:apriori_v_epsilon} and $v$ is given as in Lemma \ref{lem:conv_v_varepsilon}. Then,
    $$
     v^\varepsilon \rightarrow v\quad\text{ strongly in }L^2(\Omega\times Q_T).
    $$
\end{lem}
\begin{proof}
    By \eqref{conv:strong_conv_in_L2_of_v} we know that for a.e. $\omega, t\in \Omega\times [0, T]$
  $$
    \|v^\varepsilon(\omega, t)\|_{L^2(D)}\rightarrow \|v(\omega, t)\|_{L^2(D)}.
    $$
    Moreover, \eqref{ineq:energy_ineq_for_v_varepsilon} implies that
    \begin{equation*}
    \begin{split}
       &    \|v^\varepsilon(\omega,t)\|_{L^2(D)}
       \\ & \lesssim \|g\|_{\infty} + \int_{Q_t} m\left(2\left|\nabla_x\int_0^\tau h\diff W\right|\right)\diff x \diff \tau + \int_{Q_t} M\left(\tau, x, 2\,\nabla_x\int_0^\tau h\diff W\right)\diff x \diff \tau + \|u_0\|_{L^2(D)}.
    \end{split}
    \end{equation*}
    As the right-hand side of the above is integrable by Lemma \ref{lem:stoch_integral_in_musielak_space} and independent of $\varepsilon$, Lebesgue's dominated convergence theorem proves the thesis.
\end{proof}

\begin{lem}\label{lem:modular_conv_whole_prob_time_space_v_varepsilon}
    Suppose $v^\varepsilon$ is as in Lemma \ref{lem:apriori_v_epsilon}, and $v$ is given in Lemma \ref{lem:conv_v_varepsilon}. Then
    $$
    \nabla_x v^\varepsilon \xrightarrow{M} \nabla_x v\quad\text{ modularly in }L_M(\Omega\times Q_T).
    $$
\end{lem}
\begin{proof}
    Fix an arbitrary subsequence of $\nabla_x v^\varepsilon$ and take its subsequence which satisfies Lemma \ref{lem:almost_everywhere_conv_of_nabla_v_varepsilon}. From Lemma \ref{lem:modular_conv_v_varepsilon} we know that there exists $\lambda_1 > 0$ such that
    $$
    \int_{Q_T}M\left(t, x, \frac{\nabla_x v^\varepsilon - \nabla_x v}{\lambda_1}\right)\diff x\diff t \rightarrow 0.
    $$
    Thus, by \eqref{ineq:energy_ineq_for_v_varepsilon} and the convexity, for $\lambda_2 := 1 + \lambda_1$
    \begin{align*}
    &\int_{Q_T}M\left(t, x, \frac{\nabla_x v + \int_0^t h\diff W(s)}{\lambda_2}\right)\diff x\diff t\\
    &\leq\frac{\lambda_1}{\lambda_2}\int_{Q_T}M\left(t, x, \frac{\nabla_x v^\varepsilon - \nabla_x v}{\lambda_1}\right)\diff x\diff t + \frac{1}{\lambda_2}\int_{Q_T}M\left(t, x, \nabla_x v^\varepsilon + \int_0^t h\diff W(s)\right)\diff x\diff t\\
    & \lesssim\frac{\lambda_1}{\lambda_2}\int_{Q_T}M\left(t, x, \frac{\nabla_x v^\varepsilon - \nabla_x v}{\lambda_1}\right)\diff x\diff t + \|g\|_{\infty} + \varepsilon\int_{Q_T} m\left(2\left|\nabla_x\int_0^t h\diff W\right|\right)\diff x \diff t\\
    &\qquad + \int_{Q_T} M\left(t, x, 2\,\nabla_x\int_0^t h\diff W\right)\diff x \diff t.
    \end{align*}
    Taking $\varepsilon\to 0^+$ we obtain
    $$
        \int_{Q_T}M\left(t, x, \frac{\nabla_x v + \int_0^t h\diff W(s)}{\lambda_2}\right)\diff x\diff t \lesssim \|g\|_{\infty} + \int_{Q_T} M\left(t, x, 2\,\nabla_x\int_0^t h\diff W\right)\diff x \diff t
    $$
    Hence, for $\lambda := 1 + \lambda_2$, by \eqref{ineq:energy_ineq_for_v_varepsilon}, convexity, and the symmetry of $M$ with respect to the third variable, we can estimate
    \begin{equation*}
    \begin{split}
        &\int_{Q_T}M\left(t, x, \frac{\nabla_x v^\varepsilon - \nabla_x v}{\lambda}\right)\diff x\diff t\\
        &\leq \frac{\lambda_2}{\lambda}\int_{Q_T}M\left(t, x, \frac{\nabla_x v + \int_0^t h\diff W(s)}{\lambda_2}\right)\diff x\diff t + \frac{1}{\lambda}\int_{Q_T}M\left(t, x, \nabla_x v^\varepsilon + \int_0^t h\diff W(s)\right)\diff x\diff t\\
        &\lesssim \|g\|_{\infty} + \int_{Q_T} m\left(2\left|\nabla_x\int_0^t h\diff W\right|\right)\diff x \diff t + \int_{Q_T} M\left(t, x, 2\,\nabla_x\int_0^t h\diff W\right)\diff x  \diff t.
    \end{split}
    \end{equation*}
    As the right-hand side is integrable with respect to the probability measure due to Lemma \ref{lem:stoch_integral_in_musielak_space}, then the thesis for our subsequence follows from Lebesgue's dominated convergence Theorem. A classical subsequence argument gives us the convergence for the whole sequence.
\end{proof}

Finally, we can state a result for $u^\varepsilon$ instead of $v^\varepsilon$.

\begin{lem}
    Let $u^\varepsilon$ be as in Lemma \ref{lem:apropri_u_epsilon}, $v^\varepsilon$ as in Lemma \ref{lem:apriori_v_epsilon}, and $v$ be defined in Lemma \ref{lem:conv_v_varepsilon}. Then, there exists $u\in L^2(\Omega; L^\infty(0, T; L^2(D)))\cap L^1(\Omega\times (0, T); W^{1,1}_0(D))$ with $\nabla_x u\in L_M(\Omega\times Q_T)$ such that (up to the subsequence)
    \begin{align}
        u^\varepsilon &\wstar u &&\text{ weakly* in }L^2(\Omega; L^\infty(0, T; L^2(D))),\\
        \nabla_x u^\varepsilon &\xrightarrow{M} \nabla_x u &&\text{ modularly in }L_M(\Omega\times Q_T),\\
        u^\varepsilon &\rightharpoonup u &&\text{ weakly in }L^1(\Omega\times (0, T); W^{1,1}_0(D)).
    \end{align}
    Moreover $u := v + \int_0^t h\diff W(s)$.
\end{lem}
\begin{proof}
    By Lemma \ref{lem:apropri_u_epsilon} and the Banach--Alaoglu theorem, we know that there exists $u\in L^2(\Omega; L^\infty(0, T; L^2(D)))$ with $\nabla_x u \in L_M(\Omega\times Q_T)$, such that up to the subsequence
    \begin{align*}
        u^\varepsilon &\wstar u &&\text{ weakly* in }L^2(\Omega; L^\infty(0, T; L^2(D))),\\
        \nabla_x u^\varepsilon &\wstar \nabla_x u &&\text{ weakly* in }L_M(\Omega\times Q_T).
    \end{align*}
    Here, note that the boundedness in $L_M(\Omega\times Q_T)$ implies boundedness of the integral
    $$
    \mathbb{E}\int_{Q_T}\tilde{m}(|\nabla_x u^\varepsilon|)\diff x\diff t,
    $$
    for a Young function $\tilde{m}$, which is superlinear. Thus, de la Vall\'{e}e--Poussin criterion, together with Dunford--Pettis theorem give us
    \begin{align*}
    u^\varepsilon &\rightharpoonup u &&\text{ weakly in }L^1(\Omega\times (0, T); W^{1,1}_0(D)).
    \end{align*}
    At last note that by definition $u^\varepsilon = v^\varepsilon + \int_0^t h\diff W(s)$. Hence by Lemma \ref{lem:modular_conv_whole_prob_time_space_v_varepsilon}
    $$
      \nabla_x u^\varepsilon \xrightarrow{M} \nabla_x v + \nabla_x\int_0^t h\diff W(s) \quad \text{ modularly in }L_M(\Omega\times Q_T).
    $$
    By the uniqueness of the limit
    $$
        \nabla_x u = \nabla_x v + \nabla_x \int_0^t h\diff W(s),
    $$
    and by Poincar\'{e}'s inequality and boundary conditions, this implies
    $$
        u = v + \int_0^t h\diff W(s).
    $$
\end{proof}
The last result is enough to pass to the limit in the equation \eqref{eq:approx_problem} on $\Omega\times Q_T$ and to conclude that $u$ is a solution to \eqref{eq:main_sys} in the sense of Definition \ref{def:solution_to_main_system} for an additive $h\in \mathcal{S}^2_W(0, T; W^{s,2}_0(D))$ as defined at the beginning of this section.

\subsection{Convergence with the approximation on $h^N$}

Here, we come back to our labels from the beginning of the section. In the following let $\{h^N_n\}_{n\in\N}\subset \mathcal{S}^2_W(0, T; W^{s,2}_0(D))$ be a sequence of functions converging to $h^N$ in $\mathcal{N}^2_W(0, T; L^2(D))$. Denote by $u^\varepsilon_n$ a solution to \eqref{eq:approx_problem} for $h^N_n$ and by $u^\varepsilon_m$ the analogous quantity for $h^N_m$. Then, by It\^{o}'s Theorem~\eqref{thm:orlicz_ito_formula} for $-\alpha(t, x) = A^\varepsilon(t, x, \nabla_x u_n) - A^\varepsilon(t, x, \nabla_x u_m)$ we obtain
\begin{equation*}
    \begin{split}
        &\|(u^\varepsilon_n - u^\varepsilon_m)(t)\|^2_{L^2(D)} + 2\int_0^t\int_D(A(s, x, \nabla_x u^\varepsilon_n) - A(s, x, \nabla_x u^\varepsilon_m))\cdot\nabla_x (u^\varepsilon_n(s,x) - u^\varepsilon_m(s,x))\diff x\diff s\\
        &\phantom{=} \leq\int_0^t\|h^N_n - h^N_m\|^2_{L_2(U; L^2(D))}\diff s + 2\int_0^t\langle u^\varepsilon_n(s) - u^\varepsilon_m(s), h^N_n(s) - h^N_m(s)\diff W(s)\rangle_2.
    \end{split}
\end{equation*}
Similarly as for \eqref{44bis} we can estimate
\begin{align*}
    &\mathbb{E}\left(\sup_{t\in[0, T]}\left|\int_0^t\langle u^\varepsilon_n(s) - u^\varepsilon_m(s), h^N_n(s) - h^N_m(s)\diff W(s)\rangle_2\right|\right)\\
    &\quad\leq \frac{1}{4}\mathbb{E}(\sup_{t\in[0, T]}\|u^\varepsilon_n - u^\varepsilon_m\|^2_{L^2(D)}) + C\mathbb{E}\int_0^T\|h^N_n - h^N_m\|_{L_2(U; L^2(D))}^2\diff t.
\end{align*}
Thus
\begin{equation}\label{250207_01}
    \begin{split}
        &\mathbb{E}(\sup_{t\in[0, T]}\|(u^\varepsilon_n - u^\varepsilon_m)(t)\|^2_{L^2(D)})\\
        &\qquad+ \mathbb{E}\int_0^T\int_D(A(t, x, \nabla_x u^\varepsilon_n) - A(t, x, \nabla_x u^\varepsilon_m))\cdot\nabla_x (u^\varepsilon_n(t,x) - u^\varepsilon_m(t,x))\diff x\diff t\\
        &\phantom{=} \lesssim \int_0^T\|h^N_n - h^N_m\|^2_{L_2(U; L^2(D))}\diff t.
    \end{split}
\end{equation}
By the weak* lower semicontinuity of the $L^2(\Omega; L^\infty(0, T; L^2(D)))$ norm and the Fatou's lemma, we may pass to the limit with $\varepsilon\rightarrow 0^+$ in \eqref{250207_01} to obtain
\begin{equation}\label{ineq:for_conv_with_approx_in_h_n}
    \begin{split}
        &\mathbb{E}(\sup_{t\in[0, T]}\|(u_n - u_m)(t)\|^2_{L^2(D)})\\
        &\qquad+ \mathbb{E}\int_0^T\int_D(A(t, x, \nabla_x u_n) - A(t, x, \nabla_x u_m))\cdot\nabla_x (u_n(t,x) - u_m(t,x))\diff x\diff t\\
        &\phantom{=} \lesssim \int_0^T\|h^N_n - h^N_m\|^2_{L_2(U; L^2(D))}\diff t.
    \end{split}
\end{equation}
Now we can proceed to converging with $n\to +\infty$. As the right-hand side of \eqref{ineq:for_conv_with_approx_in_h_n} goes to $0$ as $n, m\to +\infty$, $\{u_n\}$ is a Cauchy sequence in $L^2(\Omega; L^\infty(0, T; L^2(D)))$. Moreover, by considering simply the equation for $u_n^\varepsilon$, we can get an analogous energy inequality (one can imagine putting $u_m = 0$ in the above), hence we can deduce that the sequences $\{\nabla_x u_n\}_{n\in \N}$, $\{A(t, x, \nabla_x u_n)\}_{n\in \N}$ are weakly* compact in $L_{M}(\Omega\times Q_T)$, $L_{M^*}(\Omega\times Q_T)$ respectively. \\

Thus, there exists $u\in L^2(\Omega; L^\infty(0, T; L^2(D)))$ with $\nabla_x u\in L_M(\Omega\times Q_T)$ and a $\chi\in L_{M^*}(\Omega\times Q_T)$ such that
\begin{align}
     u_n &\rightarrow u &&\text{ strongly in }L^2(\Omega; L^\infty(0, T; L^2(D))),\label{conv:strong_conv_in_n_u_n}\\
     \nabla_x u_n &\wstar \nabla_x u &&\text{ weakly* in }L_M(\Omega\times Q_T),\\
     A(t, x, \nabla_x u_n) &\wstar \chi &&\text{ weakly* in }L_{M^*}(\Omega\times Q_T).
\end{align}
Again, we need to identify $\chi$ with $A(t, x, \nabla_x u)$ to conclude. To this end, we are going to prove a lemma concerning the energy equality for \eqref{eq:weak_form_final}.

\begin{lem}\label{lem:trunc_energy_eq}
    Let $M$ be an N-function satisfying Assumption \ref{ass:N_func_M}.\\
    Suppose $u\in L^2(\Omega; L^\infty(0, T; L^2(D)))\cap L^1(\Omega; L^1(0, T; W^{1,1}_0(D))$ with $\nabla_x u\in L_{M}(\Omega\times Q_T)$ satisfies almost surely the equation
    \begin{align}\label{eq:some_eq_for_trunck_lemma_weak}
        u(t) = u_0 + \int_0^t \DIV_x \alpha\diff s + \int_0^t h\diff W(s)\quad \text{ in }\mathcal{D}'(D),
    \end{align}
    for $\alpha\in L_{M^*}(\Omega\times Q_T)$ and $h\in\mathcal{N}^2_W(0, T; L^2(D))$. Then, for $G_{k}$, $T_{k}$ defined in \eqref{def:trunc}, \eqref{def:primitive_trunc}, the following energy equality holds almost surely
    \begin{multline*}
        \int_{D}G_{k}(u(t))\diff x - \int_{D}G_{k}(u_0)\diff x = -\int_0^t\int_D \alpha(s, x)\cdot \nabla_x T_{k}(u)\diff x\diff s\\
        + \int_0^t\langle T_{k}(u), h\diff W(s)\rangle_2 + \int_0^t\sum_{j = 1}^\infty \int_D T'_{k}(u)\,|h(e_j)|^2 \diff x \diff s.
    \end{multline*}
\end{lem}

\begin{proof}
Let $0\leq \psi\leq 1$, $\psi\in C_c^\infty(D)$ be an arbitrary function and $\rho_{\kappa}$ be a family of standard mollification kernels in the space variable. For arbitrary function $f$ we denote in the following $f^\kappa = \rho_{\kappa}*f$. Then, we can mollify \eqref{eq:some_eq_for_trunck_lemma_weak} and multiply it by $\psi$ to get
\begin{align}\label{eq:weak_form_to_use_trunc_ito}
    u^\kappa(t)\psi(x) = u_0^\kappa\psi(x) + \int_0^t(\DIV_x \alpha^\kappa(s,x))\psi(x)\diff s + \left(\int_0^t h\diff W(s)\right)^\kappa\psi(x).
\end{align}
Notice that the operator $L^2(D)\ni f \mapsto f^\kappa\psi\in W^{s, 2}_0(D)$ is linear and bounded (where $s$ is big enough so that $W^{s,2}_0(D)\hookrightarrow L^\infty(D)$). Hence, by \cite[Lemma 2.4.1]{Liu2015Stochastic}
\begin{align}\label{eq:mollifier_into_stoch_int}
    \left(\int_0^t h\diff W(s)\right)^\kappa\psi(x) = \int_0^t h^\kappa\psi\diff W(s),
\end{align}
where $h^\kappa\psi \in \mathcal{N}^2_W(0, T; W^{s, 2}_0(D))$. Let $G_{k,\delta}$, $T_{k,\delta}$ be as in \eqref{def:trunc_smooth}. Define
$$
F: W^{s,2}_0(D) \rightarrow \R,\qquad F(u) = \int_{D}G_{k,\delta}(u)\diff x.
$$
Then 
$$
DF: W^{s, 2}_0(D) \rightarrow \mathcal{L}(W^{s,2}_0(D); \R),\qquad \langle DF(u), f\rangle = \int_D T_{k,\delta}(u)\,f\diff x
$$ 
and 
$$
D^2F: W^{s,2}_0(D):\rightarrow \mathcal{L}(W^{s,2}_0(D); W^{s, 2}_0(D)),\qquad \langle D^2 F(u)(f),g\rangle = \int_D T'_{k,\delta}(u)\,f\,g\diff x.
$$
It is easy to check that $F$, $DF$, $D^2 F$ are globally Lipschitz, and thus uniformly continuous. In consequence, we may use the classical It\^{o}'s lemma \ref{prop:standard_ito} in the equation \eqref{eq:weak_form_to_use_trunc_ito} to obtain almost surely
\begin{multline}\label{eq:trunc_ito_kappa_psi}
    \int_{D}G_{k,\delta}(u^\kappa(t)\psi(x))\diff x - \int_{D}G_{k,\delta}(u_0^\kappa\psi(x))\diff x = \int_0^t\int_D \DIV_x \alpha^\kappa(s, x)\, T_{k,\delta}(u^\kappa)\psi(x)\diff x\diff s\\
    + \int_0^t\langle T_{k,\delta}(u^\kappa), h^\kappa\psi\diff W(s)\rangle_2 + \int_0^t\sum_{j = 1}^\infty \langle D^2F(u^\kappa)(h^\kappa\psi)(e_j), h^\kappa\psi(e_j)\rangle\diff s.
\end{multline}
Our goal now is to converge with $\kappa\to 0^+$ and $\psi\to 1$, $\delta\to 0^+$ in this order. By Lebesgue's dominated convergence theorem
\begin{align*}
    \int_{D}G_{k,\delta}(u^\kappa(t)\psi(x))\diff x - \int_{D}G_{k,\delta}(u_0^\kappa\psi(x))\diff x  \xrightarrow {\kappa \to 0^+}\int_{D}G_{k,\delta}(u(t)\psi(x))\diff x - \int_{D}G_{k,\delta}(u_0\psi(x))\diff x,
\end{align*}
and after integrating by parts and using Proposition \ref{prop:double_mollifier_conv}
\begin{multline*}
    \int_0^t\int_D \DIV_x \alpha^\kappa(s, x)\, T_{k,\delta}(u^\kappa)\psi(x)\diff x\diff s = -\int_0^t\int_D \alpha(s, x)\cdot(\nabla_x (T_{k,\delta}(u^\kappa)\psi)^\kappa\diff x\diff s\\
    \rightarrow -\int_0^t\int_D \alpha(s, x)\cdot\nabla_x (T_{k,\delta}(u)\psi)\diff x\diff s .
\end{multline*}
Thus we only need to treat the stochastic terms. First, we will show that
\begin{align}\label{conv:kappa_stoch_term_conv_1}
    \int_0^t\langle T_{k,\delta}(u^\kappa), h^\kappa\psi\diff W(s)\rangle_2 \rightarrow \int_0^t\langle T_{k,\delta}(u), h\psi\diff W(s)\rangle_2 \quad \mbox{ as }\kappa \to 0^+.
\end{align}
Denote $\mathbb{M}(t) := \int_0^t h\diff W(s)$. Then, by It\^{o}'s isometry \cite[Proposition 2.3.5]{Liu2015Stochastic}
\begin{equation}\label{conv:kappa_stoch_term_conv_1_1}
    \begin{split}
        &\left\|\int_0^t\langle T_{k,\delta}(u^\kappa), h^\kappa\psi\diff W(s)\rangle_2 - \int_0^t\langle T_{k,\delta}(u^\kappa), h\psi\diff W(s)\rangle_2\right\|_{T}^2\\
        &\phantom{=}= \mathbb{E}\int_0^T\sum_{j\in\N}\langle T_{k,\delta}(u^\kappa), (h(e_j))^\kappa\psi - h(e_j)\psi\rangle^2\diff t\\
        &\phantom{=}=\mathbb{E}\int_0^T\sum_{j\in\N}\left(\int_D T_{k,\delta}(u^\kappa)((h(e_j))^\kappa\psi - h(e_j)\psi)\diff x \right)^2\diff t\\
        &\lesssim\mathbb{E}\int_0^T\sum_{j\in\N} \|(h(e_j))^\kappa\psi - h(e_j)\psi\|_{L^2(D)}^2\diff t\\
        &\phantom{=}=\|h^\kappa\psi - h\psi\|_T^2\\
        &\phantom{=}=\left\|\int_0^th^\kappa\psi\diff W(s) - \int_0^t h\psi\diff W(s)\right\|^2_{\mathcal{M}_T^2(L^2(D))}\\
        &
         \phantom{=}=\left\|\mathbb{M}(t)^\kappa\psi - \mathbb{M}(t)\psi\right\|^2_{\mathcal{M}_T^2(L^2(D))}\\
        &
        \phantom{=}=\mathbb{E}\|\mathbb{M}(T)^\kappa\psi - \mathbb{M}(T)\psi\|_{L^2(D)}^2\rightarrow 0,
    \end{split}
\end{equation}
where in the second to last equality we have used \eqref{eq:mollifier_into_stoch_int} and the convergence at the end is the consequence of the standard properties of the mollifier and Lebesgue's dominated convergence theorem. Similarly
\begin{equation}\label{conv:kappa_stoch_term_conv_1_2}
    \begin{split}
        &\left\|\int_0^t\langle T_{k,\delta}(u^\kappa), h\psi\diff W(s)\rangle_2 - \int_0^t\langle T_{k,\delta}(u), h\psi\diff W(s)\rangle_2\right\|_{T}^2\\
        &\phantom{=}= \mathbb{E}\int_0^T\sum_{j\in\N}\langle T_{k,\delta}(u^\kappa) - T_{k,\delta}(u), h(e_j)\psi\rangle^2\diff t\\
        &\phantom{=}=\mathbb{E}\int_0^T\sum_{j\in\N}\left(\int_D (T_{k,\delta}(u^\kappa) - T_{k,\delta}(u))\psi\,h(e_j)\diff x \right)^2\diff t\\
        &\leq\mathbb{E}\int_0^T\sum_{j\in\N} \|(T_{k,\delta}(u^\kappa) - T_{k,\delta}(u))\psi\|^2_{L^2(D)}\|h(e_j)\|_{L^2(D)}^2\diff t\rightarrow 0,
    \end{split}
\end{equation}
where the last convergence is true due to Lebesgue's dominated convergence theorem and standard properties of mollifiers. Combining \eqref{conv:kappa_stoch_term_conv_1_1} and \eqref{conv:kappa_stoch_term_conv_1_2} we get
\begin{align*}
    \left\|\int_0^t\langle T_{k,\delta}(u^\kappa), h^\kappa\psi\diff W(s)\rangle_2 - \int_0^t\langle T_{k,\delta}(u), h\psi\diff W(s)\rangle_2\right\|_{T}^2\to 0,
\end{align*}
which implies \eqref{conv:kappa_stoch_term_conv_1}. Our second claim is
\begin{align}\label{conv:kappa_stoch_term_conv_2}
    \int_0^t\sum_{j = 1}^\infty \langle D^2F(u^\kappa)(h^\kappa\psi)(e_j), h^\kappa\psi(e_j)\rangle\diff s \rightarrow \int_0^t\sum_{j = 1}^\infty \langle D^2F(u)(h\psi)(e_j), h\psi(e_j)\rangle\diff s.
\end{align}
To prove it note that, similarly as in \eqref{conv:kappa_stoch_term_conv_1_1}, by standard Cauchy-Schwarz inequality in $HS(L^2(D)))$ and Young's convolution inequality
\begin{equation}\label{conv:kappa_stoch_term_conv_2_1}
    \begin{split}
        &\mathbb{E}\left|\int_0^t\sum_{j \in\N} \langle D^2F(u^\kappa)(h^\kappa\psi)(e_j), h^\kappa\psi(e_j)\rangle\diff s - \int_0^t\sum_{j = 1}^\infty \langle D^2F(u^\kappa)(h^\kappa\psi)(e_j), h(e_j)\psi\rangle\diff s\right|\\
        &\phantom{=}= \mathbb{E}\left|\int_0^t\sum_{j \in\N} \int_D T'_{k,\delta}(u^\kappa)\,(h(e_j))^\kappa\psi((h(e_j))^\kappa\psi - h(e_j)\psi)\diff x\diff s\right|\\
        &\leq\mathbb{E}\int_0^t\|T'_{k,\delta}(u^\kappa)\,h^\kappa\psi\|_{HS(L^2(D)))}\|h^\kappa\psi - h\psi\|_{HS(L^2(D)))}\diff s\\
        &\lesssim \sqrt{\mathbb{E}\int_0^T\|h\|^2_{HS(L^2(D)))}\diff t}\sqrt{\mathbb{E}\int_0^T\|h^\kappa\psi - h\psi\|_{HS(L^2(D)))}^2\diff t}\\
        &\phantom{=}= \|h\|_T\|h^\kappa\psi - h\psi\|_T \rightarrow 0.
    \end{split}
\end{equation}
In the same way we can show
\begin{equation}\label{conv:kappa_stoch_term_conv_2_2}
    \mathbb{E}\left|\int_0^t\sum_{j \in\N} \langle D^2F(u^\kappa)(h^\kappa\psi)(e_j), h(e_j)\psi\rangle\diff s - \int_0^t\sum_{j = 1}^\infty \langle D^2F(u^\kappa)(h\psi)(e_j), h(e_j)\psi\rangle\diff s\right| \rightarrow 0.
\end{equation}
On the other hand
\begin{equation}\label{conv:kappa_stoch_term_conv_2_3}
    \begin{split}
        &\mathbb{E}\left|\int_0^t\sum_{j \in\N} \langle D^2F(u^\kappa)(h\psi)(e_j), h(e_j)\psi\rangle\diff s - \int_0^t\sum_{j = 1}^\infty \langle D^2F(u)(h\psi)(e_j), h(e_j)\psi\rangle\diff s\right|\\
        &\phantom{=}= \mathbb{E}\left|\int_0^t\sum_{j \in\N} \int_D (T'_{k,\delta}(u^\kappa) - T'_{k,\delta}(u))\,|h(e_j)|^2\,\psi^2\diff x\diff s\right| \rightarrow 0,
    \end{split}
\end{equation}
by Lebesgue's dominated convergence theorem. Combining \eqref{conv:kappa_stoch_term_conv_2_1} - \eqref{conv:kappa_stoch_term_conv_2_3} we obtain \eqref{conv:kappa_stoch_term_conv_2}. Hence, after the convergence with $\kappa\to 0^+$ in \eqref{eq:trunc_ito_kappa_psi}
\begin{multline*}
    \int_{D}G_{k,\delta}(u(t)\psi(x))\diff x - \int_{D}G_{k,\delta}(u_0\psi(x))\diff x = -\int_0^t\int_D  \alpha(s, x)\cdot\nabla_x( T_{k,\delta}(u)\psi(x))\diff x\diff s\\
    + \int_0^t\langle T_{k,\delta}(u), h\psi\diff W(s)\rangle_2 + \int_0^t\sum_{j = 1}^\infty \langle D^2F(u)(h\psi)(e_j), h\psi(e_j)\rangle\diff s.
\end{multline*}
To get rid of the localization we can notice that every term which does not contain $\nabla_x\psi$ can be easily dealt with by the virtue of Lebesgue's dominated convergence theorem. On the other hand, the only term which contains $\nabla_x\psi$
$$
\int_0^t\int_D  T_{k,\delta}(u)\, \alpha(s, x)\cdot\nabla_x\psi(x)\diff x\diff s,
$$
converges to $0$ by Proposition \ref{prop:conv_of_localization}. At last, the convergence with $\delta\to 0^+$ in every term is done simply by Lebesgue's dominated convergence theorem. Thus, we conclude the proof of our thesis.
\end{proof}

Having the above we may use monotonicity trick to identify $\chi$ with $A(t, x, \nabla_x u)$. Indeed, by Lemma \ref{lem:trunc_energy_eq} applied to the equation for $u_n$ and $u$ respectively we obtain
\begin{multline}\label{eq:mono_trick_for_u_n_energy}
        \int_{D}G_{k}(u_n(t))\diff x - \int_{D}G_{k}(u_0)\diff x = -\int_0^t\int_D A(s, x, \nabla_x u_n)\cdot \nabla_x T_{k}(u_n)\diff x\diff s\\
        + \int_0^t\langle T_{k}(u_n), h^N_n\diff W(s)\rangle_2 + \int_0^t\sum_{j = 1}^\infty \int_D T'_{k}(u_n)\,|h^N_n(e_j)|^2 \diff s,
\end{multline}
and
\begin{multline}\label{eq:mono_trick_for_u_n_limit_energy}
        \int_{D}G_{k}(u(t))\diff x - \int_{D}G_{k}(u_0)\diff x = -\int_0^t\int_D \chi(s, x)\cdot \nabla_x T_{k}(u)\diff x\diff s\\
        + \int_0^t\langle T_{k}(u), h^N\diff W(s)\rangle_2 + \int_0^t\sum_{j = 1}^\infty \int_D T'_{k}(u)\,|h^N(e_j)|^2 \diff s.
\end{multline}
To converge with $n\to +\infty$ notice that as long as $u_n\to u$ almost everywhere and $h^N_n \to h^N$ strongly in $\mathcal{N}^2_W(0, T; L^2(D))$, then we can simply repeat the arguments given in \eqref{conv:kappa_stoch_term_conv_1_1} - \eqref{conv:kappa_stoch_term_conv_1_2}, \eqref{conv:kappa_stoch_term_conv_2_1} - \eqref{conv:kappa_stoch_term_conv_2_3} to conclude
\begin{align*}
    &\int_{D}G_{k}(u_n(t))\diff x - \int_{D}G_{k}(u_0)\diff x \rightarrow \int_{D}G_{k}(u(t))\diff x - \int_{D}G_{k}(u_0)\diff x,\\
    &\int_0^t\langle T_{k}(u_n), h^N_n\diff W(s)\rangle_2 \rightarrow \int_0^t\langle T_{k}(u), h^N\diff W(s)\rangle_2,\\
    &\int_0^t\sum_{j = 1}^\infty \int_D T'_{k}(u_n)\,|h^N_n(e_j)|^2 \diff s \rightarrow \int_0^t\sum_{j = 1}^\infty \int_D T'_{k}(u)\,|h^N(e_j)|^2 \diff s,
\end{align*}
which applied to \eqref{eq:mono_trick_for_u_n_energy} gives us
\begin{multline*}
    \limsup_{n\to+\infty}\int_0^t\int_D A(s, x, \nabla_x u_n)\cdot \nabla_x T_{k}(u_n)\diff x\diff s \leq -\int_{D}G_{k}(u(t))\diff x + \int_{D}G_{k}(u_0)\diff x\\
    + \int_0^t\langle T_{k}(u), h^N\diff W(s)\rangle_2 + \int_0^t\sum_{j = 1}^\infty \int_D T'_{k}(u)\,|h^N(e_j)|^2 \diff s.
\end{multline*}
With the use of \eqref{eq:mono_trick_for_u_n_limit_energy} on the above we conclude
\begin{align}\label{ineq:mono_trick_for_u_n_limsup_ineq}
    \limsup_{n\to+\infty}\int_0^t\int_D A(s, x, \nabla_x u_n)\cdot \nabla_x T_{k}(u_n)\diff x\diff s \leq \int_0^t\int_D \chi(s, x)\cdot \nabla_x T_{k}(u)\diff x\diff s.
\end{align}
Now denote $\chi_k := \chi\,\mathbf{1}_{\{|u|\leq k\}}$. Then, by the monotonicity assumption on $A$ \ref{ass:on_A_3} we have for arbitrary $\eta\in L^\infty(Q_T)$
\begin{align*}
    \int_0^T\int_D (A(t, x, \nabla_x T_k(u_n)) - A(t, x, \eta)(\nabla_x T_k(u_n) - \eta)\diff x\diff t \geq 0.
\end{align*}
By Proposition \ref{prop:stress_bound_for_bounded} and the weak* convergences of $A(t, x \nabla_x u_n)$, $\nabla_x u_n$ and the almost everywhere convergence of $u_n$ we can converge in every mixed term above, and by \eqref{ineq:mono_trick_for_u_n_limsup_ineq} we can treat the product term to obtain
\begin{align*}
    \int_0^T\int_D (\chi_k - A(t, x, \eta)(\nabla_x T_k(u) - \eta)\diff x\diff t \geq 0.
\end{align*}
Hence, by the monotonicity trick \ref{res:monot_trick} $\chi_k = A(t, x, \nabla_x T_k(u))$ almost everywhere for any $k\in \N$, and in consequence $\chi = A(t, x, \nabla_x u)$.

\subsection{Convergence with the approximation on $h$.} In the following let $\{h^N\}_{N\in\N}\subset \mathcal{N}^2_W(0, T; L^2(D))$ be a sequence of functions converging to $h$ in $\mathcal{N}^2_W(0, T; L^2(D))$. Denote by $u^{N_1}$ a solution to \eqref{eq:main_sys} for $h^{N_1}$ and by $u^{N_2}$ the analogous quantity for $h^{N_2}$. Note that similarly to \eqref{ineq:for_conv_with_approx_in_h_n} we can obtain
\begin{equation*}
        \mathbb{E}(\sup_{t\in[0, T]}\|(u^{N_1}_n - u_n^{N_2})(t)\|^2_{L^2(D)}) \lesssim \int_0^T\|h^{N_1}_n - h^{N_2}_n\|^2_{HS(L^2(D))}\diff t.
\end{equation*}
Due to \eqref{conv:strong_conv_in_n_u_n} we can converge in the above with $n\to +\infty$ to get
\begin{align}\label{1}
    \mathbb{E}(\sup_{t\in[0, T]}\|(u^{N_1} - u^{N_2})(t)\|^2_{L^2(D)}) \lesssim \int_0^T\|h^{N_1} - h^{N_2}\|^2_{HS(L^2(D))}\diff t.
\end{align}
Thus $\{u^N\}_{N\in\N}$ is a Cauchy sequence in $L^2(\Omega; L^\infty(0, T; L^2(D)))$, and there exists $u$, such that
\begin{align}\label{conv:strong_conv_in_N_u_N}
    u^N\rightarrow u\quad\text{ strongly in }L^2(\Omega; L^\infty(0, T; L^2(D))).
\end{align}
Moreover, by applying Lemma \ref{lem:trunc_energy_eq} we deduce
\begin{multline*}
        \int_{D}G_{k}(u^N(t))\diff x - \int_{D}G_{k}(u_0)\diff x = -\int_0^t\int_D A(s, x, \nabla_x u^N)\cdot \nabla_x T_{k}(u^N)\diff x\diff s\\
        + \int_0^t\langle T_{k}(u^N), h^N\diff W(s)\rangle_2 + \int_0^t\sum_{j = 1}^\infty \int_D T'_{k}(u^N)\,|h^N(e_j)|^2 \diff s,
\end{multline*}
Thus
\begin{equation}\label{2}
\begin{split}
    &\int_0^t\int_D A(s, x, \nabla_x T_k(u^N))\cdot \nabla_x T_{k}(u^N)\diff x\diff s = \int_0^t\int_D A(s, x, \nabla_x u^N)\cdot \nabla_x T_{k}(u^N)\diff x\diff s\\
    &\leq \int_{D}G_{k}(u_0)\diff x + \int_0^t\langle T_{k}(u^N), h^N\diff W(s)\rangle_2 + \int_0^t\sum_{j = 1}^\infty \int_D T'_{k}(u^N)\,|h^N(e_j)|^2 \diff s \\
    &\leq \frac{1}{2}\int_D u_0^2\diff x + \int_0^t\langle T_{k}(u^N), h^N\diff W(s)\rangle_2 + \int_0^t\sum_{j = 1}^\infty \int_D |h^N(e_j)|^2 \diff s.
\end{split}
\end{equation}
Similarly as under \eqref{eq:ito_energy_u}
\begin{align*}
            &\mathbb{E}\left(\sup_{t\in[0, T]}\left|\int_0^t \langle T_k(u^N) ,\,h^N \diff W(s)\rangle_{L^2(D)}\right|\right)\\
            &\leq C\,\mathbb{E}\left(\sqrt{\int_0^T \|h^N(s)^* T_k(u^N(s))\|_{L^2(D)}^2\diff s}\right)\\
            &\leq C\, \mathbb{E}\left(\sqrt{\int_0^T\|T_k(u^N(s))  \|^2_{L^2(D)}\|h^N(s)\|^2_{HS(L^2(D))}\diff s}\right)\\
            &\leq \frac{1}{4}\mathbb{E}\left(\sup_{t\in [0, T]}\|T_k(u^N(t))  \|^2_{L^2(D)}\right) + C\mathbb{E}\int_0^T\|h^N\|^2_{HS(L^2(D))}\diff t\\
            &\leq \frac{1}{4}\mathbb{E}\left(\sup_{t\in [0, T]}\|u^N(t)\|^2_{L^2(D)}\right) + C\mathbb{E}\int_0^T\|h^N\|^2_{HS(L^2(D))}\diff t,
\end{align*}
which together with \eqref{1} implies
\begin{align*}
    \mathbb{E}\left(\sup_{t\in[0, T]}\left|\int_0^t \langle T_k(u^N) ,\,h^N \diff W(s)\rangle_{L^2(D)}\right|\right) \lesssim \mathbb{E}\int_0^T\|h^N\|^2_{HS(L^2(D))}\diff t.
\end{align*}
Plugging the above into \eqref{2} we get
\begin{equation*}
    \begin{split}
        \mathbb{E}\int_0^t\int_D A(s, x, \nabla_x T_k(u^N))\cdot \nabla_x T_{k}(u^N)\diff x\diff s \leq \frac{1}{2}\mathbb{E}\int_D u_0^2\diff x + C\mathbb{E}\int_0^T\|h^N\|^2_{HS(L^2(D))}\diff t.
    \end{split}
\end{equation*}
Hence, by the coercivity condition \ref{ass:on_A_2} the sequences $\{M(t, x, \nabla_x T_k(u^N)\}_{N\in \N}$\\
and $\{M^*(t, x, A(t, x, \nabla_x T_k(u^N)))\}_{N\in\N}$ are bounded in $L^1(\Omega\times Q_T)$. Since the bound is independent of $k$, by the continuity with respect to the third variable of both $M$ and $A$, we can use Lebesgue's dominated convergence theorem to conclude that $\{\nabla_x u^N\}_{N\in\N}$ and $\{A(t, x, \nabla_x u^N\}_{N\in \N}$ are bounded in $L_M(\Omega\times Q_T)$ and $L_{M^*}(\Omega\times Q_T)$ respectively. At this point we may simply repeat the reasoning given in the previous subsection to conclude that
\begin{align*}
    \nabla_x u^N &\wstar \nabla_x u &&\text{ weakly* in }L_M(\Omega\times Q_T),\\
     A(t, x, \nabla_x u^N) &\wstar \chi &&\text{ weakly* in }L_{M^*}(\Omega\times Q_T),
\end{align*}
and $\chi = A(t, x, \nabla_x u)$ a.e. in $\Omega\times Q_T$. Converging in the weak formulation for $u^N$ gives us our desired result. Moving forward, to show the stated energy equality \eqref{eq:final_energy_equality} we turn to Lemma \ref{lem:trunc_energy_eq}, which gives us the equality
\begin{equation*}
    \begin{split}
        &\mathbb{E}\int_{D}G_{k}(u(t))\diff x - \mathbb{E}\int_{D}G_{k}(u_0)\diff x\\
        &\qquad = -\mathbb{E}\int_0^t\int_D \alpha(s, x)\cdot \nabla_x T_{k}(u)\diff x\diff s + \mathbb{E}\int_0^t\sum_{j = 1}^\infty \int_D T'_{k}(u)\,|h(e_j)|^2 \diff x \diff s.
    \end{split}
\end{equation*}
for almost every $t\in (0, T)$. Here, every term can be easily treated by the Lebesgue's dominated convergence theorem. Indeed,
\begin{align*}
    G_k(u(t)) \nearrow \frac{1}{2}(u(t))^2, \quad T'_k(u) \nearrow 1,\text{ almost everywhere as }k\to +\infty,
\end{align*}
and we get our needed \eqref{eq:final_energy_equality}. Since the left-hand side is well-defined for all $t\in[0,T]$, we may consider a representative of $\Vert u(t)\Vert_{L^2(D)}$ such that \eqref{eq:final_energy_equality} holds for all $t\in [0,T]$. Note that the multiplicative case below is done via the fixed point argument, therefore the energy equality will translate to that case as well. At last, we show that $u\in C([0, T]; L^2(\Omega\times D))$. Fix $t\in [0, T]$. From \eqref{eq:final_energy_equality} for any $t_k\to t$
\begin{multline*}
    \frac{1}{2}\mathbb{E}\|u(t_k)\|_{L^2(D)}^2 - \frac{1}{2}\mathbb{E}\|u(t)\|_{L^2(D)}^2 = -\mathbb{E}\int_t^{t_k}\int_D A(s, x, \nabla_x u)\cdot \nabla_x u\diff x\diff s + \mathbb{E}\int_t^{t_k}\|h\|_{HS(L^2(D))}^2\diff s\to 0,
\end{multline*}
which implies
$$
\|u(t_k)\|_{L^2(\Omega\times D)}\rightarrow \|u(t)\|_{L^2(\Omega\times D)},\quad t_k\to t.
$$
Hence to conclude the continuity we shall show that $u$ is weakly continuous as a function taking values in $L^2(\Omega\times D)$. As weak convergence combined with a convergence of norms implies strong convergence in $L^2$, this will grant us our thesis. Fix a function $\psi\in L^2(\Omega)$ and $\varphi\in C^\infty_c(D)$. Then, test \eqref{eq:weak_form_final} by $\psi\varphi$ to get
\begin{align*}
    &\mathbb{E}\int_D u(t)\psi\varphi\diff x - \mathbb{E}\int_D u_0\psi\varphi\diff x =\\
    &\qquad -\mathbb{E}\int_0^t\int_D A(\tau, x, \nabla_x u)\cdot(\psi\nabla_x\varphi)\diff x\diff\tau + \mathbb{E}\int_D\psi\varphi\,\int_0^t h\diff W(\tau)\diff x.
\end{align*}
Since, the right-hand side is continuous in time, it follows that 
\begin{align}\label{weak_continuity_for_dense_set}
\mathbb{E}\int_D u(t_k)\psi\varphi\diff x\rightarrow \mathbb{E}\int_D u(t)\psi\varphi\diff x,\quad t_k\to t.
\end{align}
We will now show that the finite linear combinations of the functions of the form $\psi\,\varphi$, $\psi\in L^2(\Omega)$, $\varphi\in C_c^\infty(D)$ are dense in $L^2(\Omega\times D)$. Although the argument is quite standard, let us write it for the convenience of the reader. Fix $\phi\in L^2(\Omega\times D)$. Then, there exists a sequence of sets $\{S_k\}_{k\in\N}\subset \Omega\times D$, and a sequence of numbers $\{a_k\}_{k\in\N}$ such that
\begin{align}\label{simple_to_phi}
\sum_{k = 1}^n a_k\mathbf{1}_{S_k} \rightarrow \phi,\text{ strongly in }L^2(\Omega\times D),\text{ as }n\to +\infty.
\end{align}
Moreover, by the definition of the product measure on $\Omega\times D$, for any $k\in\N$ there exists a sequence of families of sets $\{A_k^{j, i}\times B_k^{j, i}\}_{j\in\N, i\in\N}\subset \Omega\times D$ such that for any $k, i\in \N$ we have $S_k\subset \bigcup_{j\in \N}(A_k^{j, i}\times B_k^{j, i})$, and
\begin{align}\label{double_simple_to_simple}
\sum_{j = 1}^\infty\mathbf{1}_{A_k^{j, i}}\mathbf{1}_{B_k^{j, i}} \rightarrow \mathbf{1}_{S_k},\text{ strongly in }L^2(\Omega\times D), \text{ as }i\to +\infty.
\end{align}
Now fix $\varepsilon > 0$. Since $\mathbf{1}_{B_k^{j, i}}\in L^2(D)$, we can find a function $\varphi_k^{j, i}\in C_c^\infty(D)$ such that
\begin{align*}
    \|\varphi_{k}^{j, i} - \mathbf{1}_{B_k^{j, i}}\|_{L^2(D)} \leq \frac{1}{|a_k| \sqrt{|A_k^{j, i}|}}\frac{\varepsilon}{4\cdot 2^k\,2^j}.
\end{align*}
Moreover, for any $k, i\in \N$, we can find $l_{k, i}\in \N$ such that
$$
 \left\|\sum_{j = 1}^{l_{k, i}} a_k\mathbf{1}_{A_k^{j, i}}\varphi_k^{j, i} - \sum_{j = 1}^\infty a_k\mathbf{1}_{A_k^{j, i}}\varphi_k^{j, i}\right\|_{L^2(\Omega\times D)} \leq \frac{\varepsilon}{4\cdot 2^k},
$$
and by \eqref{double_simple_to_simple} for any $k\in \N$, we can find $i_k\in \N$ big enough so that
$$
\left\|\sum_{j = 1}^\infty \mathbf{1}_{A_k^{j, i_k}}\mathbf{1}_{B_k^{j, i_k}} - \mathbf{1}_{S_k}\right\|_{L^2(\Omega\times D)} \leq \frac{1}{|a_k|}\frac{\varepsilon}{4\cdot 2^k}.
$$
At last, by \eqref{simple_to_phi} we can find $n\in \N$ large enough for
$$
\left\|\sum_{k = 1}^n a_k\mathbf{1}_{S_k} - \phi\right\|_{L^2(\Omega\times D)} \leq \frac{\varepsilon}{4}
$$
to hold. Thus, we may deduce
\begin{equation*}
    \begin{split}
        &\left\|\sum_{k = 1}^n\sum_{j = 1}^{l_{k, i_k}} a_k\mathbf{1}_{A_k^{j, i_k}}\varphi_k^{j, i_k} - \phi\right\|_{L^2(\Omega\times D)}\\
        &\leq \left\|\sum_{k = 1}^n\sum_{j = 1}^{l_{k, i_k}} a_k\mathbf{1}_{A_k^{j, i_k}}\varphi_k^{j, i_k} - \sum_{k = 1}^n\sum_{j = 1}^\infty a_k\mathbf{1}_{A_k^{j, i_k}}\varphi_k^{j, i_k}\right\|_{L^2(\Omega\times D)}\\
        &\qquad + \left\|\sum_{k = 1}^n\sum_{j = 1}^\infty a_k\mathbf{1}_{A_k^{j, i_k}}\varphi_k^{j, i_k} - \sum_{k = 1}^n \sum_{j = 1}^\infty a_k\mathbf{1}_{A_k^{j, i_k}}\mathbf{1}_{B_k^{j, i_k}}\right\|_{L^2(\Omega\times D)}\\
        &\qquad + \left\|\sum_{k = 1}^n\sum_{j = 1}^\infty a_k\mathbf{1}_{A_k^{j, i_k}}\mathbf{1}_{B_k^{j, i_k}} - \sum_{k = 1}^n a_k\mathbf{1}_{S_k} \right\|_{L^2(\Omega\times D)}\\
        &\qquad + \left\|\sum_{k = 1}^n a_k\mathbf{1}_{S_k} - \phi\right\|_{L^2(\Omega\times D)}\\
        &\leq \sum_{k = 1}^n\left\|\sum_{j = 1}^{l_{k, i_k}} a_k\mathbf{1}_{A_k^{j, i_k}}\varphi_k^{j, i_k} - \sum_{j = 1}^\infty a_k\mathbf{1}_{A_k^{j, i_k}}\varphi_k^{j, i_k}\right\|_{L^2(\Omega\times D)}\\
        &\qquad + \sum_{k = 1}^n\sum_{j = 1}^\infty\left\|a_k\mathbf{1}_{A_k^{j, i_k}}(\varphi_k^{j, i_k} - \mathbf{1}_{B_k^{j, i_k}})\right\|_{L^2(\Omega\times D)}\\
        &\qquad + \sum_{k = 1}^n\left\| a_k\left(\sum_{j = 1}^\infty\mathbf{1}_{A_k^{j, i_k}}\mathbf{1}_{B_k^{j, i_k}} - \mathbf{1}_{S_k}\right) \right\|_{L^2(\Omega\times D)}\\
        &\qquad + \left\|\sum_{k = 1}^n a_k\mathbf{1}_{S_k} - \phi\right\|_{L^2(\Omega\times D)}\\
        &\leq \sum_{k = 1}^n\frac{\varepsilon}{4\cdot 2^k} + \sum_{k = 1}^n\sum_{j = 1}^\infty |a_k|\sqrt{|A_k^{j, i_k}|}\frac{1}{|a_k|\sqrt{|A_k^{j, i_k}|}}\frac{\varepsilon}{4\cdot 2^k\, 2^j} + \sum_{k = 1}^n |a_k|\frac{1}{|a_k|}\frac{\varepsilon}{4\cdot 2^k} + \frac{\varepsilon}{4}\\
        &< \varepsilon
    \end{split}
\end{equation*}
For simplicity let
$$
\phi_n := \sum_{k = 1}^n\sum_{j = 1}^{l_{k, i_k}} a_k\mathbf{1}_{A_k^{j, i_k}}\varphi_k^{j, i_k}.
$$
Using \eqref{weak_continuity_for_dense_set} we see that for any $n\in \N$ it is the case that
\begin{align*}
\mathbb{E}\int_D u(t_k)\phi_n\diff x\rightarrow \mathbb{E}\int_D u(t)\phi_n\diff x,\quad t_k\to t.
\end{align*}
Hence, using H\"{o}lder's inequality and the energy equality \eqref{eq:final_energy_equality}
\begin{equation*}
    \begin{split}
        &\limsup_{t_k\to t}\left|\mathbb{E}\int_D u(t_k)\phi\diff x - \mathbb{E}\int_D u(t)\phi\diff x\right|\\
        &\leq \limsup_{t_k\to t}\left|\mathbb{E}\int_D u(t_k)\phi\diff x - \mathbb{E}\int_D u(t_k)\phi_n\diff x\right|\\
        &\qquad + \limsup_{t_k\to t}\left|\mathbb{E}\int_D u(t_k)\phi_n\diff x - \mathbb{E}\int_D u(t)\phi_n\diff x\right|\\
        &\qquad + \limsup_{t_k\to t}\left|\mathbb{E}\int_D u(t)\phi_n\diff x - \mathbb{E}\int_D u(t)\phi\diff x\right|\\
        &\leq 2\sup_{t\in [0, T]}\|u(t)\|_{L^2(\Omega\times D)}\|\phi_n - \phi\|_{L^2(\Omega\times D)}.
    \end{split}
\end{equation*}
Converging with $n\to +\infty$ gives us our thesis.

\section{Multiplicative case}\label{sec:multiplicative}

We move to extending our result to the multiplicative case, that is when $h$ can depend on the solution $u$ as given in the Introduction. To this end we shall employ the classical Banach fixed point theorem. Define the operator
$$
\Psi: L^2(0, T; L^2(\Omega\times D), \diff e^{-\alpha t}) \rightarrow L^2(0, T; L^2(\Omega\times D), \diff e^{-\alpha t}),
$$
where $e^{-\alpha t}$ is a weight in time for some soon-to-be-fixed $\alpha > 0$, and $\Psi(S) = u_S$ maps a given $S$ to a solution $u_s$ of \eqref{eq:main_sys} with $h(\cdot, S)$. Note that as the time interval is fixed, by our previous considerations $\Psi$ is well-defined. Similarly to \eqref{1} we can obtain
\begin{align*}
    \mathbb{E}(\|(u^{N}_{S_1} - u^{N}_{S_2})(t)\|^2_{L^2(D)}) \lesssim \int_0^t\|h^{N}(S_1) - h^{N}(S_2)\|^2_{}\diff \tau.
\end{align*}
Thus, after converging with $N\to +\infty$, due to \eqref{conv:strong_conv_in_N_u_N}
\begin{align*}
    \mathbb{E}(\|(u_{S_1} - u_{S_2})(t)\|^2_{L^2(D)}) \leq C\int_0^t\|h(S_1) - h(S_2)\|^2_{HS(L^2(D))}\diff \tau.
\end{align*}
Due to \ref{ass:on_h_2} in Assumption \ref{ass:func_h_multiplicative} we can estimate the right-hand side
$$
\int_0^t\|h(S_1) - h(S_2)\|^2_{HS(L^2(D))}\diff \tau \lesssim \|S_1 - S_2\|_{L^2(\Omega\times Q_t)}.
$$
Hence,
\begin{align}\label{3}
    \mathbb{E}(\|(u_{S_1} - u_{S_2})(t)\|^2_{L^2(D)}) \leq C\|S_1 - S_2\|_{L^2(\Omega\times Q_t)}^2.
\end{align}
By multiplying \eqref{3} by $e^{-\alpha t}$, and integrating we get
\begin{align*}
    \int_0^T\mathbb{E}(\|(u_{S_1} - u_{S_2})(t)\|^2_{L^2(D)})e^{-\alpha t}\diff t \leq C\int_0^T e^{-\alpha t}\int_0^t\mathbb{E}\|S_1 - S_2\|^2_{L^2(D)}\diff\tau\diff t.
\end{align*}
Integrating by parts
\begin{align*}
    C\int_0^T e^{-\alpha t}\int_0^t\mathbb{E}\|S_1 - S_2\|^2_{L^2(D)}\diff\tau\diff t = \frac{C}{\alpha}(1 - e^{-\alpha T})\int_0^T\mathbb{E}\|S_1 - S_2\|^2_{L^2(D)} e^{-\alpha t}\diff t.
\end{align*}
Thus,
\begin{align*}
    \int_0^T\mathbb{E}(\|(u_{S_1} - u_{S_2})(t)\|^2_{L^2(D)})e^{-\alpha t}\diff t \leq \frac{C}{\alpha}(1 - e^{-\alpha T})\int_0^T\mathbb{E}\|S_1 - S_2\|^2_{L^2(D)} e^{-\alpha t}\diff t.
\end{align*}
Taking $\alpha > 0$ big enough, so that $C < \alpha$, we may use Banach fixed point theorem to conclude.

\appendix

\section{It\^{o}'s formula in Orlicz spaces}

In the following section we show how to obtain It\^{o}'s formula in the setting of Orlicz spaces. We will closely follow the book \cite{Liu2015Stochastic} and the proof given there. Let us explain some of the minor differences that appear and how we combat them. First, as mentioned in Section \ref{sec:musielaki}, there is a problem of duality when converging in the product with the approximations. It is quite hard to obtain a pair of sequences, one of which converges strongly, and the other one weakly, to conclude to convergence of their product. Instead, in the Orlicz setting one focuses on the modular convergence, which by Proposition \ref{prop:conv_of_product_in_musielaki} has the needed property. This can be seen in Lemma \ref{lem:approx_by_floor_func}, where we show modular convergence of the floor functions, instead of their strong convergence as in \cite{Liu2015Stochastic}. Second, there is no Gelfand triple structure
$$
V\hookrightarrow H \hookrightarrow V^*,
$$
in our case, and no time-space factorization of our function spaces (that is we cannot write $L_m(Q_T)$ as a Bochner space in time, taking values in $V$). Here, there is no special remedy, but instead we take advantage of the fact, that we work with particular spaces, instead of general Banach space $V$, and hence we do not really need the Gelfand triple structure to understand the product $\DIV_x\alpha \, u$ in the equation \eqref{eq:weak_form_in_ito_lemma_proof}.

\begin{lem}\label{lem:approx_by_floor_func}
    Suppose that $u\in L^2(\Omega\times Q_T)\cap L^1(\Omega\times (0, T); W^{1,1}_0(D))\cap L_m(\Omega\times Q_T)$ is such that $\nabla_x u\in L_m(\Omega\times Q_T)$. Then, there exists a sequence of partitions $I_l := \{0 = t_0^l < t_1^l < ... < t_{k_l}^l = T\}$ such that $I_l\subset I_{l+1}$ and $\delta(I_l) := \max_i(t_i^l - t^l_{i-1})\to 0$ as $l\to +\infty$, and for 
    $$
    \tilde{u}^l := \sum_{i=2}^{k_l}\mathbf{1}_{[t_{i-1}^l, t_i^l)}u(t_{i-1}^l), \quad \hat{u}^l := \sum_{i=1}^{k_l-1}\mathbf{1}_{[t_{i-1}^l, t_i^l)}u(t_i^l),\quad l\in\N
    $$
    it holds that
    $$
    \nabla_x\tilde{u}^l \xrightarrow{m} \nabla_xu,\quad \nabla_x\hat{u}^l\xrightarrow{m} \nabla_x u \quad \text{ modularly in }L_m(\Omega\times Q_T),
    $$
    $$
    \tilde{u}^l \xrightarrow{m} u,\quad \hat{u}^l\xrightarrow{m}  u\quad \text{ modularly in }L_m(\Omega\times Q_T)
    $$
    as $l\to+\infty$ and
   $$
    \|\tilde{u}^l - u\|_{L^2(\Omega \times   Q_T)} + \|\hat{u}^l - u\|_{L^2(\Omega \times   Q_T)} \rightarrow 0
    \quad \mbox{ as }l\to+\infty .$$

\end{lem}
\begin{proof}
    The proof of this Lemma is extremely similar to \cite[Lemma 4.2.6]{Liu2015Stochastic}, therefore we shall only discuss the differences for the setting in Orlicz spaces. For simplicity let us assume that $T = 1$. Extend $\nabla_x u$ to $\Omega\times\mathbb{R}\times D$ by $0$. There exists an $\Omega'\in\mathcal{F}$ with full probability such that for every $\omega\in \Omega'$, we can find by \cite[Theorem 3.7.8]{chlebicka2021book} a sequence of $\{f_n\}_{n\in\N}\subset C_c^\infty(\mathbb{R}\times D)$ such that
    $$
        \int_{\mathbb{R}}\int_D m\left(\frac{\nabla_xf_n - \nabla_x u}{\lambda}\right)\diff x\diff t \leq \frac{1}{2n}
        \quad \mbox{ for some }\lambda > 0.
    $$
     Thus, for every $n\in\N$
    \begin{align*}
        &\limsup_{\delta\to 0}\int_{\mathbb{R}}\int_D m\left(\frac{\nabla_x u(\omega,\delta + s, x) - \nabla_x u(\omega, s, x)}{1 + 2\lambda}\right)\diff x\diff s \\
        &\leq\frac{\lambda}{1 + 2\lambda}\limsup_{\delta\to 0}\Bigg(\int_{\mathbb{R}}\int_D m\left(\frac{\nabla_x u(\omega, \delta+s, x) - \nabla_x f_n(\delta + s, x)}{\lambda}\right)\diff x\diff s\\
        &\qquad\quad+ \int_{\R}\int_{D} m\left(\frac{\nabla_xf_n(s, x) - \nabla_x u(\omega, s, x)}{\lambda}\right)\diff x\diff s\Bigg)\\
        &\leq \frac{\lambda}{(1 + 2\lambda)n},
    \end{align*}
    where we have used that by the uniform continuity of $f_n$ and the continuity of $m$ in $0$
    $$
    \lim_{\delta\to 0}\int_{\R}\int_D m\left(\nabla_xf_n(\delta + s, x) - \nabla_x f_n(s, x)\right)\diff x\diff s = 0.
    $$
    Letting $n\to +\infty$ we obtain
    \begin{align}\label{shift_cont_in_Orlicz}
    \limsup_{\delta\to 0}\int_{\mathbb{R}}\int_D m\left(\frac{\nabla_x u(\omega,\delta + s, x) - \nabla_x u(\omega, s, x)}{1 + 2\lambda}\right)\diff x\diff s = 0, \quad \omega\in \Omega'.
    \end{align}
    Now let $\gamma_n(t) = 2^{-n}\lfloor2^n t\rfloor$. Shifting the integral in \eqref{shift_cont_in_Orlicz} by $t$ and setting $\delta := \gamma_n(t) - t$ we can deduce
    $$
    \limsup_{n\to 0}\int_{\mathbb{R}}\int_D m\left(\frac{\nabla_x u(\omega,\gamma_n(t) + s, x) - \nabla_x u(\omega, t + s, x)}{1 + 2\lambda}\right)\diff x\diff s = 0, \quad \omega\in\Omega'.
    $$
    Since our extension of $\nabla_x u$ is by zero, we can estimate
    \begin{align*}
        &\int_0^1\int_D m\left(\frac{\nabla_x u(\omega,\gamma_n(t) + s, x) - \nabla_x u(\omega, t + s, x)}{1 + 2\lambda}\right)\diff x\diff s\\
        &\phantom{=} \leq \mathbf{1}_{[-2,2]}(t)\frac{1}{2}\left(\int_{\R}\int_D m\left(\frac{2\nabla_x u(\omega,\gamma_n(t) + s, x)}{1 + 2\lambda}\right)\diff x\diff s + \int_\R\int_D m\left(\frac{2\nabla_x u(\omega, t + s, x)}{1 + 2\lambda}\right)\diff x\diff s\right)\\
        &\phantom{=}=\mathbf{1}_{[-2,2]}(t)\int_{\R}\int_D m\left(\frac{2\nabla_x u(\omega, s, x)}{1 + 2\lambda}\right)\diff x\diff s,\quad \omega\in\Omega'.
    \end{align*}
    Hence, by Lebesgue's dominated convergence theorem, we get
    \begin{align*}
        0 &= \lim_{n\to +\infty}\mathbb{E}\int_{\mathbb{R}}\int_0^1\int_D m\left(\frac{\nabla_x u(\omega,\gamma_n(t) + s, x) - \nabla_x u(\omega, t + s, x)}{1 + 2\lambda}\right)\diff x\diff s\diff t\\
        &\geq  \lim_{n\to +\infty}\mathbb{E}\int_0^1\int_0^1\int_D m\left(\frac{\nabla_x u(\omega,\gamma_n(t-s) + s, x) - \nabla_x u(\omega, t, x)}{1 + 2\lambda}\right)\diff x\diff t\diff s.
    \end{align*}
    As the rest of the proof follows exactly as in \cite[Lemma 4.2.6]{Liu2015Stochastic} by fixing
    $$
    t_0^n(s) :=0,\, t_i^n(s) := \left(s - \frac{\lfloor2^n s\rfloor}{2^n}\right) + \frac{i-1}{2^n},\,1\leq i\leq 2^n,\, t_{2^n+1}^n := 1,
    $$
    we skip it. Similarly, we show the rest of the convergences.
\end{proof}

\begin{thm}\label{thm:orlicz_ito_formula}
    Let $u_0\in L^2((\Omega, \mathcal{F}_0, P)\times D)$, and $\alpha\in L_{m^*}(\Omega\times Q_T)$, $h\in \mathcal{N}_W^2(0, T; L^2(D))$.
    Define
    \begin{align}\label{eq:weak_form_in_ito_lemma_proof}
    u(t) := u_0 + \int_0^t\DIV_x\alpha\diff s + \int_0^t h(s)\diff W(s),\quad \text{ in }\mathcal{D}'(D),\quad t\in [0, T]
    \end{align}
    If $u\in L^2(\Omega\times Q_T)\cap L^1(\Omega\times (0,T); W^{1,1}_0(D))$ with $\nabla_x u\in L_m(\Omega\times Q_T)$, then $u\in L^2(\Omega;C([0, T]; L^2(D)))$ is $\mathcal{F}_t$ adapted with values in $L^2(D)$ and the following It\^{o} formula holds
    \begin{equation}\label{eq:ito_formula_orlicz}
    \begin{split}
        \|u(t)\|^2_{L^2(D)} &= \|u_0\|_{L^2(D)} + 2\int_0^t\int_D\alpha(s,x)\cdot\nabla_x u(s,x)\diff x\diff s\\
        &\phantom{=}+ \int_0^t\|h\|^2_{HS(L^2(D))}\diff s + 2\int_0^t\langle u(s), h(s)\diff W(s)\rangle_2,\quad\text{ for all }t\in [0, T], \text{ a.s. in }\Omega. 
    \end{split}
    \end{equation}
\end{thm}

\begin{proof}
    Since the proof of this theorem is extremely similar to the proof of the \cite[Theorem 4.2.5]{Liu2015Stochastic}, we will only discuss the differences in the setting of Orlicz spaces. By Lemma \ref{lem:approx_by_floor_func} we can fix $\hat{u}^l$ and $\tilde{u}^l$ approximating $u$ modularly. Here, we also introduce the notation
    $$
    \mathbb{M}(t) := \int_0^t h(s)\diff W(s),
    $$
    which we will use below.
    \begin{enumerate}[label = Claim (\alph*):]
        \item \label{eq:basic_eq_in_ito_lemma_proof} \begin{align*}
            \int_D u^2(t)\diff x &= \int_D u^2(s)\diff x + 2\int_s^t\int_D \alpha (r, x)\cdot\nabla_x u (t,x)\diff x\diff r\\
            &\phantom{=}+ 2\int_D u (s,x)\,(\mathbb{M}(t) - \mathbb{M}(s))\diff x\\
            &\phantom{=}+ \|\mathbb{M}(t) - \mathbb{M}(s)\|^2_{L^2(D)} - \|u(t) - u(s) - (\mathbb{M}(t) - \mathbb{M}(s))\|_{L^2(D)}^2,
        \end{align*}
        holds for almost every $t, s\in [0, T]$, $t > s$.\\
        Our first goal is to show that \eqref{eq:weak_form_in_ito_lemma_proof} can be tested by the solution. To do so, we closely follow \cite[Theorem 3.7.8, Lemma 3.7.9]{chlebicka2021book}. Let 
        $$
            D_j = D\cap G_j,
        $$
        be a split of the Lipschitz domain $D$ into a finite number of star-shaped sets with respect to the balls $B_{R_j}$ (that is $D_j$ is star-shaped with respect to every point in $B_{R_j}$). Define the partition of unity
        $$
            0\leq\theta_j\leq 1,\quad \theta_j\in C_c^\infty(G_j),\quad\sum_{j\in J}\theta_j = 1\text{, for }x\in D.
        $$
        Moreover, for any measurable function $\xi: \R^d \rightarrow\R^d$ with $\supp\,\xi\subset D_j$ we set
        $$
            \xi_{j,\delta}(x) = \int_{D_j}\rho_\delta(x-y)\xi\left(\frac{y}{\kappa^j_\delta}\right)\diff y,
        $$
        where $\rho_\delta$ is the typical mollification kernel, and $\kappa^j_\delta := 1 - \frac{2\delta}{R_j}$. At last, note that since $\nabla_x u\in L_m(\Omega\times Q_T)$ it follows that for almost every $t\in [0, T]$ and $\omega\in \Omega$, $\nabla_xu(t)\in L_m(D)$. Then, applying directly Theorem 3.7.8, Lemma 3.7.9 in \cite{chlebicka2021book} we can say
        \begin{align}
            (u(t))_\delta := \sum_{j\in J}(\theta_j\, u(t))_{j, \delta} &\rightarrow u(t) &&\text{strongly in }L^1(D),\\
            \nabla_x(u(t))_\delta := \sum_{j\in J}\nabla_x(\theta_j\, u(t))_{j, \delta} &\xrightarrow{m} \nabla_x u(t) &&\text{modularly in }L_m(D).\label{conv:modular_conv_ito_lemma_proof_delta}
        \end{align}
        In our case it is possible to upgrade the strong convergence in $L^1$ to strong convergence in $L^2$. Indeed, note that since for any $j\in J$, $|\theta_j u(t)|^2\in L^1(D)$, it is equiintegrable and by the de la Vall\'{e}e--Poussin criterion, there exists a superlinear function $\Phi$ for which
        $$
            \int_{D}\Phi(|\theta_j u(t)|^2)\diff x < +\infty.
        $$
        By a simple change of variables, the above implies
        \begin{align*}
            &\sup_{\delta > 0}\int_D\Phi\left(\theta_j\left(\frac{x}{\kappa^j_\delta}\right) u\left(t,\frac{x}{\kappa^j_\delta}\right)\right)\diff x\\
            &\leq \sup_{\delta > 0}\int_D\frac{1}{(\kappa^j_\delta)^d}\Phi\left(\theta_j\left(\frac{x}{\kappa^j_\delta}\right) u\left(t,\frac{x}{\kappa^j_\delta}\right)\right)\diff x \leq \int_{D}\Phi(|\theta_j u(t)|^2)\diff x < +\infty.
        \end{align*}
        Hence, again by the de la Vall\'{e}e--Poussin criterion, and Vitali's convergence theorem (remember that $\kappa^j_\delta\to 1$, as $\delta\to 0$)
        $$
            \theta_j\left(\frac{x}{\kappa^j_\delta}\right) u\left(t,\frac{x}{\kappa^j_\delta}\right) \rightarrow \theta_j u(t)\text{ strongly in }L^2(D),\text{ as }\delta\to 0,
        $$
        which readily implies
        \begin{align}
            \varphi_\delta := \sum_{j\in J}\theta_j\left(\frac{x}{\kappa^j_\delta}\right) u\left(t,\frac{x}{\kappa^j_\delta}\right) \rightarrow u(t)\text{ strongly in }L^2(D),\text{ as }\delta\to 0.\label{conv:some_conv_1}
        \end{align}
        The strong convergence in $L^2(D)$ implies by Kolmogorov--Riesz--Fr\'{e}chet's theorem
        \begin{align*}
            \lim_{|h|\to 0^+}\sup_{\delta > 0}\|\varphi_\delta(\cdot + h) - \varphi_\delta(\cdot)\|_{L^2(D)} = 0.
        \end{align*}
        Thus, we may estimate using Jensen's inequality
        \begin{equation*}
        \begin{split}
            &\|(u(t))_\delta  - \varphi_\delta\|^2_{L^2(D)} = \|\rho_\delta\star\varphi_\delta  - \varphi_\delta\|^2_{L^2(D)} = \int_D\left|\int_{B(0, \delta)}\rho_\delta(y)(\varphi_\delta(x - y) - \varphi_\delta(x))\diff y\right|^2\diff x\\
            &\leq \int_D\int_{B(0, \delta)}\rho_\delta(y)|\varphi_\delta(x - y) - \varphi_\delta(x)|^2\diff y\diff x = \int_{B(0, \delta)}\rho_\delta(y)\|\varphi_\delta(\cdot - y) - \varphi_\delta(\cdot)\|^2_{L^2(D)}\diff y\\
            &\leq \sup_{|y| <\delta}\sup_{\delta > 0}\|\varphi_\delta(\cdot - y) - \varphi_\delta(\cdot)\|^2_{L^2(D)} \to 0, \text{ as }\delta\to 0.
        \end{split}
        \end{equation*}
        At last, by triangle inequality, the above, and \eqref{conv:some_conv_1}
        \begin{align}
            \|(u(t))_\delta - u(t)\|_{L^2(D)}\leq \|(u(t))_\delta  - \varphi_\delta\|_{L^2(D)} + \|\varphi_\delta - u(t)\|_{L^2(D)} \rightarrow 0,\text{ as }\delta\to 0.\label{conv:some_conv_2}
        \end{align}
        Now, since $(u(t))_\delta\in C_c^\infty(D)$, we can test \eqref{eq:weak_form_in_ito_lemma_proof} evaluated at $\tau\in [0, T]$ by it to obtain
        \begin{align*}
            &\int_D u(\tau)(u(t))_\delta\diff x = \int_D u_0(u(t))_\delta\diff x\\
            &\qquad- \int_0^\tau\int_D \alpha(s, x)\cdot\nabla_x(u(t))_\delta\diff x\diff s + \int_D \mathbb{M}(\tau)\,(u(t))_\delta\diff x.
        \end{align*}
        By \eqref{conv:some_conv_2} it follows
        \begin{align*}
            &\int_D u(\tau)(u(t))_\delta\diff x\rightarrow \int_D u(\tau)\,u(t)\diff x,\\
            &\int_D u_0(u(t))_\delta\diff x \rightarrow \int_D u_0\,u(t)\diff x,\\
            &\int_D \mathbb{M}(\tau)\,(u(t))_\delta\diff x \rightarrow \int_D \mathbb{M}(\tau)\,u(t)\diff x,
        \end{align*}
        and by \eqref{conv:modular_conv_ito_lemma_proof_delta}
        \begin{align*}
            \int_0^\tau\int_D \alpha(s, x)\cdot\nabla_x(u(t))_\delta\diff x\diff s \rightarrow \int_0^\tau\int_D \alpha(s, x)\cdot\nabla_x u(t)\diff x\diff s.
        \end{align*}
        Thus, we may indeed test \eqref{eq:weak_form_in_ito_lemma_proof} by the solution, whenever it makes sense. To end the proof, we subtract \eqref{eq:weak_form_in_ito_lemma_proof} evaluated in $s$ from \eqref{eq:weak_form_in_ito_lemma_proof} evaluated in $t$ and test the result by $u(t)$ and use the binomial identities to deduce
        \begin{align*}
            &\|\mathbb{M}(t) - \mathbb{M}(s)\|^2_{L^2(D)} - \|u(t) - u(s) - (\mathbb{M}(t) - \mathbb{M}(s))\|_{L^2(D)}^2\\
            &\phantom{=}+2\int_D u (s)\,(\mathbb{M}(t) - \mathbb{M}(s))\diff x\\
            &= 2\int_D u (t)\,(\mathbb{M}(t) - \mathbb{M}(s)) \diff x - \int_D|u (t) - u (s)|^2\diff x\\
            &= 2\int_D u (t)\,(u(t) - u(s)) \diff x - 2\int_s^t\int_D\alpha (r, x)\cdot\nabla_x u (t,x)\diff x\diff r\\
            &\phantom{=} - \int_D u^2(t) \diff x - \int_D u^2(s) \diff x + 2\int_D u (t)\,u (s)\, \diff x\\
            &= \int_D u^2(t) \diff x - \int_D u^2(s) \diff x- 2\int_s^t\int_D\alpha (r, x)\cdot\nabla_x u (t,x)\diff x\diff r.
        \end{align*}
        \item We have
        \begin{align*}
            \mathbb{E}\left(\sup_{t\in[0,T]}\|u(t)\|_{L^2(D)}^2\right) < +\infty .
        \end{align*}
        To prove it we fix $I_l$, $l\in\N$ be a sequence of partitions from Lemma \ref{lem:approx_by_floor_func}. Then, by our assumption $\mathbb{E}(\|u(t)\|_{L^2(D)}^2)< +\infty$ for a.e. $t\in [0, T]$, we may choose $I_l$ such that
        \begin{align}\label{decrete_bound_u_norm_ito_lemma}
        \mathbb{E}(\|u(t)\|_{L^2(D)}^2) < +\infty\quad  \text{ for all }t\in \bigcup_{l\in\N}I_l.
        \end{align}
        Thus, by \ref{eq:basic_eq_in_ito_lemma_proof} for any $t:= t_i^l\in I_l\setminus\{0, T\}$
        \begin{equation}\label{ito_lemma_proof_some_eq_1}
        \begin{split}
            &\int_D u^2(t) \diff x - \int_D u^2_0  \diff x\\
            &=\sum_{j = 0}^{i - 1}\left(\int_D u^2(t_{j+1}^l) \diff x - \int_D u^2(t_{j}^l) \diff x \right)\\
            &=2\int_0^t\int_D\alpha (s, x)\cdot\nabla_x \hat{u}^{l}(s,x)\diff x\diff s\\
            &\phantom{=}+2\int_0^t \langle\tilde{u}^{l}(s) ,\,h (s)\diff W(s)\rangle_{L^2(D)} + 2\int_D u_0 \,\mathbb{M} (t_1^l) \diff x\\
            &\phantom{=}+\sum_{j=0}^{i-1}\|(\mathbb{M}(t^l_{j+1}) - \mathbb{M}(t^l_j))  \|^2_{L^2(D)} - \|u(t^l_{j+1}) - u(t^l_{j})  - (\mathbb{M}(t^l_{j+1}) - \mathbb{M}(t^l_j))  \|_{L^2(D)}^2.
        \end{split}
        \end{equation}
        Now, by Lemma \ref{lem:approx_by_floor_func}, and the Fenchel--Young inequality we can deduce that
        \begin{align*}
            \mathbb{E}\left(\int_0^T\int_D|\alpha (s, x)\cdot\nabla_x \hat{u}^{l}(s,x)|\diff x\diff s\right) \leq C_1,
        \end{align*}
        where $C_1$ is independent of $l$. At the same time, due to the Burkholder--Davis--Gundy inequality \cite[Theorem 3.2.47]{LR17}, denoting the adjoint operator to $h(t)$ by $h^{\ast}(t)$ 
        we get
        \begin{align*}
            &\mathbb{E}\left(\sup_{t\in[0, T]}\left|\int_0^t \langle\tilde{u}^{l}(s) ,\,h (s)\diff W(s)\rangle_{L^2(D)}\right|\right)\\
            &\leq C\,\mathbb{E}\left(\sqrt{\int_0^T \|h(s)^*\tilde{u}^{l}(s)\|_{L^2(D)}^2\diff s}\right)\\
            &\leq C\, \mathbb{E}\left(\sqrt{\int_0^T\|\tilde{u}^{l}(s)  \|^2_{L^2(D)}\|h (s)  \|^2_{HS(L^2(D))}\diff s}\right)\\
            &\leq \frac{1}{4}\mathbb{E}\left(\sup_{1\leq j\leq k_l-1}\|u (t_j^l)  \|^2_{L^2(D)}\right) + C\mathbb{E}\int_0^T\|h\|^2_{HS(L^2(D))}\diff t.
        \end{align*}
        Notice that by \eqref{decrete_bound_u_norm_ito_lemma}
        \begin{align*}
            \mathbb{E}\left(\sup_{1\leq j\leq k_l-1}\|u (t_j^l)  \|^2_{L^2(D)}\right) < +\infty.
        \end{align*}
        Finally, by \cite[Lemma 2.4.4]{Liu2015Stochastic}
        \begin{align*}
            &\mathbb{E}\left(\sum_{j=0}^{i-1}\|\mathbb{M}(t^l_{j+1}) - \mathbb{M}(t^l_j)\|^2_{L^2(D)}\right)\\
            &\phantom{=}=\sum_{j=0}^{i-1}\mathbb{E}\left(\int_{t_j^l}^{t_{j+1}^l}\|h(s)\|_{HS(L^2(D))}^2\diff s\right)\\
            &\phantom{=}=\mathbb{E}\int_0^{t_i^l}\|h\|_{HS(L^2(D))}^2\diff s.
        \end{align*}
        Combining all of the above we arrive at
        \begin{align*}
            \mathbb{E}\left(\sup_{t\in I_l\setminus\{T\}}\|u (t)  \|_{L^2(D)}^2\right)\leq C_2 
        \end{align*}
        for $C_2 >0$ independent of $l\in\mathbb{N}$. Therefore we can converge with $l\nearrow +\infty$ to obtain (after setting $I:= \bigcup_{l\geq 1}I_l\setminus\{T\}$)
        $$
        \mathbb{E}\left(\sup_{t\in I}\|u(t)\|_{L^2(D)}^2\right) \leq C_2,
        $$
        where we have used that $I_l\subset I_{l+1}$. Now fix $\{e_j \}_{j\in\N}$ an orthogonal basis of $W^{s,2}_0(D)$ (with eigenvalues $\lambda_j$), which is orthonormal in $L^2(D)$. Then, as
        $$
        \sum_{j=1}^N \frac{1}{\lambda_j}\prescript{}{W^{-s,2}(D)}{\langle} u(t),\,e_j\rangle_{W^{s,2}_0(D)}^2 \nearrow \|u(t)\|_{L^2(D)}^2 \quad \text{ as }N\to +\infty,
        $$
        for any $t\in[0, T]$ (here, we tacitly assume that $\|u(t)\|_{L^2(D)} = +\infty$ for times for which it is not well-defined). Thus, we can see that $t\mapsto \|u(t)\|_{L^2(D)}$ is l.s.c. almost surely. Since $I$ is dense in $[0, T]$, it follows that $\sup_{t\in[0,T]}\|u(t)\|_{L^2(D)}^2 = \sup_{t\in I}\|u(t)\|^2_{L^2(D)}$, and we have our claim.
        \item \begin{align}\label{conv:probability_in_proof_ito_lemma}
            \lim_{l\to +\infty}\sup_{t\in [0, T]}\left|\int_0^t\langle u (s) - \tilde{u}^{l}(s),\, h (s)\diff W(s)\rangle_{L^2(D)}\right| = 0\quad\text{ in probability}.
        \end{align}
        As the proof of this claim is done without changes as in \cite{Liu2015Stochastic} we skip it.
        \item \label{eq:Ito_lemma_localized_time}\eqref{eq:ito_formula_orlicz} holds for $t\in I$.\\
        Fix $t\in I\setminus\{0\}$. Then, for each sufficiently large $l\in\N$, there exists a unique $0 < i < k_l$ such that $t = t_i^l$. Due to Lemma \ref{lem:approx_by_floor_func} and \eqref{conv:probability_in_proof_ito_lemma} we can see that the first three terms on the right-hand side of \eqref{ito_lemma_proof_some_eq_1} converge in probability to 
        $$
        2\int_0^t\int_D\alpha (s, x)\cdot\nabla_x u (s,x)\diff x\diff s + 2\int_0^t \langle u (s) ,\,h (s)\diff W(s)\rangle_{L^2(D)}
        $$
        as $l\to +\infty$. Hence, by \cite[Lemma 2.4.4]{Liu2015Stochastic} 
        \begin{align*}
            &\int_D u^2(t) \diff x - \int_D u^2_0  \diff x\\
            &=2\int_0^t\int_D\alpha (s, x)\cdot\nabla_x u (s,x) \diff x\diff s + 2\int_0^t \langle u (s) ,\,h (s)\diff W(s)\rangle_{L^2(D)}\\
            &\phantom{=}+\int_0^t\|h (s)\|^2_{HS(L^2(D))}\diff s - \delta_0,
        \end{align*}
        where
        $$
        \delta_0 := P-\lim_{l\to+\infty}\sum_{j=0}^{i-1}\|u(t^l_{j+1}) - u(t^l_{j})   - (\mathbb{M}(t^l_{j+1}) - \mathbb{M}(t^l_j))  \|_{L^2(D)}^2.
        $$
        Let $P_n$ denote the projection onto the first $n$ vectors of the orthonormal basis in $L^2(D)$, which is orthogonal in $W^{s,2}_0(D)$, and $\hat{\mathbb{M}}^l, \tilde{\mathbb{M}}^l$ denote analogous quantities to those defined for $u$ in Lemma \ref{lem:approx_by_floor_func}. Then, using \eqref{eq:weak_form_in_ito_lemma_proof} we get
        \begin{align*}
            \delta_0 &= P-\lim_{l\to +\infty}\Big(\int_0^t\int_D \alpha (s,x)\cdot \nabla_x(\hat{u}^{l} - \tilde{u}^{l} - P_n(\hat{\mathbb{M}}^{l} - \tilde{\mathbb{M}}^{l}))\diff x\diff s\\
            &\phantom{=} + \int_D(u(t_1^l) - u_0 - \mathbb{M}(t_i^l))\,u(0)\diff x\\
            &\phantom{=} - \sum_{j = 0}^{i - 1}\int_D(u(t^l_{j+1}) - u(t^l_{j})  - (\mathbb{M}(t^l_{j+1}) - \mathbb{M}(t^l_j)) )\,(1 - P_n)(\mathbb{M}(t^l_{j+1}) - \mathbb{M}(t^l_j)) \diff x\Big).
        \end{align*}
        Since $u$ is weakly continuous in $L^2(D)$
        and $\mathbb{M}$ is continuous in $L^2(D)$, we may deduce
        $$
        \int_D(u(t_1^l) - u_0 - \mathbb{M}(t_i^l))\,u(0)\diff x \rightarrow 0 \quad \text{ as }l\to+\infty.
        $$
        Moreover, by Lemma \ref{lem:approx_by_floor_func}
        $$
        \int_0^t\int_D \alpha (s,x)\cdot \nabla_x(\hat{u}^{l} - \tilde{u}^{l})\diff x\diff s\rightarrow 0 \quad \text{ as }l\to +\infty,
        $$
        and $P_n \mathbb{M}(s)$ is a continuous process in $W^{s,2}_0(D)$, hence it satisfies the assumptions of Lemma \ref{lem:approx_by_floor_func} and we get
        $$
                \int_0^t\int_D \alpha (s,x)\cdot \nabla_x P_n(\hat{\mathbb{M}}^{l} - \tilde{\mathbb{M}}^{l})\diff x\diff s\rightarrow 0\text{, as }l\to +\infty.
        $$
        In the end, by the Cauchy--Schwarz inequality, \cite[Lemma 2.4.1, 2.4.4]{Liu2015Stochastic}
        \begin{align*}
            &\delta_0\leq P-\lim_{l\to +\infty}\left(\sum_{j=0}^{i-1}\|u(t^l_{j+1}) - u(t^l_{j})   - (\mathbb{M}(t^l_{j+1}) - \mathbb{M}(t^l_j))  \|_{L^2(D)}^2\right)^{\frac{1}{2}}\\
            &\phantom{=}\cdot\left(\sum_{j=0}^{i-1}\|(1 - P_n)(\mathbb{M}(t^l_{j+1}) - \mathbb{M}(t^l_j))  \|_{L^2(D)}^2\right)^{\frac{1}{2}}\\
            &= \delta_0^{\frac{1}{2}}\left(\int_0^t\|(1 - P_n)h (s) \|_{HS(L^2(D))}^2\diff s\right)^{\frac{1}{2}},
        \end{align*}
        which converges to $0$ with $n\to +\infty$. Therefore $\delta_0 = 0$.
        \item \eqref{eq:ito_formula_orlicz} holds for $t\in [0, T]\setminus I$.\\
        Fix a set $\Omega'\in \mathcal{F}$ with full probability such that the limit \eqref{conv:probability_in_proof_ito_lemma} holds pointwise in $\Omega'$ passing to a not relabeled subsequence if necessary, and \eqref{eq:ito_formula_orlicz} holds for all $t\in I$ in $\Omega'$. Now, if $t\notin I$, then for any $l$ there exists a unique interval $(t^l_{j(l)}, t^l_{j(l) + 1}]$, which contains it. Define $t(l) := t^l_{j(l)}$. From the definition of our partitions, we have $t(l)\nearrow t$ as $l\to +\infty$. By \eqref{eq:weak_form_in_ito_lemma_proof} and \ref{eq:Ito_lemma_localized_time} we deduce
        \begin{equation}\label{eq:some_equation_in_ito_lemma_proof_2}
        \begin{split}
            \|u(t(l)) - u(t(m))\|_{L^2(D)}^2 &= \|u(t(l)) \|_{L^2(D)}^2 - \|u(t(m)) \|_{L^2(D)}^2\\
            &\phantom{=}- \int_{t(m)}^{t(l)}\int_D\alpha (s, x)\cdot\nabla_x u (t(m),x)\diff x\diff s\\
            &\phantom{=}+2\int_{t(m)}^{t(l)} \langle u (t(m)) ,\,h (s)\diff W(s)\rangle_{L^2(D)}\\
            &= 2\int_{t(m)}^{t(l)}\int_D\alpha (s, x)\cdot\nabla_x (u (s,x) - u (t(m),x))\diff x\diff s\\
            &\phantom{=}+2\int_{t(m)}^{t(l)}\langle u(s) - u(t(m)) ,\,h (s)\diff W(s)\rangle_{L^2(D)}\\
            &\phantom{=}+\int_{t(m)}^{t(l)}\|h (s)\|_{HS(L^2(D))}^2\diff s\\
            &= 2\int_{t(m)}^{t(l)}\int_D\alpha (s, x)\cdot\nabla_x (u(s,x) - \hat{u}^{m}(s,x))\diff x\diff s\\
            &\phantom{=}+2\int_{t(m)}^{t(l)}\langle (u (s) - \hat{u}^{m}(s)) ,\,h (s)\diff W(s)\rangle_{L^2(D)}\\
            &\phantom{=}+\int_{t(m)}^{t(l)}\|h(s)\|_{HS(L^2(D))}^2\diff s.
        \end{split}
        \end{equation}
        Using \eqref{conv:probability_in_proof_ito_lemma} (holding pointwise in $\Omega'$ up to a not relabeled subsequence) and the continuity of the map $t\mapsto \int_0^t\|h (s)\|_{HS(L^2(D))}^2\diff s$ we deduce
    \begin{align*}
        \lim_{m\to +\infty}\sup_{l > m}\left\{2\left|\int_{t(m)}^{t(l)}\langle (u (s) - \hat{u}^{m}(s)) ,\,h (s)\diff W(s)\rangle_{L^2(D)}\right| + \int_{t(m)}^{t(l)}\|h (s)\|_{HS(L^2(D))}^2\diff s\right\} = 0\text{ in }\Omega'.
    \end{align*}
    Moreover by Lemma \ref{lem:approx_by_floor_func} we have
    \begin{align*}
        \lim_{m\to +\infty}\sup_{l > m}\int_{t(m)}^{t(l)}\int_D\alpha (s, x)\cdot\nabla_x(u (s,x) - \hat{u}^{m}(s,x))\diff x\diff s = 0,
    \end{align*}
    in some $\Omega''\subseteq \Omega'$ with full probability. Combining both of the convergences above with \eqref{eq:some_equation_in_ito_lemma_proof_2} we conclude
    \begin{align*}
        \lim_{m\to +\infty}\sup_{l\geq m} \|u(t(l)) - u(t(m)) \|_{L^2(D)}^2 = 0
    \end{align*}
    in $\Omega''$. Hence, $\{u (t(l)) \}_{l\in\N}$ is a Cauchy sequence in $L^2(D)$ and has a limit in this space. As the limit in, for example, $W^{-s, 2}(D)$ is equal to $u (t) $, therefore in $L^2(D)$ it is the same.
     In the end, letting $l\to +\infty$ in the equality obtained in \ref{eq:Ito_lemma_localized_time} we get our Claim in $\Omega''$.
    \item $u$ is strongly continuous in $L^2(D)$.\\
    Since the right hand side of \eqref{eq:ito_formula_orlicz} is continuous in $t\in [0, T]$ almost surely, so is the left-hand side, i.e. $t\mapsto \|u(t)\|_{L^2(D)}$ is a continuous function. Thus, the weak continuity of $u(t)$ in $L^2(D)$ implies its strong continuity.
    \end{enumerate}
\end{proof}

\section{Existence of solutions for the approximation}\label{proof:existence_theorem_approx}

\begin{proof}
    To prove our Theorem~\ref{thm:existence_for_approx_in_epsilon} we employ the standard Galerkin approximation. Let $\{\omega_j\}_{j\in\N}$ be an orthogonal basis of $W^{s,2}_0(D)$ for $s$ big enough, which is orthonormal in $L^2(D)$ (cf. \cite[Remark 4.14]{malek1996weakandmeasure}), and at the same time let $\{g_j\}_{j\in\N}$ be the orthonormal basis of $U$. Set
    \begin{align*}
        u^n(t) = \sum_{k = 1}^n \alpha^n_k(t) \omega_k(x),\quad W^n(t) = \sum_{k=1}^n\langle W(t),\, g_k\rangle_U\,g_k,
    \end{align*}
    where $\alpha_k^n(0) = \langle u_0, \omega_k\rangle_{L^2(D)}$. Then, by Theorem 3.1.1 \cite{Liu2015Stochastic} there exists an almost surely continuous, $\R^n$-valued, $(\mathcal{F}_t)_{t\geq 0}$ adapted process $(\alpha(t))_{t\in [0, T]}$ such that almost surely and for all times
    \begin{align*}
        \alpha^n(t) = \alpha^n(0) - \int_0^t\left[\int_D A^\varepsilon(s, x, \nabla_x u^n)\cdot\nabla_x\omega_k\diff x\right]_{k = 1}^n\diff s + \int_0^t\left[\langle \omega_r, h\diff W^n(s)\rangle_{L^2(D)}\right]_{k = 1}^n.
    \end{align*}
    Thus, by the finite-dimensional version of the It\^{o}'s lemma \ref{thm:orlicz_ito_formula} we get
    \begin{multline}\label{eq:ito_lemma_for_galerkin}
        \|u^n(t)\|^2_{L^2(D)} = \|u^n(0)\|^2_{L^2(D)} -2\int_0^t\int_D A^\varepsilon(s, x, \nabla_x u^n)\cdot\nabla_x u^n\diff x\diff s\\
        + \int_0^t \|h\|_{HS(U_n; L^2(D))}^2\diff s + 2\int_0^t\langle u^n, h\diff W^n(s)\rangle_{L^2(D)}
    \end{multline}
    for $U_n := \mathrm{span}\{g_1, ..., g_n\}$. Hence, after applying the expected value to both sides, by the coercivity assumption \ref{ass:A_varepsilon_3} and the fact that $u_0$ is in $L^2$, $h$ in $HS(L^2(D))$ we obtain:
    \begin{enumerate}
        \item $\{\nabla_x u^n\}_{n\in \N}$ is bounded in $L_m (\Omega\times Q_T)$,
        \item $\{A^\varepsilon(t, x, \nabla_x u^n)\}_{n\in \N}$ is bounded in $L_{m*} (\Omega\times Q_T)$,
        \item $\{u^n\}_{n\in \N}$ is bounded in $L^\infty(0, T; L^2(\Omega\times D))$, and therefore also in $L^2(\Omega;L^2(0,T;L^2(D)))$, since by \cite[Corollary 1.2.23, p.25]{HNVW16} this space is isometrically isomorphic to $L^2(0,T;L^2(\Omega;L^2(D)))$.
    By the modular Poincar\'{e} inequality ([Theorem 9.3, \cite{chlebicka2021book}])
        \item $\{u^n\}_{n\in \N}$ is bounded in $L_m (\Omega\times Q_T)$
    \end{enumerate}
    Thus by the Banach--Alaoglu Theorem there exists $u\in L_m(\Omega\times Q_T)\cap L^\infty(0, T; L^2(\Omega\times D))$ with $\nabla_x u\in L_m(\Omega\times Q_T)$ and $Y\in L_{m*}(\Omega\times Q_T)$, such that (up to the subsequence)
\begin{align*}
    	u_n &\rightharpoonup u&& \text{in} \ L^2(\Omega;L^2(0,T;L^2(D))),\\
            u_n &\wstar u && \text{ weakly* in }L_m(\Omega\times Q_T),\\
            \nabla_x u_n &\wstar \nabla_x u && \text{ weakly* in }L_m(\Omega\times Q_T),\\
            A^\varepsilon(t, x, \nabla_x u_n) &\wstar Y && \text{ weakly* in }L_{m*}(\Omega\times Q_T).
    \end{align*}
    Thanks to the properties of the stochastic integral, the fact that $h\in S^2_W(0,T;W^{s,2}_0(D))$ and Burkholder-Davis-Gundy inequality
    \begin{align*}
       \int_0^t h\diff W^n_s \rightarrow \int_0^t h\diff W(s) \text{ in }\mathcal{M}^2_T(W^{s,2}_0(D)) \text{ for all $t\in [0,T]$ and in } L^2(\Omega;\mathcal{C}([0,T];W^{s,2}_0(D)))
    \end{align*}
    Therefore, for any $\phi\in C^\infty_c(D)$ and $\psi\in L^\infty(\Omega\times (0, T))$
    \begin{equation*}
        \begin{split}
            &\mathbb{E}\left(\int_0^T\int_D u\,\phi\,\psi\diff x\diff t\right)\\
            &=\lim_{n\to +\infty}\mathbb{E}\left(\int_0^T\int_D u^n\,\phi\,\psi\diff x\diff t\right)\\
            &=\lim_{n\to +\infty}\mathbb{E}\bigg(\int_0^T\int_D u^n(0)\,\phi\,\psi\diff x\diff t\\
            &\phantom{=}- \int_0^T\int_D\int_0^t A^\varepsilon(s, x, \nabla_x u^n)\cdot\nabla_x(\phi\,\psi)\diff s\diff x\diff t\\
            &\phantom{=} + \int_0^T\int_0^t\langle\phi\,\psi,\,h\diff W^n_s\rangle_{L^2(D)}\diff t\bigg)\\
            &= \mathbb{E}\bigg(\int_0^T\int_D u_0\,\phi\,\psi\diff x\diff t\\
            &\phantom{=}- \int_0^T\int_D\int_0^t Y(s, x)\cdot\nabla_x(\phi\,\psi)\diff s\diff x\diff t\\
            &\phantom{=} + \int_0^T\int_0^t\langle\phi\,\psi,\,h\diff W(s)\rangle_{L^2(D)}\diff t\bigg),
        \end{split}
    \end{equation*}
    which readily implies a.e. in $\Omega\times (0,T)$
    \begin{align}\label{eq:almost_good_equation_for_approximation}
        u(t) = u_0 - \int_0^t\DIV_x Y(s, x)\diff s + \int_0^t h\diff W(s) \quad \text{ in }\mathcal{D}'(D)
    \end{align}
   Since $u\in L^2(\Omega\times Q_T)\cap L^1(\Omega\times (0,T); W^{1,1}_0(D))$ with $\nabla_x u\in L_m(\Omega\times Q_T)$, in particular, by Theorem \ref{thm:orlicz_ito_formula} $u$ is continuous in $L^2(D)$,
  \begin{align}\label{eq:ito_lemma_for_limit}
  		\begin{aligned}
   	&\|u(t)\|^2_{L^2(D)} = \|u_0\|^2_{L^2(D)} -2\int_0^t\int_D Y(s,x)\cdot\nabla_x u\diff x\diff s\\
   	&+ \int_0^t \|h\|_{HS(U_n; L^2(D))}^2\diff s + 2\int_0^t\langle u, h\diff W(s)\rangle_{L^2(D)},
\end{aligned}   
\end{align}
and
    \begin{align*}
        \mathbb{E}\left(\sup_{t\in [0, T]}\|u(t)\|_{L^2(D)}^2\right) < +\infty.
    \end{align*}
    It remains to identify the weak* limit $Y$ with $A^\varepsilon(\cdot, \nabla_x u)$. To this end we employ the monotonicity trick \ref{res:monot_trick}. Applying the expected value to \eqref{eq:ito_lemma_for_galerkin} gives us
    \begin{align*}
        2\mathbb{E}\int_0^t\int_D A^\varepsilon(s, x, \nabla_x u^n)\cdot\nabla_x u^n\diff x\diff s = -\mathbb{E}\|u^n(t)\|^2_{L^2(D)} + \mathbb{E}\|u^n(0)\|^2_{L^2(D)} + \mathbb{E}\int_0^t\|h\|^2_{HS(U_n; L^2(D))}\diff s.
    \end{align*}
    Thus, by the weak lower semicontinuity of the $L^2$-norm, strong convergence of $u^n(0)$ to $u_0$ and the continuity of the projection, after applying $\limsup$ to both sides of the above equation we get
    \begin{align}\label{ineq:some_ineq_galerkin}
    	\begin{aligned}
        &2\limsup_{n\to +\infty}\mathbb{E}\int_0^t\int_D A^\varepsilon(s, x, \nabla_x u^n)\cdot\nabla_x u^n\diff x\diff s\\
         &\leq -\mathbb{E}\|u(t)\|^2_{L^2(D)} + \mathbb{E}\|u_0\|^2_{L^2(D)} + \mathbb{E}\int_0^t \|h\|_{HS(L^2(D))}^2\diff s.
        \end{aligned}
    \end{align}
    Taking expectation in \eqref{eq:ito_lemma_for_limit} and subtracting it from \eqref{ineq:some_ineq_galerkin} it follows that
    \begin{align}\label{ineq:combined}
    	\limsup_{n\to +\infty}\mathbb{E}\int_0^t\int_D (A^\varepsilon(s, x, \nabla_x u^n)-Y(s,x))\cdot\nabla_x u^n\diff x\diff s\leq 0
    	\end{align}
    With this in mind we can proceed with a standard argument. By the monotonicity assumption \ref{ass:A_varepsilon_2} for arbitrary $\eta\in L^\infty(\Omega\times Q_T)^d$
    \begin{align}\label{ineq:for_galerkin_monotonicity_2}
       \limsup_{n\rightarrow\infty} \mathbb{E}\int_0^T\int_D (A^\varepsilon(t, x, \nabla_x u^n) - A^\varepsilon(t, x, \eta))\cdot(\nabla_x u^n - \eta)\diff x\diff t \geq 0.
    \end{align}
    Now, we may write \eqref{ineq:for_galerkin_monotonicity_2} on the right-hand side of \eqref{ineq:combined}. Thanks to the weak* convergence of $A^\varepsilon(t, x, \nabla_x u^n)$ to $Y$, the weak* convergence of $\nabla_x u^n$ to $\nabla_x u$, and, since, by Proposition \ref{prop:stress_bound_for_bounded}, for essentially bounded $\eta$ the function $A(t, x, \eta)$ is actually in $E_{m*}$, we may partially pass to the limit in \eqref{ineq:for_galerkin_monotonicity_2} to obtain 
    \begin{align*}
        \mathbb{E}\int_0^T\int_D (Y - A^\varepsilon(t, x, \eta))\cdot(\nabla_x u - \eta)\diff x\diff t \geq 0,
    \end{align*}
    which by \ref{res:monot_trick} grants us the thesis $Y(t, x) = A^\varepsilon(t, x, \nabla_x u)$ a.e. in $\Omega\times Q_T$.
\end{proof}

\section{Auxiliary propositions and lemmas}

\subsection{In Musielak--Orlicz spaces.}

Below, one can find a lemma concerning the modular approximation in Musielak--Orlicz spaces. We note here, that \cite{bulicek2021parabolic} contains proof of the first statement, but the second one can be easily tracked and proven the same way. The main ingredients are Assumption \ref{ass:N_func_M}, and the presence of the truncation operator.

\begin{prop}\textup{(\cite[Theorem 2.4]{bulicek2021parabolic})}\label{prop:double_mollifier_conv}
Let $D\in\R^d$,  $\psi: D\rightarrow \R$ be compactly supported, satisfying $0\leq\psi\leq 1$. Moreover let $u\in L^1(0, T; W^{1,1}_0(D))\cap L^\infty(0, T; L^2(D))$ with $\nabla_x u\in L_M(Q_T)$, where M is an N-function satisfying Assumption \ref{ass:N_func_M}. Then, there exists $\kappa_0$ such that for any $\kappa\in (0, \kappa_0)$
\begin{align*}
\nabla_x(T_k(u^\kappa)\psi)^\kappa &\xrightarrow{M} \nabla_x(T_k(u)\psi)&&\text{ modularly in }L_M(Q_T),\\
\nabla_x(T_k(u^\kappa))^\kappa\,\psi &\xrightarrow{M} \nabla_xT_k(u)\,\psi &&\text{ modularly in }L_M(Q_T).
\end{align*}
Here $f^\kappa$ is a standard mollification in a spatial variable of a function $f$.  
\end{prop}

\begin{prop}\textup{(\cite[Lemma 4.1]{bulicek2021parabolic})}\label{prop:conv_of_localization}
Let $D\subset \R^d$ be a Lipschitz domain, and M be an $N$-function satisfying Assumption \ref{ass:N_func_M}. Then there exists a family of functions $\{\psi_j\}_{j\in\N}\subset C_c^\infty(D)$ fulfilling $0\leq\psi_j\leq 1$, and $\psi_j\to 1$ as $j\to +\infty$, such that for any $u\in L^1(0, T; W^{1,1}_0(D))\cap L^\infty(0, T; L^2(D))$ with $\nabla_x u\in L_M(Q_T)$, we have
$$
\int_0^t\int_D M\left(s, x, \left|\frac{\nabla_x \psi(x) u(s, x)}{C_u}\right|\right)\diff x\diff s\to 0 \quad\text{ as }j\to +\infty,
$$
where the constant $C_u$ can be chosen as $C_u = C\|\nabla_x u\|_{L_M(Q_T)}$, where $C$ depends only on $D$.
\end{prop}

\begin{lem}\label{res:monot_trick}\textup{(\cite[Lemma 2.16]{bulicek2021parabolic})}
Let $A$ satisfy Assumption \ref{ass:stress_A}, $M$ be an $N$-function. Assume there are $\chi \in L_{M^*}(Q_T)$ and $\xi \in L_M(Q_T)$, such that
\begin{equation*}
\int_{Q_T} \left(\chi - A(t,x,\eta)\right) \cdot (\xi - \eta)\diff t \diff x \geq 0
\end{equation*}
for all $\eta \in L^{\infty}(Q_T)$. Then
$$
A(t,x,\xi) = \chi(t,x) \mbox{ a.e. in } Q_T.
$$
\end{lem}

\subsection{For a stochastic integral.}


\begin{prop}\textup{(Standard It\^{o}'s formula, \cite[Theorem 4.17]{daprato2014stochastic})}\label{prop:standard_ito}
    Let $\Phi$ be an $L_2(U; H)$ valued process, stochastically integrable in $[0, T]$, $\phi$ is an $H$ valued, predictable proces, Bochner integrable on $[0, T]$ almost surely, $X(0)$ is an $\mathcal{F}_0$ measurable, $H$ valued random variable, and $W$ is a $U$ valued $Q-$Wiener process. Then, the following process
    $$
        X(t) = X(0) + \int_0^t\phi(s)\diff s + \int_0^t \Phi(s)\diff W(s), \qquad t\in [0, T],
    $$
    is well-defined. Assume that the function $F:[0, T]\times H \rightarrow \R$ and its partial derivatives $F_t$, $F_x$, $F_{xx}$ are uniformly continuous on bounded subsets of $[0, T]\times H$. Then, almost surely and for all $t\in [0, T]$
    \begin{equation*}
        \begin{split}
            F(t, X(t)) =& F(0, X(0)) + \int_0^t\langle F_x(s, X(s)), \,\Phi(s)\diff W(s)\rangle + \int_0^t F_t(s, X(s)) + \langle F_x(s, X(s)),\,\phi(s)\rangle\diff s\\
            & + \int_0^t\frac{1}{2}\mathrm{tr}(F_{xx}(s, X(s))(\Phi(s)Q^{\frac{1}{2}})(\Phi(s)Q^{\frac{1}{2}})^*)\diff s.
        \end{split}
    \end{equation*}
\end{prop}


\bibliographystyle{abbrv}
\bibliography{parpde_mo_discmodul}
\end{document}